\documentclass[leqno, 12pt]{amsart}
\usepackage[utf8]{inputenc}
\usepackage{amsmath}
\usepackage[foot]{amsaddr}
\usepackage{amsfonts}
\usepackage{amssymb}
\usepackage{amsthm}
\usepackage{float}
\usepackage{mathtools}
\usepackage{csquotes}
\usepackage[numbers]{natbib}
\usepackage{bbold}
\usepackage{pdfpages}

\RequirePackage{doi}

\usepackage[margin=3cm]{geometry}

\usepackage{tikz-cd}
\usetikzlibrary{arrows,matrix,shapes,decorations.text, mindmap,shapes.misc}
\usepackage{metalogo}
\newcommand{\C}{\mathbb{C}}
\newcommand{\Z}{\mathbb{Z}}

\newcommand{\T}{\mathbb{T}}
\renewcommand{\S}{\mathbb{S}}
\renewcommand{\P}{\mathbb{P}}
\newcommand{\R}{\mathbb{R}}
\newcommand{\norm}[1]{\left\lVert#1\right\rVert}
\newcommand{\ip}[2]{\left\langle #1,#2\right\rangle}

\newcommand{\arsinh}[0]{\text{arsinh}}
\newcommand{\tr}[0]{\text{tr}}
\newcommand{\Tr}[0]{\text{Tr}}
\newcommand{\Vol}[0]{\text{Vol}}
\newcommand{\diam}{\text{diam}}
\newcommand{\n}{\nabla}
\newcommand{\on}{\overline{\nu}}
\newcommand{\ob}{\overline{\beta}}
\newcommand{\al}{\alpha}

\newcommand{\oz}{\overline{z}}
\newcommand{\ow}{\overline{w}}
\newcommand{\adj}{\text{adj}}
\newcommand{\osc}{\text{osc}}

\makeatletter
\@namedef{subjclassname@2020}{%
  \textup{2020} Mathematics Subject Classification}
\makeatother

\hypersetup{colorlinks=true, linkcolor=violet, citecolor=blue, urlcolor=black}

\newtheorem{Thm}{Theorem}[section]
\newtheorem{Conj}{Conjecture}
\newtheorem{Quest}{Question}
\newtheorem{Lem}[Thm]{Lemma}
\newtheorem{Cor}[Thm]{Corollary}
\newtheorem{Prop}[Thm]{Proposition}
\usepackage{hyperref}
\theoremstyle{definition}
\newtheorem{Def}[Thm]{Definition}
\newtheorem{Not}[Thm]{Notation}
\newtheorem{Rem}[Thm]{Remark}

\newtheorem{assum}[Thm]{Assumption}
\usepackage[toc,page]{appendix}

\numberwithin{equation}{section}

\title[Geodesics on a K3]{Geodesics on a K3 Surface near the orbifold limit}
\author{J\o rgen Olsen Lye}
\address{Institut f\"{u}r Differentialgeometrie, Leibniz Universit\"at Hannover}
\email{joergen.lye@math.uni-hannover.de}
\begin{document}
\pagestyle{plain}

\phantomsection

\begin{abstract}
This article studies Kummer K3 surfaces close to the orbifold limit. We improve upon  estimates for the Calabi-Yau metrics  due to R. Kobayashi. As an application, we study stable closed geodesics. We use the metric estimates to show how there are generally restrictions on the existence of such geodesics. We also show how there can exist stable, closed geodesics in some highly symmetric circumstances due to hyperk\"{a}hler identities.
\end{abstract}

\maketitle

\tableofcontents

\section{Introduction}
Einstein metrics are interesting objects both from a physics and a geometry perspective. By postulating enough symmetry, non-compact, examples have been found in both Lorenzian and Riemannian geometry. For instance the solutions due to Schwarzschild \cite{SSMetric}, Eguchi-Hanson \cite{EH79}, Calabi \cite{Calabi}, Gibbons-Hawkings \cite{GH}, and Kronheimer \cite{Kronheimer1}, \cite{Kronheimer2}, to name just a few. These are so-called gravitational instantons. On compact manifolds, very few examples are known explicitly. An idea dating back to Page \cite{Page}  and Gibbons-Pope \cite{GibPope} is to desingularize certain orbifolds by blowing up the singular points and gluing in gravitational instantons. This procedure produces a family of almost-Einstein metrics with concentrated curvature. By an implicit function argument,\cite{JoyceSpin},  \cite{JoyceG2}, \cite{Foscolo},  by Twistor methods, \cite{SingdeBr}, \cite{Top}, or by solving the Monge-Amp\`{e}re equation, \cite{DonaldsonKummer}, \cite{KT}, \cite{Kob90}, one can perturb the given metric to an Einstein metric. Since the original metric was close to solving the Einstein equation, one could hope that the metric is close to the solution of the Einstein equation. This turns out to be the case in several instances. This in turn allows one to study the geometry of the unkown Einstein metric by studying its known approximation. 

As the size of the exceptional divisor in the blow-up tends to 0, both the patchwork metric and the Einstein metric degenerate to an orbifold metric. In fact, work by Bando, Kasue, and Nakajima, \cite{Bandobubble},  \cite{NakCollapse}, \cite{Bakana}, and Anderson \cite{AndersonCollapse} show that a converse is true; a sequence of compact 4-manifolds and Einstein metrics with volume, diameter, and Euler characteristic bounds have convergent subsequences. The limit spaces have at worst isolated  orbifold singularities. The orbifold limits are the added points of a compactification of the moduli space of Einstein metrics \cite{AndersonCollapse}, \cite{NakCollapse}, \cite{Bakana}. This makes understanding the orbifold limit more important.

In this paper, we will study the oldest and best-known example, namely the Kummer construction of a K3 surface. In this case, one can go beyond the general convergence statements of Nakajima and Anderson, and describe in more detail exactly how the Einstein metric degenerates. This program was initially carried out by Todorov and R. Kobayashi, \cite{KT}, \cite{Kob90}. We have chosen to go through the arguments of Kobayashi in great detail in the hope that the present paper can serve as an introduction to this fascinating topic.

As an application, we study stable, closed geodesics on Ricci-flat Kummer K3 surfaces close to the orbifold limit. The lengths of closed geodesics on a Riemannian manifold is a much studied geometric quantity (see \cite{Huber}, \cite{DuistGuil}, \cite{deVerdiere1}, \cite{deVerdiere2}, and \cite{Colin} for some highlights).  For generic metrics it has been shown \cite[2. Corollary]{Rademacher} that there are infinitely many geometrically distinct, closed geodesics. The same holds for an arbitrary metric if one imposes mild topological assumptions (see \cite[Theorem 4]{GM69} and \cite[Theorem (2nd)]{VPS76}). The topological conditions are fulfilled by all compact Calabi-Yau manifolds. We propose to restrict attention to the lengths of \textit{stable}, closed geodesics.
\begin{Def}
Let $(M,g)$ be a Riemannian manifold. A closed geodesic $\gamma\colon \S^1\to M$ is said to be \textbf{stable} if $\delta^2E_{\gamma}\geq 0$. Written out, this means
\begin{equation}
\delta^2E_{\gamma}(\xi,\xi)=\int_{\S^1} \vert \nabla_{\dot{\gamma}} \xi\vert^2 -\ip{R(\dot{\gamma},\xi)\xi,\dot{\gamma}}\, dt\geq 0
\label{eq:Energy}
\end{equation}
for all vector fields $\xi$ along $\gamma$, where $R(U,V)\coloneqq \nabla_U \nabla_V -\nabla_V\nabla_U-\nabla_{[U,V]}$ is the Riemann tensor.
\end{Def}
 Consider the number $\mathcal{N}(L)$ of stable, closed geodesics of length at most $L$, where one counts families of geodesics  as a single geodesic. If the manifold $(M,g)$ is compact and real analytic then it is a consequence of \cite[1.2. Proposition]{SS11} that $\mathcal{N}(L)$ is finite for any $L\geq 0$. For stable geodesics one has the following trichotomy based on curvature. 
\begin{Thm}[{\cite{Myers},\cite{Mar69}}]
\label{Thm:Comparison}
Let $(M,g)$ be a compact, connected, real-analytic Riemannian manifold of real dimension $n$. Let $\mathcal{N}(L)$ be as above. Then we have the following asymptotic behaviour as $L\to \infty$.
\begin{itemize}
\item If $Ric\geq \kappa (n-1)$ for some $\kappa>0$, then $\mathcal{N}(L)$ is constant for $L>\frac{\pi}{\sqrt{\kappa}}$. 
\item If the sectional curvature vanishes, then $\mathcal{N}(L)\sim c(M) L^n$ for some constant $c(M)>0$ depending on the manifold.
\item If the sectional curvature is negative everywhere, then $\mathcal{N}(L)\sim \frac{e^{c(M)L}}{c(M)L}$ for some constant $c(M)>0$ depending on the manifold.  
\end{itemize}
\end{Thm}
The hardest part of the above statement is the negative curvature case, which is \cite{Mar69}, \cite[Equation 6.87]{Mar00}. The positive case follows by the proof of the Bonnet-Myers theorem, and the flat case is a direct computation on a flat torus.

Theorem \ref{Thm:Comparison} is an example of a comparison geometry, and a natural question is whether or not one can replace sectional curvature by Ricci-curvature. This motivates the following question.

\begin{Quest}
Let $(M,g)$ be a compact, connected, Ricci-flat manifold of dimension $n$. Is it true that $\mathcal{N}(L)\sim c(M)L^n$ for some constant $c(M)>0$?
\end{Quest}

\begin{Rem}
We do not know what happens for $Ric<0$. Our guess is that this condition is too weak, seeing how (in light of \cite[Theorem A]{Loh94}) the condition $Ric<0$ gives absolutely no information about the underlying manifold in dimensions $n\geq 3$.
\end{Rem}

If one additionally assumes that the manifold is K\"{a}hler, then P. Gao and  M. Douglas have  put forward physics-based arguments in \cite{GD13} for why the answer to the above question should be yes.
\begin{Conj}[\cite{GD13}]
\label{Conj}
Any compact, Ricci-flat Calabi-Yau manifold\footnote{Our definition of a Calabi-Yau manifold is a K\"{a}hler manifold $(X,g)$ with vanishing first Chern class $c_1(X)=0$ and trivial first cohomology group, $H^1(X;\R)=0$.} $(X,\tilde{g})$ has stable, closed, non-constant geodesics. In fact, if the manifold is of real dimension $n$ then $\mathcal{N}(L)\sim C(X)L^n$ for some constant $C(X)>0$. 
\end{Conj}
When Douglas and Gao published their work, there were no examples of a single stable, closed geodesic on a Ricci-flat, compact Calabi-Yau manifold. They however suggest as a starting point to investigate the conjecture on a Kummer K3 surface (the construction will be recalled in Section \ref{Section:Kummer}). In this article, we follow their advice and derive some constraints on stable, closed geodesics on Kummer K3 surfaces. Additionally, we show that the Riemann curvature tensor vanishes at certain points if the K3 surface has enough symmetry.  Roughly speaking, a Kummer K3  surface is the minimal resolution of the orbifold $\T^4/\{\pm 1\}$ equipped with a K\"{a}hler metric $g$ which is Eguchi-Hanson near any blown-up point, flat far away from the exceptional divisor, and a gluing of these two in between. We shall refer to this metric $g$ as the patchwork metric. By the Calabi-Yau theorem, \ref{Thm:CY}, there exists a unique Ricci-flat metric $\tilde{g}$ in the K\"{a}hler class of $g$. What we show is then the following.

\begin{Thm}
\label{Thm:IntroNoGo1}
Let $X$ be a Kummer K3 surface with metrics $g$ and $\tilde{g}$ as above. Then, when the exceptional divisor $E\subset X$ has small enough volume, there is an open set $U\subset X$ with $E\subset U\subset X$ such that no stable, closed geodesic (with respect to either $g$ or $\tilde{g}$) on X ever enters $U$.
\end{Thm} 

\begin{Thm}
\label{Thm:IntroNoGo2}
Assume the setup of Theorem \ref{Thm:IntroNoGo1}. Let $U_i\subset X$ be a suitable neighbourhood  around a single component $E_i$ of the exceptional divisor. Then, when the volume of $E$ is small enough there are no stable, closed geodesics which stay completely inside $U_i$.
\end{Thm}
\noindent See Theorem \ref{Thm:NoGo1} and Theorem \ref{Cor:NoGo} for the detailed statements.

\begin{Thm}
\label{Thm:IntroStable}
Let $X$ be the Kummer K3 surface associated to the torus $T=\C^2/\Gamma$ where $\Gamma\coloneqq (\Z\{1,i\})^2\subset \C^2$. Let $g$ be the patchwork metric and let $\tilde{g}$ denote the unique Ricci-flat metric in the K\"{a}hler class of $g$. Then there are totally geodesic tori $\T^2\subset X$ and points $p\in \T^2$ where the Riemann tensor of $(X,\tilde{g})$ vanishes. Furthermore, if the minima of the curvature of $\T^2$ are local minima of the holomorphic sectional curvature of $X$, the tori are flat.
\end{Thm} 
See Theorem \ref{Thm:StableGeod} for a more detailed statement and a precise description of how some of these tori look like. Theorem \ref{Thm:FlatTori} is the precise statement of the second half of the theorem.

To our knowledge, the only previous work on stable geodesics on compact, Ricci-flat Calabi-Yau manifolds are the articles \cite{GD13}, \cite{BY73}, and \cite{B76}.  What Bourguignon and Yau prove in \cite{BY73} is the following, a result we will need later.
\begin{Thm}[\cite{BY73}]
\label{Thm:BY}
 Assume $(X,\tilde{g})$ is a hyperk\"{a}hler manifold of real dimension $4$. Assume $\gamma\colon \S^1\to X$ is a non-constant geodesic. Then $\gamma$ is stable if and only if the entire Riemann curvature tensor vanishes along $\gamma$. 
\end{Thm}
Theorem \ref{Thm:BY} is both a clear cut criterion as well as a major hurdle for stability. A priori, it is not clear that a Ricci-flat space has a single point with vanishing Riemann tensor. Indeed, the Eguchi-Hanson space of \cite{EH79} and \cite{Calabi} has a Ricci-flat metric with   nowhere vanishing Riemann tensor.  The Eguchi-Hanson K\"{a}hler potential is given by \eqref{eq:EHPot} below, and the norm of the curvature tensor in \eqref{eq:EHCurvature}. This is a non-compact manifold, so it does not contradict the above conjecture. In the presence of symmetries, Theorem \ref{Thm:IntroStable} tells us that the Riemann tensor vanishes at certain points on a Kummer K3 surface. We do not know of other results of this kind. In particular, we do not know what happens on an arbitrary K3 surface.

The layout of the paper is as follows. We recall the Kummer construction, including the patchwork metric and metric estimates, in Section \ref{Section:Kummer}. Section \ref{Section:NoGo} deals with the no-go results Theorem \ref{Thm:IntroNoGo1} and \ref{Thm:IntroNoGo2}.  Section \ref{Section:Stable} starts by studying what one can say about the curvature of hyperk\"{a}hler 4-manifolds in the presence of symmetry, before specializing to a particular Kummer K3 surface in Section \ref{Section:SpecialK3}. To improve the flow, we have relegated some of the computations of the derivatives of the curvature to Appendix \ref{App:Riemann}. In section \ref{Section:Estimates}, we reprove the metric estimates of R. Kobayashi and correct some of the methods. Appendix \ref{App:Yau} gives a self-contained proof of some identities from \cite{Yau77} which are used in Section \ref{Section:Estimates}. 

We end the introduction by listing some general obstructions to studying the conjecture of \cite{GD13}.

\begin{itemize}
\item There is no explicitly known, non-flat, Ricci-flat metric on a compact Riemannian manifold. The stability of geodesics is however very dependent on the details of the metric (e.g. Theorem \ref{Thm:BY}).
\item In the case of a K3 surface, Theorem \ref{Thm:BY} has as corollary that any stable, closed geodesic has nullity equal to 3, and the linearised Poincar\'{e} map has all eigenvalues equal to 1. This says that stable, closed geodesics on K3 surfaces are very degenerate critical points of the energy function, making them harder to study using Morse-Bott-type methods.
\item The Ricci-flatness condition on a Calabi-Yau manifold can be though of as specifying the volume for to be "Euclidean" (equation \eqref{eq:MA} is the precise meaning of this). Deriving statements about the length spectrum can as such be seen as asking for length-information when given information about the volume.
\item A compact Calabi-Yau manifold with Ricci-flat metric always has finite isometry group (a fact due to S. Bochner -  see \cite{RiemGeomPetersen}[Theorem 1.5, p. 167] for instance). This makes it challenging to construct geodesics as fixed point sets of isometries.
\item The fundamental group of a compact Calabi-Yau manifold is always finite. So, unlike in the case of a flat torus, one cannot realize enough stable, closed geodesics as non-trivial homotopy classes to fulfil the conjecture. Indeed, on a non-simply connected, compact manifold there are always closed geodesics which minimize the energy in their homotopy class, and are as such \textbf{strictly stable}, meaning all variations of the energy function are non-negative, and not just the second variation.
\item In \cite[2.9 Th\'{e}or\`{e}me]{B76}, J. P. Bourguignon proves that if one has a hyperk\"{a}hler 4-manifold with a strictly stable, closed, non-constant geodesic then the manifold is flat. In particular, K3 surfaces never have \textit{strictly} stable, closed geodesics.
\end{itemize}

Some words about the notation. Local expressions for K\"{a}hler metrics $g$ on manifolds $X$ of complex dimension $n$ will be treated as hermitian $n\times n$ matrices. Determinants and traces of Hermitian matrices are with respect to the complex matrices, not the associated real matrices. The \textbf{complex Hessian} of $f$ will be denotes by $\nabla^2 f$. So
\[(\nabla^2 f)_{\mu\on}\coloneqq \partial_{\mu}\partial_{\on} f\coloneqq \frac{\partial^2 f}{\partial z^{\mu} \partial \overline{z}^{\nu}}.\]
 Tensor norms of complex tensors and the \textbf{Laplacian} will also be defined using the hermitian matrices as follows.
\[\Delta f \coloneqq g^{\on \mu} \partial_{\mu} \partial_{\on} f=\Tr(g^{-1} \nabla^2 f)\eqqcolon \Tr_g(\nabla^2 f),\]
\[\vert \nabla f\vert^2_g\coloneqq g^{\on \mu} \partial_{\mu} f \overline{\partial_{\nu} f},\]
\[\vert \nabla^2 f\vert^2_g \coloneqq g^{\on \mu} g^{\ob \al} (\partial_{\mu}\partial_{\ob} f)(\partial_{\al}\partial_{\on} f)=\Tr(g^{-1} (\nabla^2 f) g^{-1} (\nabla^2 f)).\]
Similar definitions hold for higher rank tensors. The complex Laplacian acting on functions coincides (up to a constant scaling) with the real Laplacian. The complex Hessian does \textit{not} coincide with the real Hessian. Indeed, in complex dimension 1 we have $\nabla^2 f= g \Delta f$. 

For the \textbf{real Hessian} we write $D^2 f$. Its pointwise norm is 
\[\vert D^2 f\vert^2_g\coloneqq \Tr\left(g_{\R}^{-1} D^2 f g_{\R}^{-1} D^2 f\right),\]
 where $g_{\R}$ is the symmetric $2n\times 2n$-matrix associated the $g$. The H\"{o}lder semi-norm of $D^2f$ is defined as\footnote{See \cite[p. 44]{YamabeProblem}.} 
\[\vert D^2 f\vert_{\alpha} =\sup_{x,y} \frac{\vert D^2f(x)-D^2f(y)\vert_g}{d(x,y)^{\alpha}} ,\]
where the supremum is over all $x\in X$ and all $y\neq x$ contained in normal coordinate charts centred at $x$, and the tensor $D^2f(y)$ means the tensor at $x$ obtained by parallel transport along the radial geodesic between $x$ and $y$.

\section{The Kummer Construction}
\label{Section:Kummer}
The Kummer construction is a well-known construction which associates to any 4-torus $T\cong \C^2/\Gamma$ a K3 surface $X$. The idea of the patchwork metric which we put on $X$ comes from \cite{Page} and \cite{GibPope}. To our knowledge, R. Kobayashi was the first to write out the details of this metric in \cite{Kob90}. This is our main source on the Kummer construction. There is a twistorial discussion in \cite{SingdeBr}, but they do not have any explicit metric estimates.  See also \cite[Chapter 5]{PhD} for more details. 

Let us first give an algebraic-geometric description of  the Kummer construction. This is a standard result and can be found in \cite[p. 224]{Surfaces} for instance. 
Let  $\Gamma\subset \C^2$ be a non-degenerate lattice. Let $T\coloneqq \C^2/\Gamma$ be the associated $4$-torus. Let $\mu_2\coloneqq \{\pm 1\}$ act diagonally on $\C^2$. Then this induces an action on $T$ with precisely $16$ fixed points. The quotient $Y\coloneqq T/\mu_2$ is a complex space with singular set $Sing(Y)=\Gamma/2\Gamma$. The singular points are $A_1$-singularities. Blowing up each of these singular points once leads to a non-singular space $X$ along with a blow-down map $\pi\colon X\to Y$. This is the minimal resolution of $X$, and is called the \textbf{Kummer K3 surface associated to the torus $T$}.

Near any of the fixed points, the singular space $Y$ looks like $\C^2/\mu_2$, and the resolution of this can be identified with $\mathcal{O}_{\C\P^1}(-2)=T^*\C\P^1$. The resolution $\mathcal{O}_{\C\P^1}(-2)$ can be equipped with the Eguchi-Hanson metric, which was discovered in \cite{EH79} and generalized in \cite{Calabi}. See \cite{LyeEH} for a detailed discussion about these metrics and the resolutions.  The Eguchi-Hanson K\"{a}hler potential is given in \eqref{eq:EHPot}. 

To describe the metric we put on $X$, we need to have a look at what happens near a blown-up point. This is done for us in \cite[pp. 293-297]{Kob90}.
Let $z$ be the coordinates on $\C^2$, and define $u\coloneqq \vert z\vert^2_{\C^2}\coloneqq \vert z_1\vert^2 +\vert z_2\vert^2$. Choose some $a>0$ and let $f_{Euc},f_a\colon (\C^2\setminus \{0\})/\mu_2\to \R$ be the  \textbf{Euclidean K\"{a}hler potential} and \textbf{Eguchi-Hanson K\"{a}hler potential} respectively, namely
\[f_{Euc}(z)\coloneqq u\]
\begin{equation}
f_a(z)\coloneqq \sqrt{a^2+u^2}-a\cdot \arsinh\left(\frac{a}{u}\right).
\label{eq:EHPot}
\end{equation}
Let $0<\delta\ll 1$ be some fixed number and let $\chi\colon [0,\infty)\to \R$ be a smooth cutoff function with the following properties.
\begin{itemize}
\item $\chi(u)=1$ for $u\leq 1$
\item $\chi(u)=0$ for $u\geq 1+\delta$. 
\end{itemize}
Then 
\begin{equation}
\Phi_a(z)\coloneqq f_{Euc}(z)+\chi(u(z))(f_a(z)-f_{Euc}(z))
\label{eq:GluedPot}
\end{equation}
defines a spherically symmetric K\"{a}hler potential on $(\C^2\setminus \{0\})/\mu_2$ for all values of $a$ small enough. We shall write $\Phi_{a}(z)\eqqcolon \varphi_a(u(z))$. Furthermore, on the complement of any orbiball $V\coloneqq (\C^2\setminus B_R(0))/\mu_2$ we may write 
\begin{equation}
\Phi_a(z)=u+a^2 \xi_a(z)
\label{eq:NeckPotential}
\end{equation}
for some function $\xi_a \colon V\to \R$ which is regular as $a\to 0$.
For later use, we also record the norm squared of the Riemann tensor of the Eguchi-Hanson metric,
\begin{equation}
\vert Riem_{g_{EH}}\vert_{g_{EH}}^2=\frac{24 a^4}{(a^2+u^2)^3}.
\label{eq:EHCurvature} 
\end{equation}
This is an $L^2$-function with most of its mass concentrated near $u=0$, hence the patchwork metric is a metric of concentrated curvature.

We want to define a K\"{a}hler metric on all of $X$ whose K\"{a}hler potential is given by \eqref{eq:GluedPot} close to the exceptional divisor and flat far away from it. Let $\pi\colon X\to Y$ be the blow-down map as above. Let $Sing(Y)=\cup_{i=1}^{16}\{p_i\}$ denote the singular points of $Y$ and let $E\coloneqq \cup_{i=1}^{16} E_i\coloneqq \cup_{i=1}^{16} \pi^{-1}(\{p_i\})$ be the \textbf{exceptional divisor} of $X$. Fix a number $0<\delta\ll 1$, and choose numbers $a_i$, $1\leq i\leq 16$, such that $0<a_i\ll 1$. Then there exists a K\"{a}hler metric
$g$ on $X$ with the following properties. Each component $E_i\subset X$ has a neighbourhood $U_i\subset X$ such that $U_i\coloneqq Bl_0(B_{1+2\delta}(0)/\mu_2)$ and $E_i=\C\P^1$. By scaling $X$ if necessary we may assume $U_i\cap U_j =\emptyset$ for $i\neq j$. On $Bl_0(B_{1+2\delta}(0)/\mu_2)$ the K\"{a}hler potential of $g$ is given by \eqref{eq:GluedPot} with parameter $a_i$. In particular $g$ is equal to the Eguchi-Hanson metric with potential \eqref{eq:EHPot} on $Bl_0(B_{1}(0)/\mu_2)$ and $g_{\vert E_i}=a_i g_{FS}$ where $g_{FS}$ is the Fubini-Study metric on $\C\P^1$. On any of the necks $N_i\coloneqq (B_{1+\delta}(0)\setminus B_{1}(0))/\mu_2$ the metric is not Ricci-flat.  Outside of all the sets $U_i$ the metric $g$ is flat. 

The metric $g$ will be called the \textbf{patchwork metric}.
We will follow \cite{Kob90} and write $\vert a\vert^2\coloneqq \sum_{i=1}^{16} a_i^2$. The limit $\vert a\vert \to 0$ is called the \textbf{orbifold limit}, and $\vert a\vert^2$ being small is what we mean by being close to the orbifold limit.

\begin{Rem}
The cohomology-class of the patchwork metric does not depend on the specific choice of $\chi$. If $\xi$ is another cutoff function, then the difference of the resulting K\"{a}hler forms will locally   look like $i\partial \overline{\partial} ((\chi-\xi)(f_{Euc}-f_{a_i}))$, which is the differential of a globally defined function (one simply extends by 0 to the whole manifold). Hence the Ricci-flat metric in the next theorem does not depend on the choice of $\chi$.
\end{Rem}

The patchwork metric is not Ricci-flat due to the neck regions, hence is not the final metric we want to put on $X$. The celebrated Calabi-Yau theorem of \cite{CalabiConj1}, \cite{CalabiConj2} and \cite{Yau78} provides the solution to this problem.
\begin{Thm}[{\cite[Theorem 2]{Yau78}}]
\label{Thm:CY}
Assume $(X,g)$ is a K\"{a}hler manifold of complex dimension  $n$, and assume the first Chern class vanishes, $c_1(X)=0$. Let $\omega$ be the K\"{a}hler form associated to $g$, and let $\psi\colon X\to \R$ be a function such that $Ric_g=\nabla^2 \psi$. Define the constant $A>0$ by
\begin{equation}
A\coloneqq \frac{\int_X \omega^n}{\int_X e^{\psi} \omega^n}.
\label{eq:ADef0}
\end{equation}  
Then there is a unique function $\phi\colon X\to \R$ subject to the normalization $\int_X \phi \, \omega^n=0$ such that $\tilde{\omega}\coloneqq \omega +i\partial \overline{\partial}\phi$ is a K\"{a}hler form satisfying the \textbf{Monge-Amp\`{e}re} equation
\begin{equation}
\tilde{\omega}^n=Ae^{\psi} \omega^n. 
\label{eq:MA}
\end{equation}
The K\"{a}hler metric $\tilde{g}$ associated to $\tilde{\omega}$ is Ricci-flat.
\end{Thm}
When we additionally assume $H^1(X;\R)=0$, then $X$ has a finite fundamental group. This follows from Cheeger-Gromoll splitting theorem, \cite{CG71} and the Calabi-Yau \cite[Theorem 2]{Yau78}. Let $\tilde{X}$ denote the universal cover of $X$. Then $H^1(\tilde{X};\Z)=0$, and $c_1(X)=0\implies c_1(\tilde{X})=0$. So $\tilde{X}$ admits a nowhere zero holomorphic $n$-form $\eta$. We normalize this so that
\[\exp(\psi)=\frac{i^{n^2} \eta \wedge \overline{\eta}}{\omega^n}.\]
In terms of this $n$-form, we may write
\begin{equation}
A= \frac{\int_{\tilde{X}} \omega^n}{\int_{\tilde{X}} e^{\psi} \omega^n}=\frac{\int_{\tilde{X}} \omega^n}{\int_{\tilde{X}} i^{n^2}\eta\wedge\overline{\eta}}.
\label{eq:ADef}
\end{equation}  

The Calabi-Yau theorem applies to any Kummer K3 surface. K3 surface are already simply-connected, so we do not need to pass to a universal cover. The only work one has to do is to decide on what $\eta$ is and then use this to compute the constant $A$ in \eqref{eq:ADef}. We will formulate this as a little lemma and supply a proof since \cite{Kob90} does not compute $A$. Let $\eta$ be the nowhere vanishing holomorphic 2-form on $X$ induced from 
\[\eta=\sqrt{2}dz_1\wedge dz_2\] on $\C^2$. For (measurable) subsets $V\in X$, we write
\[\Vol_{Euc}(V)=\int_V \eta\wedge \overline{\eta}\]
and
\[\Vol_g(V)=\int_{V} d\Vol_g=\int_V\det(g)\eta\wedge \overline{\eta}.\]
\begin{Lem}
Let $(X,g)$ be a Kummer K3 surface with patchwork metric $g$ as described above.   Then the constant $A$ given by \eqref{eq:ADef} takes the value
\begin{equation}
A=1-\vert a\vert^2 \frac{\pi^2}{2\Vol_{Euc}(T)}=1-\vert a\vert^2 \frac{\Vol_{Euc}(B_1^4)}{\Vol_{Euc}(T)}.
\label{eq:AValue}
\end{equation}
\end{Lem}
\begin{proof}

By \eqref{eq:MA} we have 
\begin{equation}
\Vol_g(X)=A \int_X \eta \wedge \overline{\eta}=\frac{A\cdot \Vol_{Euc}(T)}{2},
\label{eq:AComp1}
\end{equation}
where we have implicitly used $\int_X =\int_Y=\frac{1}{2}\int_T$.
Let $N=\cup_{i=1}^{16} N_i$ denote the union of all the neck regions of $X$. We then compute
\begin{align*}
\Vol_g(X)&=\int_{X\setminus N} dVol_g +\int_N dVol_g \\  &=\int_{X\setminus N} \eta \wedge \overline{\eta}+\int_N dVol_g\\  &=\frac{1}{2}\Vol_{Euc}(T) -\Vol_{Euc}(N)+\Vol_g(N).
\end{align*}
Inserting this into \eqref{eq:AComp1} results in
\begin{equation}
A=1+2\frac{\Vol_g(N)-\Vol_{Euc}(N)}{\Vol_{Euc}(T)},
\label{eq:AComp2}
\end{equation}
and it only remains to compute $\Vol_g(N)$.

Working on a single neck $N_i$, the K\"{a}hler potential of $g$ is given by \eqref{eq:GluedPot}. For any spherically symmetric K\"{a}hler potential in 2 complex dimensions $F(z)=\Psi(u(z))$ we have $\det(\nabla^2 F)=(\Psi'(u))\cdot(u\Psi'(u))'$, as one can check directly, where $u=\vert z\vert^2_{\C^2}$ as before. Using spherical coordinates with $r^2=u$ (and hence also $r^3 dr=\frac{u du}{2}$), and recalling that $\Phi_{a}(z)=\varphi_a(u(z))$,  we therefore have
\begin{align*}
2\Vol_g(N)&=\int_{u=1}^{u=1+\delta} \int_{\S^3} \det(\nabla^2 \Phi_{a_i})\, d\Vol_{\S^{3}} \frac{u du}{2}\\
& = \frac{\Vol(\S^3)}{2} \int_{u=1}^{u=1+\delta} (u\varphi_{a_i}'(u))(u\varphi_{a_i}'(u))'\, du=\frac{\pi^2}{2} \left( u\varphi_{a_i}'(u)\right)^2 \big\vert_{u=1}^{u=1+\delta}.
\end{align*}
By \eqref{eq:GluedPot} and the fact that the cutoff function $\chi$ satisfies $\chi(u=1+\delta)=0=\chi'(u=1)=\chi'(u=1+\delta)$ and $\chi(u=1)=1$, we find $\varphi_{a_i}'(u=1)=\sqrt{1+a_i^2}$ and $\varphi_{a_i}(u=1+\delta)=1$, hence
\[2\Vol_g(N_i)=2\Vol_{Euc}(N_i)-\frac{\pi^2a_i^2}{2}.\]
Summing over $1\leq i\leq 16$ and inserting back into \eqref{eq:AComp2} gives \eqref{eq:AValue}.
\end{proof}
\begin{Rem}
The fact that the expression \eqref{eq:AValue} becomes negative for large enough values of $\vert a\vert^2$ shows why one had to restrict to small values of $\vert a\vert^2$ when gluing the metrics. The problem is traceable to the fact that \eqref{eq:GluedPot} ceases to be plurisubharmonic for large values of $\vert a\vert^2$.

We also remark that \eqref{eq:AValue} is independent of the choice of cutoff since 
$\int_X \omega^2$ only depends on the K\"{a}hler class and $\eta$ is metric-independent. As such, the positivity of \eqref{eq:AValue} puts a strict upper bound on $\vert a\vert$.
\end{Rem}

\begin{Def}
Let $X$ be a Kummer K3 surface with nowhere vanishing holomorphic volume form $\eta$. A pair $z,w$ of locally defined coordinates on $X$ such that $\eta=\sqrt{2} dz\wedge dw$ are called \textbf{holomorphic Darboux coordinates}.
\end{Def}

\begin{Rem}
The factor $\sqrt{2}$ is unconventional but convenient. If $\omega=i g_{\mu \overline{\nu}} dz^\mu d\overline{z}^\nu$ is the K\"{a}hler form in holomorphic Darboux coordinates, then 
\[\omega^2=\det(g)\eta \wedge \overline{\eta}.\]
\end{Rem}

\subsection{Estimates}
The metrics $g$ and $\tilde{g}$ on a Kummer K3 surface are related by the (non-linear) elliptic PDE \eqref{eq:MA}, so one could hope to get estimates on $g-\tilde{g}$ using Moser iteration and a maximum principle due to \cite{Yau78}. This works as long as the curvature is sufficiently concentrated and as long as no component $E_i$ is shrunk a lot faster than the rest. Concretely, we make the following assumption.
\begin{assum}
\label{Assumption:ra}
Let
\[r_a\coloneqq \frac{\max_i a_i}{\min_i a_i}\]
be the \textbf{ratio} between the largest and smallest component of the exceptional divisor $E$. We assume there is a constant $C\geq 1$ independent of $a$ such that  
\[r_a\leq C\]
for all values $a=(a_1,\dots, a_{16})$ under consideration.
\end{assum}
The estimates we need were obtained by R. Kobayashi in \cite{Kob90}, and the results are as follows.
\begin{Thm}[{\cite[Equations 46-48]{Kob90}}]
\label{Thm:Kob}
Assume $X$ is a Kummer K3 surface with K\"{a}hler form $\omega$ associated to the patchwork metric $g$ and $\tilde{\omega}=\omega+i\partial \overline{\partial} \phi$ satisfies \eqref{eq:MA} with $\int_X \phi \,\omega^2=0 $. Let $U\subset X$ be any open set such that $E\subset U$. Then there are constants $C_k>0$, $k\geq 0$, depending on $U$ but not on the parameters $a_i$ such that for all small enough values of $\vert a\vert$ we have
\begin{equation}
\norm{\phi}_{C^k(X\setminus U, g)}\leq C_k \vert a\vert^2.
\label{eq:FarBounds}
\end{equation}

Moreover, there is a constant $C>0$, independent of the parameters $a_i$ such that for all small enough values of $\vert a\vert$ we have
\begin{equation}
\norm{\phi}_{C^{k}(X,g)}\leq C \vert a\vert^{2-\frac{k}{2}}
\label{eq:GlobalBounds}
\end{equation}
\end{Thm}
Kobayashi only states \eqref{eq:GlobalBounds} for even $k$'s. The odd cases are handled by the Gagliardo-Nierenberg interpolation inequality (see for instance \cite[3.70 Theorem]{Aubin}).

These estimates will be crucial in Section \ref{Section:NoGo} to prove Theorem \ref{Thm:IntroNoGo1} and \ref{Thm:IntroNoGo2}.
\begin{Rem}
In words, \eqref{eq:FarBounds} says that when the exceptional divisor $E$ is small, $g$ is close to $\tilde{g}$ as long as one stays away from  $E$. The second bound \eqref{eq:GlobalBounds} gives weaker estimates valid also near $E$. The $C^k$-estimates of $\phi$ translate into $C^{k-2}$-estimates for the metric $\tilde{g}$. The $C^4$-estimates of $\phi$ are sometimes called estimates at the level of curvature.
\end{Rem}

We will reprove Theorem \ref{Thm:Kob} in Section \ref{Section:Estimates}. Our arguments are roughly the same as Kobayashi's, with some added detail and some corrections.

\subsection{Isometries}
Even though the Ricci-flat metric $\tilde{g}$ is not known explicitly, it will inherit the isometries of $g$.

\begin{Prop}
\label{Prop:CYIsom}
Assume the setup of Theorem \ref{Thm:CY}. Let $\tilde{g}$ be the Ricci-flat K\"{a}hler metric in the K\"{a}hler class of $g$. Assume $F\colon X\to X$ is a (anti-) holomorphic isometry of $g$. Then $F$ is also an isometry of $\tilde{g}$.  
\end{Prop}
\begin{proof}
The isometry $F$ preserves the  Monge-Amp\`{e}re equation \eqref{eq:MA} and the K\"{a}hler form $\omega$ up to sign. By the uniqueness of the solution, it also has to preserve $\tilde{\omega}$ up to sign. The details are as follows.

Let $\omega$ and $\tilde{\omega}$ be the K\"{a}hler forms of $g$ and $\tilde{g}$ respectively.
Let $\epsilon =+1$ if $F$ is holomorphic and $\epsilon=-1$ if $F$ is anti-holomorphic. Then $F^*\omega =\epsilon \omega$. Since $F$ is an isometry, we also have $\psi \circ F=\psi$ and 
\[A=\frac{\int_X \omega^n}{\int_X e^{\psi} \omega^n}=\frac{\int_X F^*\left(\omega^n\right)}{\int_X F^*\left(e^{\psi} \omega^n\right) }.\]
 Applying $F^*$ to the Monge-Amp\`{e}re equation \eqref{eq:MA} gives us
\[(F^*\tilde{\omega})^n=F^*(\tilde{\omega}^n)=F^*\left(Ae^{\psi}\omega^n\right)=Ae^{\psi\circ F} (F^*\omega)^n=\epsilon^n A e^\psi \omega^n =\epsilon^n \tilde{\omega}^n.\]
So $\epsilon F^*\tilde{\omega}$ solves the Monge-Amp\`{e}re equation. Since $F$ is (anti-) holomorphic, we find
\[F^* (\tilde{\omega})=\epsilon \left(\omega+i\partial \overline{\partial} \phi\circ F\right).\]
So both $\phi$ and $\phi \circ F$ solve the equation subject to the normalisation
\[\int_X \phi\, \omega^n =\int_X (\phi\circ F) F^*\omega^n=0.\]
By the uniqueness, we that conclude 
$\phi=\phi\circ F$ and $F$ is an isometry of $\tilde{g}$.
\end{proof}
\begin{Rem}
Proposition \ref{Prop:CYIsom} is known to the experts. In \cite[Proposition 2.2]{AG90} a very similar statement is proven for projective Calabi-Yau manifolds, $\iota\colon X\to \C\P^N$ when one takes $\omega=\iota^*(\omega_{FS})$, i.e. the metric is induced by the Fubini-Study metric. Their arguments are essentially the above ones.
\end{Rem}

\begin{Prop}
\label{Prop:KummerIsom}
Assume $X$ is a Kummer K3 surface with patchwork metric as described above. Assume all the components of the exceptional divisor have the same size, meaning $a_i=a_j$ for $1\leq i,j\leq 16$. Let $\tilde{g}$ be the unique Ricci-flat metric in the K\"{a}hler class of $g$. Assume $F^{\C}\colon \C^2\to \C^2$ is an affine map $F^{\C}(z)=Bz+b$ with $B\in U(2)$ such that $B\Gamma=\Gamma$ and $b\in \frac{1}{2}\Gamma$. Then  $F^{\C}$ induces an isometry $F\colon X\to X$ of both $g$ and $\tilde{g}$. 

Similarly, if $\tau^{\C}\colon \C^2\to \C^2$ denotes the complex conjugation map and $\tau^{\C}(\Gamma)=\Gamma$, then $\tau^{\C}$ induces an isometry $\tau\colon X\to X$ of both $g$ and $\tilde{g}$. 

Maps of these form are the only (anti-) holomorphic isometries of $(X,g)$.
\end{Prop}

\begin{proof}
The argument that one gets an induced map $F\colon X\to X$ goes as follows. Since $F^{\C}(z+\Gamma)=F^{\C}(z)+\Gamma$ holds for all $z\in \C^2$ we get an induced map $F^T\colon T\to T$. The requirement $b\in \frac{1}{2}\Gamma$ implies $F^T(-z)=-F^T(z)$, so we get a well-defined map $F^Y\colon Y\to Y$. Since this came from an affine map, it extends to the blow-up, and the argument, when written out, looks like this. To extend to the blow-up, it suffices to see what happens locally. Since blowing up commutes with taking the quotient, i.e. $Bl_0(\C^2/\mu_2)=Bl_0(\C^2)/\mu_2$, it suffices to see what happens when blowing up points in $\C^2$.  In this case, we extend a given affine map $F^{\C}\colon \C^2\to \C^2$ to a map $F^{Bl}\colon Bl_p(\C^2)\to Bl_{F(p)}(\C^2)$ by sending a pair $(q,\ell)\in \C^2\times \C\P^1$ to $(F(q),F(\ell))$, which makes sense since affine maps map lines to lines. This gives us our required map $F\colon X\to X$. The same line of arguments works for anti-holomorphic maps.

To see that the induced map is an isometry of $g$, we split $X$ into two different kinds of regions; the flat region and the neck+Eguchi-Hanson regions (the sets called $U_i$ in the construction above). In the flat region there is nothing to show. In one of the patches $U_i$ we have a metric whose K\"{a}hler potential is given by \eqref{eq:GluedPot}, and this potential is preserved by maps of the above form. Completely analogous arguments work for the map $\tau$.
 
We may therefore apply Proposition \ref{Prop:CYIsom}. 
 
 To see that the above maps are the only (anti-) holomorphic isometries of $(X,g)$ one can argue using the curvature as follows. The neck regions are not Ricci-flat, whereas the complement is. So an isometry has to map the neck regions to neck regions. The Euclidean region is flat, and the Eguchi-Hanson patches has nowhere vanishing curvature. So these regions cannot be interchanged either. From this it follows that the maps need to be of the above form.
\end{proof}
\begin{Rem}
We do not know is the isometry group of $(X,\tilde{g})$ is bigger than that of $(X,g)$. For toy models, this can easily be the case. Consider $(\C^n/\Gamma,g_0)$ with any K\"{a}hler metric $g_0$ without non-trivial isometries. Use Theorem \ref{Thm:CY} to find the Ricci-flat metric $g$ in the K\"{a}hler class of $g_0$. Then $g$ is the flat metric, and thus has an infinite isometry group, even though the isometry group of $g_0$ is trivial.
\end{Rem}


\subsection{Homothety of a Kummer K3 surface}
Both the Euclidean metric and the Eguchi-Hanson K\"{a}hler potentials are homogeneous in simultaneous scaling of the coordinate $z$ and the parameter $a$. This property will then be inherited by the solution $\phi$. Here are the details.

Let $\alpha>0$ and consider the homothety $S_{\alpha}^{\C}\colon \C^2\to \C^2$ given by 
\[S_{\alpha}^{\C}(z)=\alpha z.\]
Let $\Gamma \subset \C^2$ be any non-degenerate lattice and let $\Gamma_{\alpha}=S_{\alpha}^{\C} (\Gamma)$ be the scaled lattice. Denote by $X$ and $X_{\alpha}$ the Kummer K3 surface associated to the lattice $\Gamma$ and $\Gamma_{\alpha}$ respectively. Let $g_a$ denote the patchwork metric on $X$ with parameter $a=(a_1,\dots, a_{16})$, and denote by $\tilde{g}_{a}$ the Ricci-flat metric in the K\"{a}hler class of $g_a$. These are denoted by $g$ and $\tilde{g}$ respectively for most of the paper. On $X_{\alpha}$ we put the same kind of patchwork metric with locally defined K\"{a}hler potential \eqref{eq:GluedPot}, but we modify the cutoff function to be
\[\chi_{\alpha}(z)=\chi\left(\frac{z}{\alpha}\right)=\chi(S_{\alpha}^{-1}(z)).\]
\begin{Lem}
\label{Lem:Homothety}
The map $S_{\alpha}^{\C}$ induces an isometry
\[S_{\alpha}\colon (X,\alpha^2g_{a}) \to (X_{\alpha},g_{\alpha^2 a})\]
and
\[S_{\alpha}\colon (X,\alpha^2\tilde{g}_{a}) \to (X_{\alpha},\tilde{g}_{\alpha^2 a}).\]
\end{Lem}
\begin{proof}
That we get induced maps can be argued as in the proof Proposition \ref{Prop:KummerIsom}. The important properties are $S_{\alpha}^{\C}(z)=-S_{\alpha}^{\C}(-z)$ and that $S_{\alpha}^{\C}$ is affine.

One easily checks that the Eguchi-Hanson potential \eqref{eq:EHPot} satisfies 
\[f_{\alpha^2a}(\alpha z)=\alpha^2 f_{a}(z).\]
The same goes for the Euclidean potential. Hence, in an Eguchi-Hanson and neck region,
\[\Phi_{\alpha^2 a}(\alpha z)=\alpha^2 \Phi_a(z).\]
Away from the Eguchi-Hanson and neck regions, the isometry is clear. This shows the claimed isometry with respect to the patchwork metric. 

For the isometry of the Ricci-flat metric, we argue as in the proof of Proposition \ref{Prop:CYIsom}. It is also clear that $Ae^{\psi}$ is the same for $(X,g_a)$ and $(X_{\alpha},g_{\alpha^2 a})$. Hence\footnote{Here we are indulging in some abuse of notation. $\phi_a \colon X\to \R$ whereas $\phi_{\alpha^2a} \colon X_{\alpha}\to \R$. We also write simply $\psi$ even though it depends on $a$ and appears once as a function $\psi\colon X\to \R$ and once as a function $\psi\colon X_{\alpha} \to \R$.} 
\[(\alpha^2 (\omega_a + i\partial\overline{\partial} \phi_a))^2(z)=Ae^{\psi}(\alpha^2 \omega_a)^2(z)=Ae^{\psi}( \omega_{\al^2a})^2(S_{\al}(z))=( \omega_{\alpha^2a} + i\partial\overline{\partial} \phi_{\alpha^2a})^2(S_{\al}(z)),\]
which shows that both $\alpha^2\phi_a(z)$ and $\phi_{\alpha^2 a}(S_{\alpha}(z))$ solve the Monge-Amp\`{e}re equation. By uniqueness, $\alpha^2\phi_a(z)=\phi_{\alpha^2 a}(S_{\alpha}(z))$ and $S_{\alpha}$ is an isometry between the Ricci-flat metrics as well.
\end{proof}

The homothety $S_{\alpha}$ clearly maps $\Gamma/2\Gamma$ to $\Gamma_{\alpha}/2\Gamma_{\alpha}$, hence maps the exceptional divisor of $X$ to the exceptional divisor of $X_{\alpha}$. In fact, $(S_{\alpha})_{\vert E}$ is $\alpha$-independent. So  Lemma \ref{Lem:Homothety} says in particular
\[(\alpha^2 \phi_a)_{\vert E}=(\phi_{\alpha^2 a})_{\vert E}.\]
Here we are abusing notation again, since the right hand $E$ is a subset of $X_{\alpha}$.

\section{Closed Geodesics - No-Go Theorems}
\label{Section:NoGo}
There are two main results about geodesics on a Kummer K3 surface in this section. Theorem \ref{Thm:NoGo1} constrains stable, closed geodesics to stay away from the exceptional divisor $E$, whereas Theorem \ref{Cor:NoGo} says that closed, stable geodesics cannot stay inside an Eguchi-Hanson patch. 

\begin{Thm}
\label{Thm:NoGo1}
Let $(X,g)$ be a Kummer K3 as constructed above. Let $\tilde{g}$ be the Ricci-flat metric in the K\"{a}hler class of $g$. Then, for all values of $\vert a\vert$ small enough, there is an open set $V\subset X$ with $E\subset V$ such that no stable, closed geodesic (with respect to either $g$ or $\tilde{g}$) in $X$ ever enters $V$.
\end{Thm}
\begin{proof}
The idea is to use the estimates of \cite{Kob90} to show that the curvature of $\tilde{g}$ doesn't vanish anywhere near $E$. Then we appeal to Theorem \ref{Thm:BY} to get our conclusion for geodesics with respect to $\tilde{g}$. Here are the details. 

Let $p\in E_i \subset E$ and pick holomorphic normal coordinates with respect to $g$ at $p$, meaning $g_{\mu\overline{\nu}}(p)=\delta_{\mu\nu}$, $g_{\mu\overline{\nu},\alpha}(p)=0,$
and $g_{\mu\overline{\nu},\alpha\overline{\beta}}(p)=-R_{\mu\overline{\nu}\alpha\overline{\beta}}(p)$ hold, where $R$ is the Riemann curvature tensor of $g$. Write $\phi_{\mu\overline{\nu}}(p)\coloneqq \frac{\partial^2 \phi}{\partial z^{\mu} \partial \overline{z}^{\nu} }(p)$ and so on for more indices. Introduce the $2\times 2$-matrix $h$ via $h\coloneqq \tilde{g}^{-1}(p) - g^{-1}(p)=\tilde{g}^{-1}(p)-\mathbb{1}$. By \eqref{eq:GlobalBounds} for $k=2$ we know that $\phi_{\mu\overline{\nu}}(p)\in \mathcal{O}(\vert a\vert)$, hence also $h\in \mathcal{O}(\vert a\vert)$, as is seen by writing $\tilde{g}^{-1}(p)=(\mathbb{1}+\nabla^2\phi(p))^{-1}=\mathbb{1} +\sum_{k=1}^{\infty} (-1)^k (\nabla^2 \phi(p))^k$. We may use a standard formula for the Riemann tensor associated to a K\"{a}hler metric (see \cite[Eq. 1.14]{Yau78} for instance)  namely
\[\tilde{R}_{\mu\overline{\nu}\alpha\overline{\beta}}=-\frac{\partial^2 \tilde{g}_{\mu\overline{\nu}}}{\partial z^\alpha \partial \overline{z}^\beta} +\tilde{g}^{\overline{\lambda} \sigma} \frac{\partial \tilde{g}_{\mu\overline{\lambda}}}{\partial z^\alpha}\frac{\partial \tilde{g}_{\sigma \overline{\nu}}}{\partial \overline{z}^\beta}\]
to write the Riemann tensor $\tilde{R}$ of $\tilde{g}$ in the above coordinates as
\[\tilde{R}_{\mu\overline{\nu}\alpha\overline{\beta}}(p)=R_{\mu\overline{\nu}\alpha\overline{\beta}}(p)- \phi_{\mu\overline{\nu} \alpha \overline{\beta}}(p)+(\delta^{\lambda \sigma}+h^{\overline{\lambda} \sigma})\phi_{\mu\alpha \overline{\lambda}}(p)\phi_{\overline{\nu}\overline{\beta}\sigma}(p).\]
By \eqref{eq:GlobalBounds} for $k=3$ and $k=4$ and the above bound on $h$ we may estimate this as
\begin{equation}
R_{\mu\overline{\nu}\alpha\overline{\beta}}(p)-C \leq \tilde{R}_{\mu\overline{\nu} \alpha \overline{\beta}}(p) \leq R_{\mu\overline{\nu}\alpha\overline{\beta}}(p)+C
\label{eq:RiemBound1}
\end{equation}
for some $\vert a\vert$- independent constant $C>0$. Let $V\in T_p E_i \subset T_pX$ be such that $\vert V\vert_{\tilde{g}}=1$. Then \eqref{eq:GlobalBounds} for $k=2$ gives $1-\tilde{C}\vert a\vert \leq \vert V\vert_g^2 \leq 1+\tilde{C}\vert a\vert$ for some constant $\tilde{C}>0$. Inserting this into \eqref{eq:RiemBound1} we find
\begin{equation}
Sect_g(p)(1-\tilde{C}\vert a\vert)^2 -C \leq Sect_{\tilde{g}}(p) \leq (1+\tilde{C}\vert a\vert) ^2 Sect_g(p)+C,
\label{eq:SectBound1}
\end{equation}
where $Sect_g(p)$ and $Sect_{\tilde{g}}(p)$ denote the holomorphic sectional curvatures of $g$ and $\tilde{g}$ respectively, evaluated on $T_pE_i$. Since we are on $E_i$ we may use that $g_{\vert E_i} =a_i g_{FS}$,  FS being short for Fubini-Study, to write $Sect_g(p)=\frac{2}{a_i}$. Using the ratio $r_a = \frac{\max_{1\leq i\leq 16} a_i}{\min_{1\leq i\leq 16} a_i}$, we bound this as
$\frac{2}{\vert a\vert} \leq Sect_g(p)\leq \frac{4r_a}{\vert a\vert}$, hence \eqref{eq:SectBound1} finally becomes
\[\frac{2}{\vert a\vert} -C_1 +C_2\vert a\vert \leq Sect_{\tilde{g}}(p)\leq 4r_a \left( \frac{1}{\vert a\vert} +C_3 +C_4\vert a\vert\right)\]
for $\vert a\vert$-independent constants $C_i>0$. This shows that $Sect_{\tilde{g}}(p)>0$ for all $\vert a\vert$ small enough. 

To see that there are no closed, stable geodesics with respect to $g$ one can argue in a couple of ways. The fastest is probably to appeal to Theorem \ref{Thm:BY}, which applies since the Eguchi-Hanson patch has a hyperk\"{a}hler metric whose Riemann curvature tensor doesn't vanish in  any point (see \cite[Equation 2.28]{EH79} or \cite[Equation 4.8]{PhD}). Another argument is to first see directly (see \cite[Theorem 5.3]{TsaiWang},  \cite[Chapter 4.5]{PhD}), or \cite[Theorem 8]{LyeEH} that the only closed, non-constant geodesics in Eguchi-Hanson space are the ones contained in the $\C\P^1$. Hence they are closed geodesics in $(\C\P^1,g_{FS})$, all of which are unstable.

\end{proof}

\begin{Rem}
That the components $E_i$ of the exceptional divisor have induced metrics with positive curvature (even for the Ricci-flat metric) can be seen in numerical solutions like \cite[Figure 3]{NumK3}.
\end{Rem}

The next result says that any geodesic with respect to $\tilde{g}$ is close to being a geodesic with respect to $g$. 

%

\begin{Thm}
\label{Thm:NoGo2}
Assume the same setup as in Theorem \ref{Thm:NoGo1}. Then, for all values of $\vert a\vert$ small enough there is a constant $C>0$ independent of $\vert a\vert$ such that if $\gamma\colon (-\epsilon,\epsilon)\to X$ is a geodesic with respect to $\tilde{g}$ we have 
\begin{equation}
\vert D_t^g \dot{\gamma}(t)\vert_{g} \leq C\vert a\vert^{\frac{1}{2}} \vert \dot{\gamma}\vert_g^2,
\label{eq:GlobalGeodClose}
\end{equation}
where $D_t^g=\nabla_{\dot{\gamma}}$ is the covariant derivative associated to $g$.
\end{Thm}
\begin{proof}
Let $\Gamma$ and  $\tilde{\Gamma}$ denote the Christoffel symbols in some coordinates of $g$ and $\tilde{g}$ respectively. Define the tensor $\Psi$ by 
\begin{equation}
\Psi_{\mu\alpha}^\lambda \coloneqq \tilde{\Gamma}_{\mu\alpha}^\lambda -\Gamma_{\mu\alpha}^\lambda.
\label{eq:PsiDef}
\end{equation}
Then
\begin{equation}
\Psi_{\mu\alpha}^\lambda =\tilde{g}^{\overline{\sigma}\lambda}\left(\tilde{g}_{\alpha,\overline{\sigma},\mu}-\Gamma_{\alpha\mu}^\rho \tilde{g}_{\rho\overline{\sigma}}\right)= \tilde{g}^{\overline{\sigma}\lambda} \nabla_\mu \tilde{g}_{\alpha\overline{\sigma}},
\label{eq:PsiDef2}
\end{equation}
where $\nabla$ denotes the covariant derivative associated to $g$. Let  $\gamma\colon (-\epsilon,\epsilon)\to X$ be a geodesic with respect to $\tilde{g}$. Then
\begin{equation}
\vert D_t^g \dot{\gamma}\vert_g =\left\vert D_t^g\dot{\gamma}-D_t^{\tilde{g}} \dot{\gamma}\right\vert_g=\left\vert \Psi(\dot{\gamma},\dot{\gamma})\right\vert_g.
\end{equation}
This is the equation which will give us \eqref{eq:GlobalGeodClose} after we estimate $\left\vert \Psi(\dot{\gamma},\dot{\gamma})\right\vert $. To this effect, we claim there exists $C>0$ independent of $\vert a\vert$ such that $\norm{\Psi}_{C^0(M,g)} \leq C\vert a\vert^{\frac{1}{2}}$, and this will follow by Kobayashi's estimates. Let $p\in X$ and, like in the proof of Theorem \ref{Thm:NoGo1}, choose holomorphic normal coordinates at $p$ with respect to $g$, meaning $g_{\mu\overline{\nu}}(p)=\delta_{\mu\nu}$ etc. In these coordinates we may write
\[\Psi(p)_{\mu\alpha}^\lambda =(\delta^{\sigma \lambda} +h^{\overline{\sigma}\lambda})\phi(p)_{\mu\alpha\overline{\sigma}},\]
where $h$ and $\phi(p)_{\mu\alpha\overline{\sigma}}$ are as in the proof of Theorem \ref{Thm:NoGo1}. In that proof we also saw that $h\in \mathcal{O}(\vert a\vert)$. From \eqref{eq:GlobalBounds} with $k=3$ we have $\vert \phi(p)_{\mu\alpha\overline{\sigma}}\vert \leq \norm{\phi}_{C^3(X,g)} \leq C_1\vert a\vert^{\frac{1}{2}}$. Altogether we thus find
\[\vert \Psi(p)_{\mu\alpha}^\lambda\vert \leq C_1\vert a\vert^{\frac{1}{2}} (1+C_2 \vert a\vert)\leq C\vert a\vert^{\frac{1}{2}}.\]
At the point $p$, still in holomorphic normal coordinates, we thus compute
\[ \vert \Psi\vert_g^2(p)=\sum_{\mu,\alpha, \lambda= 1}^2 \left\vert \Psi(p)_{\mu \alpha}^\lambda\right \vert^2\leq \tilde{C} \vert a\vert.\]
The left hand side is independent of coordinate system, and since $p$ was arbitrary this proves the claim.

\end{proof}

\begin{Thm}
\label{Cor:NoGo}
Assume the same setup as in Theorem \ref{Thm:NoGo1}. Then, for all values of $\vert a\vert$ small enough, no stable, closed geodesic with respect to $\tilde{g}$ can stay completely within an Eguchi-Hanson patch $U_i$.
\end{Thm}

\begin{proof}
Assume there is a closed, stable geodesic $\gamma$ with respect to $\tilde{g}$ in an Eguchi-Hanson patch $U_i$. By Theorem \ref{Thm:NoGo1} we may assume that $\gamma$ stays away from $E_i\cong \C\P^1$. Consider the distance squared $d(t)$  between $\gamma(t)$ and $\C\P^1$. This distance is realized by radial geodesics as follow. Let $\rho_s(t)$ denote the family of radial geodesics with respect to the metric $g$ connecting $\gamma(t)$ and $\C\P^1$, meaning $\rho_0(t)\in \C\P^1$ for all $t$, $\rho_{1}(t)=\gamma(t)$, and $D_s^g \partial_s \rho_s(t)=0$ for all $t\in \S^1$ and $s\in (0,1)$. Then 
\[d(t)=\int_0^1 \vert \partial_s \rho_s(t)\vert^2_g\, ds.\] 
The function $d$ needs to have a maximum, meaning there is some $T$ such that $\dot{d}(T)=0$ and $\ddot{d}(T)\leq 0$. We shall show that \eqref{eq:GlobalGeodClose} implies  $\ddot{d}(T)>0$, hence forcing a contradiction.  

We start by computing $\dot{d}(t)$. Note that $D_t^g \partial_s =D_s^g \partial_t$ (see for instance \cite[3.4 Lemma]{DoCarmo}). Since $s\mapsto \rho_s(t)$ is a geodesic, $D_s^g \partial_s \rho_s(t)=0$. So\footnote{In this section we are thinking of the metric as a \textit{hermitian} metric, whose real part is the corresponding Riemannian metric. This is to make better use of the complex coordinates on $\C^2$.}
\begin{align*}
\dot{d}(t)&=2\text{Re} \int_0^1 \ip{D^g_t \partial_s \rho_s(t)}{\partial_s \rho_s(t)}_g\, ds \\ &=  2\text{Re} \int_0^1 \ip{D^g_s \partial_t \rho_s(t)}{\partial_s \rho_s(t)}_g\, ds\\  & =2\text{Re} \int_0^1 \partial_s  \left( \ip{\partial_t\rho_s(t)}{\partial_s\rho_s(t)}_g\right)\, ds \\ &= 2\text{Re}\left(\ip{\partial_t\rho_s(t)}{\partial_s\rho_s(t)}_g\right)\big\vert_{s=0}^{s=1}.
\end{align*}
The lower limit vanishes for all $t$ since $\partial_s \rho_s(t)_{\vert s=0}$ is tangential to $E_i\cong \C\P^1$ whereas $\partial_t \rho_{s}(t)_{\vert s=0}$ is normal.
Differentiating this again we find
\[\ddot{d}(t)=2\text{Re}\left(\ip{D^g_t\partial_t\rho_s(t)}{\partial_s\rho_s(t)}_g+\ip{\partial_t \rho_s(t)}{D_t^g \partial_s \rho_s(t)}_g\right)\big\vert_{s=1}.\]
To say something more about these expressions we need to recall some basic facts about the Eguchi-Hanson metric $g$.
 The region we are interested in is $(B_1(0)\setminus \{0\})/\mu_2$ with K\"{a}hler potential given by \eqref{eq:EHPot}. We use $z=(z_1, z_2)$ as coordinates and write $u=\vert z\vert^2_{\C^2}$ as before. The K\"{a}hler metric associated to \eqref{eq:EHPot} reads
\begin{equation}
\ip{U}{V}_g=\sqrt{1+\frac{a_i^2}{u^2}} \left(\ip{U}{V}_{\C^2} -\frac{a^2_i}{a^2_i+u^2} \frac{\ip{U}{z}_{\C^2}\ip{z}{V}_{\C^2}}{u}\right),
\label{eq:EHMetric}
\end{equation}
where $\ip{U}{V}_{\C^2}\coloneqq \overline{U_1}V_1+\overline{U_2}V_2$ is the Euclidean inner product.
From this it follows that 
\begin{equation}
\ip{z}{V}_g=\frac{u}{\sqrt{a_i^2+u^2}} \ip{z}{V}_{\C^2}
\label{eq:zIp}
\end{equation}
holds for any $V\in \C^2$.

Using the formula (see for instance \cite[Equation 4.39]{Kahler}) $\Gamma_{\mu\alpha}^{\lambda} = \frac{\partial g_{\mu \overline{\nu}}}{\partial z^\alpha}  g^{\overline{\nu}\lambda}$ we also find
\begin{equation}
\Gamma_{\mu\alpha}^{\lambda} =-\frac{a_i^2}{u(a_i^2+u^2)} \left( \overline{z}_{\mu} \delta^{\lambda}{}_\alpha +\overline{z}_{\alpha} \delta^{\lambda}{}_\mu -3 \frac{\overline{z}_\alpha \overline{z}_\mu}{u} z^{\lambda}\right).
\label{eq:Christoffel}
\end{equation}
We can locally write $\rho_s(t)=\theta(s,t) z(t)$ for some function $\theta$ satisfying (amongst others\footnote{One can find $\theta$ more or less explicitly. Since $\rho_s(t)$ is supposed to be a radial geodesic, the geodesic equation has to be fulfilled, so using \eqref{eq:Christoffel} and \eqref{eq:zIp} we see that $\partial_s^2 \theta + \frac{a_i^2}{(a_i^2 + u^2 \theta^4)\theta} (\partial_s\theta)^2=0$ is the ODE satisfied by $\theta$. This has $\partial_s \theta(s,t) =\frac{\sqrt{d(t)}(a_i^2 + \theta(s,t)^4 u(t)^2)^{\frac{1}{4}}}{\theta(s,t)u(t)}$ as a first integral, where we can write $\sqrt{d(t)}=\int_0^1 \frac{s u(t)}{(a_i^2 +s^4 u(t)^2)^{\frac{1}{4}}}\, ds$. We shall not need these explicit expressions, however.} ) $\theta(1,t)=1$ and $\partial_s \theta >0$ for all $t$. In particular $\partial_t \rho_{\vert s=1}=\dot{z}(t)$. Inserting this into our above expressions for $\dot{d}(t)$,  we find
\begin{align*}
\dot{d}(t)&=2\text{Re}\left( (\partial_t \theta)(\partial_s \theta) \vert z(t)\vert^2_g + \theta (\partial_s \theta) \ip{z(t)}{\dot{z}(t)}_g\right)\big \vert_{s=1}\\
 &= 2(\partial_s \theta)(1,t) \text{Re}  \ip{z(t)}{\dot{z}(t)}_g.
 \end{align*}
 We deduce that $\dot{d}(T)=0\iff \text{Re}\ip{z}{\dot{z}}_g(T)=0$, which by \eqref{eq:zIp} happens if and only if $\text{Re}\ip{z}{\dot{z}}_{\C^2}(T)=0$. We similarly find 
\[\ip{D^g_t\partial_t\rho_s(t)}{\partial_s\rho_s(t)}_g \big\vert_{s=1}=\partial_s \theta (1,t) \ip{D_t^g \dot{z}(t)}{z(t)}_g. \]
The term $\ip{\partial_t \rho_s(t)}{D_t^g \partial_s \rho_s(t)}_g\big\vert_{s=1}$ needs a bit more work. We have
\begin{align*}
D_t^g \partial_s \rho_s^\lambda&=D_t^g(\partial_s \theta  z)^\lambda =\partial_t(\partial_s \theta z)^\lambda +(\partial_s \theta) \Gamma^\lambda_{\mu\alpha} \dot{z}^\mu z^\alpha\\
&= (\partial_t\partial_s\theta) z^{\lambda} +(\partial_s \theta) \dot{z}^{\lambda} +(\partial_s \theta) \Gamma^\lambda_{\mu\alpha} \dot{z}^\mu z^\alpha.
\end{align*}
Here the indices are raised using the Euclidean metric; $z_\mu=z^{\mu}$ and so on.
Using \eqref{eq:Christoffel} we have 
\[\Gamma^\lambda_{\mu\alpha} \dot{z}^\mu z^\alpha=-\frac{a_i^2}{a_i^2 +u^2} \dot{z}^\lambda+\frac{2a_i^2}{u(a_i^2+u^2)} \ip{z}{\dot{z}}_{\C^2}z^{\lambda}\stackrel{\eqref{eq:zIp}}{=}-\frac{a_i^2}{a_i^2 +u^2} \dot{z}^\lambda+\frac{2a_i^2}{\sqrt{a_i^2+u^2}}\ip{z}{\dot{z}}_g z^{\lambda}. \]
 Inserting this, we find
\begin{align*}
D_t^g \partial_s \rho_s^\lambda= (\partial_t\partial_s\theta) z^{\lambda}+(\partial_s\theta) \frac{u^2}{a_i^2+u^2} \dot{z}^{\lambda}+(\partial_s\theta)\frac{2a_i^2}{\sqrt{a_i^2+u^2}}\ip{z}{\dot{z}}_g z^{\lambda}.
\end{align*}
 Hence
\begin{align*}
&\ip{\partial_t \rho_s(t)}{D_t^g \partial_s \rho_s(t)}_g\big\vert_{s=1}=\ip{\dot{z}(t)}{D_t^g \partial_s \rho_s(t)}_g\\
& = (\partial_t\partial_s\theta)\ip{\dot{z}(t)}{z(t)}_g +(\partial_s\theta)\left(\frac{2a_i^2}{\sqrt{a_i^2+u(t)^2}} \vert \ip{z(t)}{\dot{z}(t)}_g\vert^2+\frac{u(t)^2}{a_i^2+u(t)^2}\vert \dot{z}(t)\vert_g^2\right).
\end{align*}
Taking the real part and setting $t=T$ removes the term $ (\partial_t\partial_s\theta)\ip{\dot{z}(t)}{z(t)}_g$ since $\text{Re}\ip{\dot{z}(T)}{z(T)}_g=0$.

We thus conclude
\[\ddot{d}(T)=2(\partial_s \theta(1,T)) \left( \text{Re}\ip{z}{D_t^g \dot{z}}_g +\frac{u^2}{a_i^2+u^2} \vert \dot{z}\vert^2_g +\frac{2a_i^2}{\sqrt{a_i^2+u^2}} \vert \ip{z}{\dot{z}}_g\vert^2\right).\]
We estimate this from below by dropping the non-negative term $\frac{2a_i^2}{\sqrt{a_i^2+u^2}} \vert \ip{z}{\dot{z}}_g\vert^2$ (recalling that $\partial_s \theta>0$) and observing that
\[\text{Re}\ip{z}{D_t^g \dot{z}}_g\geq -\vert z\vert_g \vert D_t^g \dot{z}\vert_g \stackrel{\eqref{eq:GlobalGeodClose}}{\geq} -\vert z\vert_g C \vert a\vert^{\frac{1}{2}} \vert \dot{z}\vert^2_g\stackrel{\eqref{eq:zIp}}{=}-C\frac{u^2}{\sqrt{a_i^2+u^2}} \vert a\vert^{\frac{1}{2}} \vert \dot{z}\vert^2_g.\]
From this, we deduce
\[\ddot{d}(T)\geq \frac{2u^2 (\partial_s \theta)}{\sqrt{a_i^2+u^2}} \vert \dot{z}\vert_g^2\left( -C\vert a\vert^{\frac{1}{2}} +\frac{1}{\sqrt{a_i^2+u^2}}\right)\geq \frac{u^2 (\partial_s \theta)}{\sqrt{a_i^2+u^2}} \vert \dot{z}\vert_g^2\left( -C\vert a\vert^{\frac{1}{2}} +\frac{1}{\sqrt{a_i^2+1}}\right). \]
By choosing $\vert a\vert$ small enough we have $\ddot{d}(T)>0$ and $d$ therefore does not have a maximum. This is the desired contradiction.

\end{proof}

\begin{Rem}
The strategy of the above proof is similar to \cite[Theorem 5.3]{TsaiWang} and \cite[Theorem 8]{LyeEH}. Enlarging slightly to include the neck region in the consideration would not change anything, since a similar argument would go through, where one compares with the Euclidean metric instead of the Eguchi-Hanson metric. See \cite[7.3 Theorem]{PhD} for details. 
\end{Rem}

Let us stress that the above theorem forbids a stable, closed geodesic from \textit{staying} inside of $U_i$. It does not rule out a geodesic entering and leaving $U_i$ and closing up somewhere else in $X$. This is a possibility envisioned in \cite[pp. 12-13]{GD13}, and we do not rule this out, but would note the restrictions imposed upon such a geodesic by Theorem \ref{Thm:NoGo1} and \ref{Thm:NoGo2}.

\section{Curvature of Hyperk\"{a}hler 4-Manifolds}
\label{Section:Stable}
In this section, we analyse the Riemann curvature tensor of a Hyperk\"{a}hler 4-manifold. Most of the statements rely on on having enough isometries. We need some preliminaries first. We will start by recalling some facts from Riemannian geometry.

\begin{Lem}
\label{Lem:RiemIsom}
Assume $(M,g)$ is a Riemannian manifold with Levi-Civita connection $\nabla$. Assume $F\colon M\to M$ is an isometry. Then 
\begin{equation}
F_*(\nabla_U V)=\nabla_{F_*U}F_* V
\label{eq:CovIsom}
\end{equation}
holds for all tangent vector fields $U,V\in \Gamma(TM)$. In particular, if $R$ denotes the Riemann curvature tensor, then
\begin{equation}
\ip{R(V,W)U}{Z}_p =\ip{R(F_*V,F_*W)F_*U}{F_*Z}_{F(p)}
\label{eq:RiemIsom}
\end{equation}
holds for all $p\in M$ and all $U,V,W,Z\in T_pM$.
\end{Lem}
\begin{proof}
Define $\nabla'$ by $\nabla'_{U}V \coloneqq F_*^{-1} \left(\nabla_{F_*U} F_* V\right)$ and verify that $\nabla'$ is a torsion-free metric connection, hence $\nabla'=\nabla$. 
\end{proof}

\subsection{Curvature of hyperk\"{a}hler 4-manifolds}
We next turn to some facts about hyperk\"{a}hler manifolds in real dimension four and at the same time establish some notation.
\begin{Lem}
\label{Lem:HK}
Let $(X,\tilde{g})$ be a Ricci-flat K\"{a}hler manifold of real dimension $4$ with complex structure $I$. Then there exist complex structures $J$, $K$ such that $IJ=K$. These complex structures are metric compatible, meaning $\nabla J=J\nabla$ and $\nabla K=K\nabla$.
\end{Lem}
In fact, one can write down such complex structures explicitly, following \cite[p. 287]{Calabi} or \cite[p. 6]{Bielawski}. Let $z,w$ be local coordinates such that the metric $\tilde{g}$ takes the matrix form
\[\tilde{g}=\begin{pmatrix}
\tilde{g}_{z\oz}& \tilde{g}_{z\ow}\\ \tilde{g}_{w\oz} & \tilde{g}_{w\ow}
\end{pmatrix}\]
with 
\[\det(\tilde{g})=A\]
being some constant.
Then we may define a complex structure by
\begin{equation}
J\frac{\partial}{\partial z} =\frac{1}{\sqrt{A}}\left(-\tilde{g}_{z\ow} \frac{\partial}{\partial {\oz}}+\tilde{g}_{z\oz} \frac{\partial}{\partial {\ow}}\right)
\label{eq:Jaction1}
\end{equation}
and
\begin{equation}
J\frac{\partial}{\partial w} =\frac{1}{\sqrt{A}}\left(-\tilde{g}_{w\ow} \frac{\partial}{\partial {\oz}}+\tilde{g}_{w\oz} \frac{\partial}{\partial {\ow}}\right).
\label{eq:Jaction2}
\end{equation}

For any $p\in X$ and $V\in T_pM$ we introduce the notation
\[\sigma_{IJ}(V)\coloneqq \ip{R(V,IV)JV}{V}\]
and similarly for $\sigma_{IK}(V), \sigma_{JK}(V), \sigma_{II}(V)$ etc. These satisfy $\sigma_{IJ}(V)=\sigma_{JI}(V)$ and so on. 

The Ricci-flatness condition\footnote{Once one knows that there exists 3 metric-compatible, mutually orthogonal complex structures $I,J,K$, then \eqref{eq:HK-Ricci} is actually a direct consequence of the Bianchi identity.} reads
\begin{equation}
\sigma_{II}(V)+\sigma_{JJ}(V)+\sigma_{KK}(V)=0
\label{eq:HK-Ricci}
\end{equation} 
for any $V\in T_pM$ .

The significance of the above $\sigma'$s is that they determine the Riemann curvature tensor of $X$. 
\begin{Lem}
\label{Lem:SectCurv}
Let $(X,\tilde{g})$ be a hyperk\"{a}hler manifold of real dimension 4. Fix some $V\in T_pX\setminus \{0\}$ for some $p\in X$. Then the holomorphic sectional curvature of any $W\in T_pX\setminus \{0\}$ is uniquely determined by $\sigma_{II}(V), \sigma_{JJ}(V), \sigma_{IJ}(V), \sigma_{IK}(V)$, and $\sigma_{JK}(V)$.
\end{Lem}
\begin{proof}
Let $W=\alpha V +\beta IV +\mu JV +\nu KV\in T_pX$ be any tangent vector. Then, by iteratively using identities like $\ip{R(Iv,Jw)u}{t}=\ip{R(I^2v,IJw)u}{t}=-\ip{R(v,Kw)u}{t}$ for any $t,u,v,w\in T_pX$ and the standard Riemann tensor symmetries, one arrives at
\begin{align}
\ip{R(W,IW)IW}{W}&=\left(\alpha^2+\beta^2-\mu^2-\nu^2\right)^2\sigma_{II}(V)+4(\beta\mu-\alpha\nu)^2 \sigma_{JJ}(V)\notag \\ &+4(\alpha\mu+\beta\nu)^2\sigma_{KK}(V) \notag +4(\beta\mu-\alpha\nu)(\alpha^2+\beta^2-\mu^2-\nu^2)\sigma_{IJ}(V)\notag \\ &+4(\alpha\mu+\beta\nu)(\alpha^2+\beta^2-\mu^2-\nu^2)\sigma_{IK}(V) \notag \\ &+ 8(\alpha\mu+\beta\nu)(\beta\mu-\alpha\nu)\sigma_{JK}(V).
\label{eq:ExplicitSigma}
\end{align}
Using \eqref{eq:HK-Ricci} one can solve away one of the $\sigma$'s, $\sigma_{KK}(V)=-\sigma_{II}(V)-\sigma_{JJ}(V)$, say. So in particular, the holomorphic sectional curvature of $W$ is completely determined by $\sigma_{II}(V), \sigma_{JJ}(V), \sigma_{IJ}(V), \sigma_{IK}(V)$, and $\sigma_{JK}(V)$. 
\end{proof}

\begin{Rem}
\label{Rem:Angular}
A neat way of writing \eqref{eq:ExplicitSigma} is to use the following coordinates on $\S^3$. Assuming $\vert W\vert =\vert V\vert$, we can write
\[\alpha+i\beta=\cos\left(\frac{\theta}{2}\right) e^{\frac{i}{2}(\psi+\phi)}\]
\[\mu+i\nu=\sin\left(\frac{\theta}{2}\right) e^{\frac{i}{2}(\psi-\phi)}\]
where 
\[0\leq \theta <\pi\]
\[0\leq \phi\leq 2\pi\]
\[0\leq \psi\leq 4\pi\]
as in \cite[Eq. 2.4]{EH79}. Then 
\begin{align}
\ip{R(W,IW)IW}{W} &=\cos^2(\theta)\sigma_{II}(V)+\sin^2(\theta) \sin^2(\phi) \sigma_{JJ}(V) \notag \\&+\sin^2(\theta)\cos^2(\phi)\sigma_{KK}(V)  +\sin(2\theta)\sin(\phi)\sigma_{IJ}(V) \notag \\ &+\sin(2\theta)\cos(\phi)\sigma_{IK}(V) + \sin^2(\theta) \sin(2\phi)\sigma_{JK}(V).
\label{eq:ExplicitSigma2}
\end{align}
We note the absence of $\psi$ in \eqref{eq:ExplicitSigma2}, which corresponds to the $U(1)$-symmetry of holomorphic sectional curvature.
\end{Rem}

The next result shows how the presence of symmetry can drastically reduce the available degrees of freedom in the Riemann tensor.
\begin{Thm}
\label{Thm:IsomThm}
Let $(X,\tilde{g})$ be a hyperk\"{a}hler manifold of real dimension 4. Assume $F\colon X\to X$ is a holomorphic isometry of order $4$ with positive-dimensional fixed point set $M\coloneqq Fix(F)$. Then $\sigma_{IJ}(V)=\sigma_{IK}(V)=\sigma_{JK}(V)=0$ and $\sigma_{JJ}(V)=\sigma_{KK}(V)$ for any $V\in T_pM\subset T_pX$ and $p\in M$. 
\end{Thm}

\begin{proof}
Let us start by seeing how $F_*$ acts in a single tangent space. Let $p\in M$, and $V\in T_pM$. Then $V$ is necessarily fixed by $F_*$.  Since $F$ is holomorphic, $M$ is a complex submanifold, and  $IV$ is fixed by $F_*$ as well. Since $p=F(p)$, everything will be taking place at $p$ and we shall be omitting references to the point everywhere.  Consider now $F_*(JV)$. Using that $V$ is fixed, that $F$ is an isometry, and that $JV\perp V$ we deduce $\ip{V}{F_*(JV)}=\ip{F_*V}{F_*(JV)}=\ip{V}{JV}=0$. Similarly $\ip{IV}{F_*(JV)}=0$. It thus follows that there exists $\theta\in [0,2\pi)$ such that $F_*(JV)=(\cos(\theta)+I\sin(\theta))JV$. A direct computation shows $F_*^4(JV)=(\cos(4\theta)+I\sin(4\theta))JV$. By assumption $F_*^4(JV)=JV$, so $\cos(4\theta)=1$, $\sin(4\theta)=0$. The solutions $\theta=0$ and $\theta=\pi$ can be ruled out as these would make $F$ an order 1 or 2 isometry respectively (see \cite[Theorem 1.1, p. 137]{RiemGeomPetersen} for instance). We conclude that $F_*(JV)=\pm KV$ and $F_*(KV)=\mp JV$. In words; $F_*$ fixes $T_pM$ and rotates the orthogonal complement of $T_pM\subset T_pX$ by $90^\circ$.

This already provides us with plenty of information due to Lemma \ref{Lem:RiemIsom}. Recalling the notation of Lemma \ref{Lem:HK}, we next aim to show that for any $p\in M$ and  $V\in T_pX$ the following holds
\begin{equation}
\sigma_{IJ}(V)=\sigma_{IK}(V)=\sigma_{JK}(V)=0
\label{eq:OffDiag}
\end{equation}
and
\begin{equation}
\sigma_{JJ}(V)=\sigma_{KK}(V).
\label{eq:Diag}
\end{equation}
To see this, we use \eqref{eq:RiemIsom} and the above knowledge of how $F_*$ acts
\begin{align*}
\sigma_{JK}(V)&=\ip{R(V,JV)KV}{V} \\ &\overset{\eqref{eq:RiemIsom}}{=} \ip{R(F_*V,F_*(JV))F_*(KV)}{F_*V}\\ &=- \ip{R(V,KV)JV}{V}\\&=-\sigma_{KJ}(V).
\end{align*}

But $\sigma_{JK}(V)=\sigma_{KJ}(V)$, so $\sigma_{JK}(V)=0$. Similarly one computes $\sigma_{IJ}(V)=\pm \sigma_{IK}(V)$ and $\sigma_{IK}(V)=\mp \sigma_{IJ}(V)$, which combine to say $\sigma_{IJ}(V)=-\sigma_{IJ}(V)$ and $\sigma_{IK}(V)=-\sigma_{IK}(V)$. The same kind of computation also gives $\sigma_{JJ}(V)=\sigma_{KK}(V)$.

\end{proof}
The proposition can actually be generalized to isometries of other orders.
\begin{Thm}
\label{Thm:GenIsomThm}
Let $(X,\tilde{g})$ be a hyperk\"{a}hler manifold of real dimension 4. Assume $F\colon X\to X$ is a holomorphic isometry of order $k\geq 2$ with positive-dimensional fixed point set $M\coloneqq Fix(F)$. Then $\sigma_{IJ}(V)=\sigma_{IK}(V)=\sigma_{JK}(V)=0$ for any $V\in T_pM\subset T_pX$ and $p\in M$. If $k>2$, then  $\sigma_{JJ}(V)=\sigma_{KK}(V)$ and $\sigma_{JK}(V)=0$ as well.  
\end{Thm}
\begin{proof}
One argues as above that for $p\in M$ and $V\in T_pM$, $F_*$ acts by rotating the orthogonal complement by $\theta=\pm \frac{2\pi}{k}$. One then computes as above that
\[\sigma_{IJ}(V)=\cos(\theta)\sigma_{IJ}(V)+\sin(\theta)\sigma_{IK}(V)\]
and
\[\sigma_{IK}(V)=\cos(\theta)\sigma_{IK}(V)-\sin(\theta)\sigma_{IJ}(V).\]
This means $(\sigma_{IJ},\sigma_{IK})$ is fixed by a non-trivial rotation, hence has to be 0. Using this, we further compute
\[\sigma_{JJ}(V)=\cos^2(\theta) \sigma_{JJ}(V)+\sin^2(\theta)\sigma_{KK}(V)\]
When $\sin(\theta)\neq 0$, this implies $\sigma_{JJ}(V)=\sigma_{KK}(V)$. Similarly, one computes
\[\sigma_{JK}(V)=\cos(2\theta)\sigma_{JK}(V)+\sin(\theta)\cos(\theta)(\sigma_{KK}(V)-\sigma_{JJ}(V)),\]
hence $\sigma_{JK}(V)=0$. 
\end{proof}

\begin{Cor}
Assume the setup of Theorem \ref{Thm:GenIsomThm} with $k>2$.  Then the Riemann curvature tensor of $X$ at $p\in M\subset X$ is uniquely determined by the Gauss curvature of $M$ at $p$. In fact, if $\mathcal{K}\colon M\to \R$ is the Gauss curvature, $V\in T_pM \subset T_pX$ is a unit vector, and $W=\alpha V +\beta IV +\mu JV+\nu KV\in T_pX$ is arbitrary, then
\[\ip{R(W,IW)IW}{W}=\left( (\alpha^2+\beta^2)^2+(\mu^2+\nu^2)^2-4(\alpha^2+\beta^2)(\mu^2+\nu^2)\right)\mathcal{K}(p)\]
\end{Cor}
\begin{proof}
By the Gauss-Codazzi theorem (see \cite[Theorem 2.5]{DoCarmo} for instance) and the fact that $M$ is totally geodesic, $\sigma_{II}(V)$ equals the holomorphic sectional curvature (of the induced metric) of $M$, and the holomorphic sectional curvature exactly equals the Gauss curvature of a Riemann surface. The expression then follows from \eqref{eq:ExplicitSigma}.
\end{proof}
\begin{Rem}
Yet another corollary of the above result is that if the fixed point set has a connected component which is a torus, then there has to be points where the curvature of $M$ vanishes (simply by the Gauss-Bonnet theorem), hence points where the entire curvature tensor of $X$ vanishes. The zero set of the curvature on a Riemann surface does not have to be made up of geodesics however. Indeed, just compute what happens for a torus embedded in $\R^3$, via $(\theta,\phi)\mapsto \left(((R+r\cos(\theta))\cos(\phi),(R+r\cos(\theta))\sin(\phi),r\sin(\theta)\right)$. One finds that in this case, $\mathcal{K}=0$ on the two circles $\left(R\cos(\phi),R\sin(\phi),\pm r\right)$, neither of which are geodesics. 
\end{Rem}

\begin{Rem}
In the coordinates of Remark \ref{Rem:Angular}, the sectional curvature reads
\[\ip{R(W,IW)IW}{W}=\frac{1}{4}\left(3\cos(2\theta)-1\right)\mathcal{K}\]
\end{Rem}

\subsection{Derivatives of the Curvature}

To proceed, we will formulate some results about the derivative the holomorphic sectional curvature on a hyperk\"{a}hler manifold. We will  write $\nabla$ (instead of $\tilde{\nabla}$) for the covariant derivative associated to $\tilde{g}$. We only refer to a single (hyperk\"{a}hler) metric in this section, so the risk of confusion is minimal. 
\begin{Prop}
\label{Prop:RiemannCrit}
Let $(X,\tilde{g})$ be as in Theorem \ref{Thm:IsomThm}. Let $V$ be a unit vector field defined in some open neighbourhood $U$ of $M$ with the property that $V_p\in T_pM$ for all $p\in M$. Then, for any critical point of the function $\sigma_{II}(V)\colon U\to \R$ lying in $M$, we have 
\begin{align}
\Delta\ip{R(V,IV)IV}{V} &\coloneqq (\nabla_V^2+\nabla_{IV}^2+\nabla_{JV}^2+\nabla_{KV}^2) \ip{R(V,IV)IV}{V} \notag \\
&=-6\sigma_{II}(V)^2 -24 \vert \nabla_{JV} V\vert^2 \sigma_{II}(V).
\end{align} 
\end{Prop}
\begin{proof}
Let $W$ be any vector field. Then
\[\nabla_{W} \ip{R(V,IV)IV}{V}=\ip{\nabla_WR(V,IV)IV}{V}+4\ip{R(\nabla_WV,IV)IV}{V}\]
and
\begin{align*}
\nabla^2_W \ip{R(V,IV)IV}{V}&= \ip{\nabla_W^2R(V,IV)IV}{V}+8\ip{\nabla_WR(\nabla_WV,IV)IV}{V} \\
&+4\ip{R(\nabla_W^2V,IV)IV}{V}+4\ip{R(\nabla_WV,I\nabla_WV)IV}{V}\\
&+8\ip{R(\nabla_WV,IV)IV}{\nabla_WV}.
\end{align*}
The terms $\nabla_WR$ and $\nabla_W^2R$ denote covariant derivatives of $R$ as a 4-tensor, and we will deal with these terms last using the second Bianchi identity. Since $\vert V\vert=1$, $\ip{\nabla_W V}{V}=0$ so there are real functions $\alpha,\mu,\nu$ depending on $W$ such that 
\[\nabla_WV=\alpha IV+\mu JV+\nu KV.\]
A quick computation reveals
\[\nabla^2_W V=-(\alpha^2+\mu^2+\nu^2)V +(\nabla_W\alpha)IV +(\nabla_W \mu)JV+(\nabla_W \nu)KV.\]
We compute each of the above terms on the right hand side, using the shorthand $\sigma_{II}(V)$ and so on introduced above.
\[\ip{\nabla_WR(\nabla_WV,IV)IV}{V}=\mu\ip{\nabla_WR(V,KV)IV}{V}-\nu \ip{\nabla_WR(V,JV)IV}{V},\]
\[\ip{R(\nabla_W^2V,IV)IV}{V}=-(\alpha^2+\mu^2+\nu^2)\sigma_{II}(V)+(\nabla_W \mu)\sigma_{IK}(V)-(\nabla_W \nu) \sigma_{IJ}(V),\]
\[\ip{R(\nabla_WV,I\nabla_WV)IV}{V}=(\alpha^2-\mu^2-\nu^2)\sigma_{II}(V)+2\alpha\mu \sigma_{IJ}(V)+2\alpha\nu \sigma_{IK}(V),\]
\[\ip{R(\nabla_WV,IV)IV}{\nabla_WV}=\mu^2 \sigma_{KK}(V)+\nu^2\sigma_{JJ}(V)-2\mu\nu \sigma_{JK}(V).\]
So
\begin{align*}
\nabla^2_W \ip{R(V,IV)IV}{V}&= \ip{\nabla_W^2R(V,IV)IV}{V} \\
&+8(\mu\ip{\nabla_WR(V,KV)IV}{V}-\nu \ip{\nabla_WR(V,JV)IV}{V}) \\
&+8\left(\mu^2\sigma_{KK}(V)+\nu^2 \sigma_{JJ}(V)-(\mu^2+\nu^2)\sigma_{II}(V)\right)\\
&+4\left( (\nabla_W\mu)+2\alpha\nu\right)\sigma_{IK}(V)\\
&+4\left( -(\nabla_W\nu)+2\alpha\mu\right)\sigma_{IJ}(V)\\
&-4\mu\nu \sigma_{JK}(V).
\end{align*}
Evaluating this at a critical point in the fixed point set,  $p\in M$, simplifies the expression vastly. By Theorem \ref{Thm:IsomThm}, $\sigma_{IJ}=\sigma_{JK}=\sigma_{IK}=0$ and $\sigma_{JJ}=\sigma_{KK}=-\frac{1}{2}\sigma_{II}$. Furthermore, being a critical point implies $\nabla_W \sigma_{II}=0$ for all vector fields $W$, hence
\begin{align*}
0&=\nabla_W \sigma_{II}=\ip{\nabla_WR(V,IV)IV}{V}+4\ip{R(\mu JV+\nu KV,IV)IV}{V}\\
&=\ip{\nabla_WR(V,IV)IV}{V}+4\mu \sigma_{IK}-4\nu \sigma_{IJ} =\ip{\nabla_WR(V,IV)IV}{V}.
\end{align*}
This will imply that also $\ip{\nabla_WR(V,KV)IV}{V}=0=\ip{\nabla_WR(V,JV)IV}{V}$, as we now demonstrate case by case. When $W=V$ or $IV$, the result follows by the same argument using isometries as in the proof of Theorem \ref{Thm:IsomThm}. The key being that the expression contains an odd number of $JV$- and  $KV$-factors. When $W=JV$, we use the second Bianchi identity to say
\[\ip{\nabla_{JV} R(V,KV)IV}{V}=-\ip{\nabla_{V} R(KV,JV)IV}{V}-\ip{\nabla_{KV} R(JV,V)IV}{V}.\]
On the last term, we apply the isometry sending $JV\mapsto \pm KV$, $KV\mapsto \mp JV$ to argue
\[\ip{\nabla_{JV} R(V,KV)IV}{V}=-\frac{1}{2}\ip{\nabla_{V} R(KV,JV)IV}{V}=-\frac{1}{2} \ip{\nabla_{V} R(V,IV)IV}{V}.\]
The right hand side vanishes at a critical point. When $W=KV$, we use the Bianchi identity to say
\[\ip{\nabla_{KV} R(V,KV)IV}{V}=-\ip{\nabla_{IV} R(V,KV)V}{KV}-\ip{\nabla_{V} R(V,KV)KV}{IV}.\]
Writing $\nabla_V V=\alpha IV$ and $\nabla_{IV} V=\beta IV$ for some functions $\alpha,\beta \colon M\to \R$, we find
\begin{align*}
-\ip{\nabla_{IV} R(V,KV)V}{KV}&=\nabla_{IV} \sigma_{KK}-4\ip{R(\nabla_{IV}V,KV)KV}{V}\\
&=-\frac{1}{2} \nabla_{IV} \sigma_{II} -4\beta \sigma_{JK}.
\end{align*}
The last term vanishes on $M$, and the first term is $0$ at a critical point. This shows that $\ip{\nabla_WR(V,KV)IV}{V}=0$ at a critical point $p\in M$. Using the isometry mapping $JV$ to $\pm KV$ shows that $\ip{\nabla_WR(V,JV)IV}{V}=0$ as well.

All in all, this shows
\[\nabla_W^2 \sigma_{II}= \ip{\nabla_W^2R(V,IV)IV}{V}-12(\mu^2+\nu^2)\sigma_{II}\]
at any critical point $p\in M$. We recall $\mu=\ip{JV}{\nabla_W V}$, $\nu=\ip{KV}{\nabla_W V}$. Using the order 4 isometry shows
\[0=\ip{JV}{\nabla_V V}=\ip{JV}{\nabla_{IV} V},\]
\[\ip{JV}{\nabla_{JV} V}=\ip{KV}{\nabla_{KV} V},\]
and
\[\ip{KV}{\nabla_{JV} V}=-\ip{JV}{\nabla_{KV} V}.\]
We therefore find
\[\Delta \sigma_{II}=\ip{\Delta R(V,IV)IV}{V}-24 \vert \nabla_{JV} V\vert^2 \sigma_{II}.\]
Using Proposition \ref{Prop:LaplaceRiemann} from Appendix \ref{App:Riemann} for $\Delta R$, we finally find
\[\Delta \sigma_{II}=-6\sigma_{II}^2 -24\vert \nabla_{JV} V\vert^2 \sigma_{II}\]
at a critical point $p\in M$.

\end{proof}

\begin{Rem}
A key fact used is that the Laplacian acting on the Riemann tensor gives something proportional to the square of the Riemann tensor. \eqref{eq:LaplacianGen} is the precise statement. Something similar is true on an arbitrary Ricci-flat K\"{a}hler manifold without assuming it to be hyperk\"{a}hler. The result is (see \cite[Eq. 3.108]{KahlerRicci})
\[-\frac{1}{2}\Delta R_{\mu\overline{\nu}\alpha\overline{\beta}}=
R_{\mu\overline{\nu}\lambda\overline{\sigma}}R^{\overline{\sigma}\lambda}{}_{\alpha\overline{\beta}}
+R_{\mu\overline{\sigma}\lambda \overline{\beta}}R^{\overline{\sigma}}{}_{\overline{\nu}\alpha}{}^\lambda 
- R_{\mu\overline{\lambda}\alpha\overline{\sigma}}R^{\overline{\lambda}}{}_{\overline{\nu}}{}^{\overline{\sigma}}{}_{\overline{\beta}}.\]
The second thing used in the above proof is that the order 4 symmetry restricts the degrees of freedom of the Riemann tensor, allowing us the get a simple expression for the Laplacian.
\end{Rem}

\section{A Special K3 Surface}
\label{Section:SpecialK3}
From now on we will specialize to a special Kummer K3 surface. The precise assumptions and notation is as follows.
\begin{Not}
 Let $\Lambda \coloneqq \Z\{1,i\}\subset \C$ and let $(X,g)$ denote the Kummer K3 surface associated to the torus with lattice $\Gamma\coloneqq \Lambda \oplus \Lambda$. We will also assume that all components of the exceptional divisor have the same volume, $a_i=a_j$ for $1\leq i,j\leq 16$.

We write $Q\colon \C^2\to Y$ for the quotient map (quotienting with respect to both $\Gamma$ and $\mu_2$) and $\pi\colon X\to Y$ for the blow-down map, with inverse $\pi^{-1}\colon Y\setminus Sing(Y)\to X\setminus E$.  

For any suitable affine map $F^{\C}\colon \C^2\to \C^2$, $F^{\C}(\mathbf{z})=B\mathbf{z}+\mathbf{b}$ we denote the induced map (using Proposition \ref{Prop:KummerIsom}) by $F\colon X\to X$. These maps will all have the property that they map $E$ to $E$, so give rise to isometries (using the same name for the restrictions) $F\colon X\setminus E\to X\setminus E$.

These restricted maps satisfy $\pi\circ F \circ \pi^{-1}\circ Q=Q\circ F^{\C}$, a fact which will be used implicitly.

For $M^{\C}\subset \C^2$ with $Q(M^{\C})\subset Y\setminus Sing(Y)$, we write $M\coloneqq \pi^{-1}\circ Q(M^{\C})\subset X$.

\end{Not}
We start by describing the isometries induced by affine maps.
\begin{Prop}
The holomorphic isometries of $(X,g)$ are induced by the affine maps of the form
\[F^{\C}(\mathbf{z})=B\mathbf{z} +\mathbf{b}\]
where 
\[B\in U(2)\cap GL(2,\Z[i])=\left\{\begin{pmatrix} \alpha & 0 \\ 0& \beta\end{pmatrix}, \begin{pmatrix} 0& \alpha\\ \beta& 0\end{pmatrix}\, \Big\vert \, \alpha,\beta\in \mu_4\right\}\]
with $\mu_4\coloneqq \{\pm 1,\pm i\}$, and $\mathbf{b}\in \left(\frac{1}{2}\Lambda\right)^2$. Any anti-holomorphic isometry is induced by one of the above form composed with $\tau^{\C}$ for $\tau^{\C}(\mathbf{z})\coloneqq \overline{\mathbf{z}}$. There are 512 distinct isometries. All of them are isometries of $(X,\tilde{g})$ as well. The maximum order of these isometries is 8.
\end{Prop}

These isometries play a central role in the arguments to come. We will single out some of them.

\begin{Not}
We consider the following map $\C^2\to \C^2$
\begin{equation}
f^{\C}\begin{pmatrix} z\\ w\end{pmatrix}\coloneqq \begin{pmatrix} z\\ iw+\frac{1}{2}\end{pmatrix},
\end{equation}
along with the induced map $f\colon X\to X$. This map is an isometry of $(X,g)$ of order 4.

We also consider the following subset
\begin{equation}
M^{\C}\coloneqq \left\{\begin{pmatrix} z\\ \frac{1+i}{4} \end{pmatrix} \Big \vert z\in \C\right\}
\end{equation}
The corresponding set in $X$ is denoted by $M=\pi^{-1} \circ Q(M^{\C})$. Note that $Q(M^{\C})$  lands in $Y\setminus Sing(Y)$ due to the shift by an element not in $\frac{1}{2}\Gamma$, so $\pi^{-1}$  is well-defined.

\end{Not}

\begin{Thm}
\label{Thm:StableGeod}
The submanifold $M\subset X$ described above is a totally geodesic torus. There are points $p\in M$ such that the Riemann tensor of $(X,\tilde{g})$ vanishes at $p$.
\end{Thm}
\begin{proof}
The image $Q(M^{\C})\subset Y$ lands in the smooth part due to the shift by $\frac{1+i}{4}\notin \frac{1}{2}\Gamma$. Clearly $Q(M^{\C})$ is a torus, and $\pi$ is biholomorphic when restricted, $\pi\colon X\setminus E \xrightarrow{\cong} Y\setminus Sing(Y)$, hence $M\subset X$ is a torus.

Furthermore, $M$ is fixed by the order $4$ isometry $f$. So by Theorem \ref{Thm:IsomThm}, the Riemann tensor of $(X,\tilde{g})$ at $p\in M$ is uniquely determined by the Gauss curvature of $M$ at $p$. But any torus has points of vanishing Gauss curvature due to the Gauss-Bonnet theorem, hence there are points where the curvature of $(X,\tilde{g})$ vanish. 
\end{proof}

The next proposition indicates that the torus in Theorem \ref{Thm:StableGeod} could be flat. The proposition contains a non-trivial assumption, however. Assume $M\subset X$ is a torus which is fixed by an order 4 isometry $f$. Let $V$ be a unit vector field on $M$. By an extension of $V$ to a neighbourhood $M\subset U\subset X$ we mean a unit vector field $V\colon U\to TX$ such that

\begin{itemize}
\item $V$ restricts to a tangent vector field on $M$, $V_{\vert M} \colon M\to TM$;
\item $\nabla_W V$  restricts to a tangent vector field on $M$, $(\nabla_W V)_{\vert M} \colon M\to TM$, for all vector fields $W\colon U\to TX$.
\end{itemize} 
A concrete example would be that if $(z,w)$ are local coordinates with $M$ locally being given by $\{w=0\}$, and $V=h \frac{\partial}{\partial z}$ with $h\coloneqq \frac{1}{\vert\partial_z\vert_{\tilde{g}}}$. Then \eqref{eq:Jaction1} along with $\nabla_{\partial_{\oz}} \partial_z =0=\nabla_{\partial_{\ow}} \partial_z$ say $\nabla_{JV} V \parallel V$, and similarly for $\nabla_{KV} V$. 
\begin{Thm}
\label{Thm:FlatTori}
Let $M \subset (X,\tilde{g})$ be a torus which is fixed by an order 4 holomorphic isometry. Let $V$ be a unit vector tangent field of $M$, extended to a neighbourhood $U$ of $M$. If the function $\sigma_{II}(V)\colon U\to \R$ has its minima on $M$, then $M$ is flat. 
\end{Thm}
\begin{proof}
Proposition \ref{Prop:RiemannCrit} gives us an expression for $\Delta \sigma_{II}$ at critical points. We will argue that the last term drops out. 
Using the order 4 isometry $f$, we see that $\nabla_{JV} V$ has to be proportional to a combination of $JV$ and $KV$ when restricted to $M$. Hence  $\nabla_{JV}V=0$ on $M$. Proposition \ref{Prop:RiemannCrit} therefore says
\[\Delta \sigma_{II}=-6\sigma_{II}^2 \]
for any critical point on $M$. Hence a minimum is possible if and only if $\sigma_{II}=0$ at the minimum. But $\int_M \sigma_{II}\, d\Vol_{\tilde{g}_{\vert M}}=0$ by the Gauss-Bonnet theorem, and $\sigma_{II}=0$ everywhere on $M$ as a consequence. 
\end{proof}
\begin{Rem}
Any critical point for $(\sigma_{II})_{\vert M}$ is also a critical point for $\sigma_{II}$ due to the order 4 isometry. Indeed, $\nabla_{JV} \sigma_{II}=-\nabla_{JV} \sigma_{II}$ and similarly for $\nabla_{KV}$. So $\nabla \sigma_{II}=0$ if and only if $\nabla_{V} \sigma_{II}=\nabla_{IV} \sigma_{II}=0$. Hence there are critical points of $\sigma_{II}$ on $M$. What is not clear (hence the assumption in the above result) is that these critical points are minima.
\end{Rem}

We end by pointing out that the lattice $\Gamma=\Lambda\oplus \Lambda $ with $\Lambda=\Z\{1,i\}$ is not the only possible choice leading to a result like Theorem \ref{Thm:StableGeod}. Another possibility is to choose $\zeta\coloneqq \exp\left(\frac{2\pi i}{3}\right)$, $\Lambda\coloneqq \Z\{1,\zeta\}$ and $\Gamma=\Lambda\oplus \Lambda$. The affine map
\[f^{\C}\begin{pmatrix} z\\ w\end{pmatrix}=\begin{pmatrix}
z \\ \zeta w + \frac{1+\zeta}{2}
\end{pmatrix} \]
induces an order 3 holomorphic isometry $f\colon X\to X$. The set 
\[M^{\C}\coloneqq \left\{ \begin{pmatrix} z\\ \frac{1+2\zeta}{6} \end{pmatrix}\,\big\vert \, z\in \C\right\}\]
maps to a torus $M\subset X$ which is a connected component of the fixed point set of $f$. Theorem \ref{Thm:StableGeod} therefore applies to this $M$.

\section{Proof of Kobayashi's estimates}
\label{Section:Estimates}
Here we will go through the proof of Kobayashi's estimates. Most of the key arguments are due to \cite{Kob90}, but we simplify a couple of steps, correct a minor mistake in the $C^2-$estimates, correct the arguments for the H\"{o}lder estimates.

Our notation is as before; $(X,g)$ denotes a Kummer K3 surface with patchwork metric depending on 16 parameters $a_i$ and $\vert a\vert^2=\sum_i a_i^2.$ The constant $A$ is defined in \eqref{eq:ADef} and takes the value given in \eqref{eq:AValue}. We introduce the function $\psi\colon X\to \R$ via 
\begin{equation}
e^{\psi}=\frac{\eta\wedge\overline{\eta}}{\omega^2}.
\label{eq:psidef}
\end{equation}
This can either be interpreted as the Radon-Nikodym derivative, or simply the proportionality function which must exist between the two top forms $\eta\wedge\overline{\eta}$ and $\omega^2$. In holomorphic Darboux coordinates, i.e. where $\eta=\sqrt{2}dz^1 \wedge dz^2$, we have
\[\psi=-\ln \det(g),\]
and \eqref{eq:psidef} is a way of making this function globally well-defined. As we prove in Appendix \ref{App:Param}, \eqref{eq:NeckDetg}, we have
\begin{equation}
\norm{\psi}_{C^0(X)}\leq C\vert a\vert^2.
\label{eq:psibound}
\end{equation}
for some constant $C>0$. The argument is that $\psi=0$ outside of the neck regions, and in the necks one can see that the Euclidean and Eguchi-Hanson metrics differ by terms of order $a^2$.
 With this notation out of the way, we may write the Monge-Amp\`{e}re equation as 
\begin{equation} \tilde{\omega}^2\coloneqq (\omega+i\partial \overline{\partial} \phi)^2=A\eta \wedge\overline{\eta}=A e^{\psi} \omega^2.
\label{eq:NewMA}
\end{equation}

\subsection{$C^0$-estimates}
The $C^0$ estimate of Kobayashi is as follows.
\begin{Prop}[{\cite{Kob90}}]
Assume $\phi$ is the solution to \eqref{eq:NewMA} subject to the normalisation
\[\int_X \phi\, \omega^2=0.\]
Then there is a constant $C>0$ such that for all values of $\vert a\vert$ small enough, we have
\begin{equation}
\norm{\phi}_{C^0(X)}\leq C\vert a\vert^2
\label{eq:C0Corr}
\end{equation}
\end{Prop}

\begin{proof}
From the Monge-Amp\`{e}re equation \eqref{eq:NewMA}, we have that

\begin{equation}
(1-Ae^{\psi})\omega^2=\omega^2-\tilde{\omega}^2=-i\partial \overline{\partial}\phi\wedge(\tilde{\omega}+\omega),
\label{eq:intermediate}
\end{equation}
Multiply both sides of \eqref{eq:intermediate} by $\phi \vert\phi\vert^{2(p-1)}$, for $p\geq 1$, and integrate. Integrating the right hand side by parts leads to
\begin{align}
 \int_X(1-Ae^{\psi})\phi \vert\phi\vert^{2(p-1)}\omega^2&=-\int_X i \phi \vert\phi\vert^{2(p-1)}\partial \overline{\partial}\phi\wedge(\tilde{\omega}+\omega) \notag \\
&=(2p-1)\int_X i\vert \phi\vert^{2(p-1)}\partial \phi\wedge \overline{\partial}\phi\wedge(\tilde{\omega}+\omega) \notag \\
&=\frac{(2p-1)}{p^2}\int_X i\partial \vert \phi\vert^{p}\wedge \overline{\partial}\vert \phi\vert^{p}\wedge(\tilde{\omega}+\omega). 
\label{eq:IntegratedDifference2}
\end{align}
Here we have used that $\partial \omega=0=\partial \tilde{\omega}$ since they are K\"{a}hler forms. After we argue that the 2-form $i\partial \vert \phi\vert^{p}\wedge \overline{\partial}\vert \phi\vert^{p}$ is semi-positive\footnote{Recall that a real $(1,1)$ form $\eta$ is called semi-positive (or simply positive, according to some authors) if $\eta=i\eta_{\mu\bar{\nu}}dz^\mu \wedge d\bar{z}^\nu$ in local coordinates, where the matrix $\eta_{\mu\bar{\nu}}$ is positive semi-definite.}, one gets a lower bound on \eqref{eq:IntegratedDifference2} by dropping the first integral (the one with $\tilde{\omega}$);
\begin{equation}
\int_X(1-Ae^{\psi})\phi \vert\phi\vert^{2(p-1)}\omega^2\geq \frac{4(p-1)}{p^2}\int_X i\partial \vert \phi\vert^{p}\wedge \overline{\partial}\vert \phi\vert^{p}\wedge \omega.
\label{eq:stepping}
\end{equation}
To see that the $(1,1)$-form $i\partial \vert \phi\vert^{p}\wedge \overline{\partial}\vert \phi\vert^{p}$ is a semi-positive form, one can observe that the matrix 
\begin{equation}
B_{\mu\overline{\nu}}\coloneqq \left(\frac{\partial \vert \phi\vert^{p}}{\partial z^{\mu}}\right)\left(\frac{\partial \vert \phi\vert^{p}}{\partial \overline{z}^{\nu}}\right)
\notag
\end{equation}
has eigenvalue 0 with multiplicity $n-1$, corresponding to the $n-1$ eigenvectors orthogonal to  $\left(\frac{\partial \vert \phi\vert^{p}}{\partial \overline{z}^{\nu}}\right)$, and eigenvalue $\lambda=\text{Tr}(B)=\left\vert \frac{\partial \vert \phi\vert^{p}}{\partial z}\right\vert^2$ 
with multiplicity 1. These eigenvalues are non-negative everywhere, hence we have a semi-positive $(1,1)$-form.

To further estimate \eqref{eq:stepping}, consider the two sides separately. The left hand side can be roughly estimated by
\begin{equation}
\int_X(1-Ae^{\psi})\phi \vert\phi\vert^{2(p-1)}\omega^2\leq \int_X\left \vert (1-Ae^{\psi})\right\vert \vert\phi\vert^{2p-1}\omega^2.
\label{eq:LHSSTepping}
\end{equation}
For all small enough values of $\vert a\vert$, \eqref{eq:psibound} and \eqref{eq:AValue} tells us that one can find a constant $C>0$, independent of $a$, such that $\left \vert (1-Ae^{\psi})\right\vert\leq C\vert a\vert^2$. We therefore get an estimate on \eqref{eq:LHSSTepping}
\begin{equation}
\int_X(1-Ae^{\psi})\phi \vert\phi\vert^{2(p-1)}\omega^2 \leq C\vert a\vert^2 \int_X \vert \phi\vert^{p-1} \omega^2.
\label{eq:LHS}
\end{equation}

The right-hand side of \eqref{eq:stepping} can be massaged a bit. Write $\zeta\coloneqq \vert \phi\vert^{p}$ temporarily. Then $\partial \zeta$ is an exact $(1,0)$-form, and by some linear algebra, we have the equality.
\begin{equation}
i\partial\zeta\wedge \overline{\partial}\zeta \wedge\omega=\vert d\zeta\vert_{g}^2 \omega^2
\label{eq:Kahler}
\end{equation}

We will write $\vert d\zeta\vert^2$ and not mention that it depends on the metric $g$. Inserting \eqref{eq:Kahler} into the right hand side of \eqref{eq:stepping} and \eqref{eq:LHS} into the left hand side yields

\begin{equation}
\int_X \vert d\vert \phi\vert^{p}\vert^2 \omega^2\leq C\vert a\vert ^2 p\int_X  \vert \phi\vert^{2p-1}\omega^2.
\label{eq:dphi_and_g}
\end{equation}

The constant $C$ on the right hand side of \eqref{eq:dphi_and_g} depends on neither $p$ nor $a$.

\textbf{Notation:} For the rest of the proof, we will use a shorthand for $L^p$-norms, $1\leq p<\infty$, namely
\begin{equation}
\norm{f}_p^p\coloneqq \int_X f^p \omega^2. 
\notag
\end{equation}

Setting $p=1$, we can apply the Poincar\'{e} inequality \eqref{eq:Poincare} to the function $\phi$   on the left hand side of \eqref{eq:dphi_and_g}, resulting in\footnote{There are three things to remark here. Firstly, $\int_X \phi \,\omega^2=0$ by the normalization requirement. Secondly, $\vert d\vert \phi\vert \vert=\vert d\phi\vert$, so one can indeed apply the Poincar\'{e} inequality. We refer the reader to the Proposition \ref{Prop:PoincareIndep} in Appendix \ref{App:Param} for a proof that the constant in the Poincar\'{e} inequality can estimated independently of $\vert a\vert$.}
\begin{equation}
C \norm{\phi}_2^2\leq \int_X \left\vert d\vert \phi\vert \,\right\vert ^2\omega^2.
\label{eq:LHS2}
\end{equation}
To the right hand side of \eqref{eq:dphi_and_g}, we apply the H\"{o}lder inequality 
\begin{equation}
\int_X  \vert \phi\vert\omega^2\leq \left(\Vol_g(X)\right)^{\frac{1}{2}}\left(\int_X \vert \phi\vert^{2} \omega^2\right)^{\frac{1}{2}}\leq C\norm{\phi}_2 ,
\label{eq:RHS2}
\end{equation}
where the proof of Proposition \ref{Prop:PoincareIndep} has been implicitly invoked to say that $\Vol_g(X)$ can be estimated independent of $a$ for all small enough values of $\vert a\vert$.

 Combining the separate approximations (\eqref{eq:LHS2} and \eqref{eq:RHS2}) of the left and right hand sides of \eqref{eq:dphi_and_g} gives our $L^2$-estimate of $\phi$,
 \begin{equation}
 \norm{\phi}_2\leq C_0 \vert a\vert ^2.
 \label{eq:L2_est}
 \end{equation}
The $L^2$-bound \eqref{eq:L2_est} is going to be the start of a Moser iteration process, which will end up proving \eqref{eq:C0Corr}. 

Return to \eqref{eq:dphi_and_g} with arbitrary $p\geq 2$. Apply Li's Sobolev inequality \eqref{eq:LiSob} with $f=\vert \phi\vert ^{p}$ to the left hand side. On the right hand side, use the H\"{o}lder inequality with $q=\frac{2p}{2p-1}$ to derive the bound
\begin{align}
\norm{\phi}_{4p}^{2p}&\leq C \norm{d\vert \phi\vert ^{2p}}_2^2+C'\norm{\phi^{2p}}_2^2 \notag \\
&\leq Cp \vert a\vert^2 \norm{\phi}_{2p}^{2p-1}+C'\norm{\phi}_{2p}^{2p},\notag
\end{align}
hence
\begin{equation}
\norm{\phi}_{4p}^{2p}\leq C\left( 2p\vert a\vert^2 +\norm{\phi}_{2p}\right)\norm{\phi}_{2p}^{2p-1},
\label{eq:Estim1}
\end{equation}
where  $C>0$ is a constant (which depends on neither $p$ nor $a$).

Let now $p_n\coloneqq 2^{n+1}$ for $n\geq 0$. Define $C_n$ inductively by 
\begin{equation}
C_{n+1}^{p_n} \coloneqq C C_n^{p_n} \left(1+\frac{p_n}{C_n}\right)
\label{eq:C_nDef}
\end{equation}

where $C$ is the constant in \eqref{eq:Estim1} and $C_0$ is the constant in \eqref{eq:L2_est}.
For $n=0$, we have shown in \eqref{eq:L2_est} that 
\begin{equation}
\norm{\phi}_{p_n}\leq C_n \vert a\vert^2.
\label{eq:InductionAssumption}
\end{equation} 
That this inequality holds for all $n$ follows inductively from \eqref{eq:Estim1}.


We next have to show that iterating \eqref{eq:InductionAssumption} gives a uniform bound for all $n$. Let $c\coloneqq \inf_{n\geq 0} C_n$. Then $c>0$ since if not, then $\phi=0$ from \eqref{eq:InductionAssumption}, which we know is not the case. With this, \eqref{eq:C_nDef} can be estimated as
\[C_{n+1}\leq C_n \cdot \left(C+\frac{C}{c}p_n\right)^{1/p_n}\eqqcolon C_n\cdot \alpha_n,\]
hence
\[C_{n+1}\leq \left(\prod_{k=0}^n \alpha_k\right) C_0.\]
To see that the product $\prod\limits_{k=0}^n \alpha_k$ is uniformly bounded, write
\[\log\prod_{k=0}^n \alpha_k=\sum_{k=0}^n \frac{1}{2^{k+1}}\log\left(C+\frac{C}{c}2^{k+1}\right).\]
The right hand side is convergent as $n\to \infty$ by the ratio test. Hence
\[\norm{\phi}_{\infty}\leq \left(\prod_{k=0}^\infty \alpha_k\right)C_0 \vert a\vert^2.\]

%

\end{proof}

\subsection{$C^2$-estimates}
The $C^2$-estimates follow from the Monge-Amp\`{e}re equation as soon as we have estimates on $\Delta \phi$. To see this, we recall the function
\[\exp(\psi)=\frac{\eta\wedge \overline{\eta}}{\omega^2}.\]
In holomorphic Darboux coordinates, the Monge-Amp\`{e}re equation \eqref{eq:NewMA} reads
\[\det\left( g + \nabla^2\phi\right)=A,\]
where $A$ is defined by  \eqref{eq:ADef}. The determinant of the complex $2\time 2$-matrix we write as
\[\det\left( g + \nabla^2\phi\right)=\det(g)\det\left( 1 + g^{-1} \nabla^2\phi\right)=\exp(-\psi)\left(1+\tr(g^{-1} \nabla^2 \phi) + \det(g^{-1} \nabla^2 \phi)\right).\]
Now $\tr(g^{-1} \nabla^2 \phi)=\Delta \phi$ and $\tr(g^{-1} \nabla^2\phi g^{-1} \nabla^2 \phi)= \vert \nabla^2 \phi\vert^2_g$ per definition, and 
\begin{align*}
\det(g^{-1} \nabla^2 \phi)&=\frac{1}{2} \left(\tr(g^{-1} \nabla^2 \phi)^2 -\tr(g^{-1} \nabla^2\phi g^{-1} \nabla^2\phi)\right)\\
&=\frac{1}{2}\left((\Delta \phi)^2 -\vert \nabla^2\phi\vert_g^2\right).
\end{align*}
With this, we may write the Monge-Amp\`{e}re equation as 
\begin{equation}
2(A \exp(\psi) -1)=2\Delta \phi + (\Delta \phi)^2 -\vert \nabla^2\phi\vert_g^2.
\label{eq:NormMA}
\end{equation}
We note how \eqref{eq:NormMA} is independent of the choice of coordinates.

This tells us  that a bound on $\Delta \phi$ directly translates into a bound on $\nabla^2 \phi$.  So we set about bounding $\Delta \phi$ as in \cite{Yau78} and \cite{Kob90}. Let $r_a=\frac{\max_i a_i}{\min_i a_i}$. Then we will prove there is a constant $C>0$ independent of $\vert a\vert$ such that 
\begin{equation}
-C \vert a\vert^2 \leq \Delta \phi\leq C r_a \vert a\vert
\label{eq:C2Est}
\end{equation}

holds for any point in $X$.
 The trick to be employed in proving \eqref{eq:C2Est} is essentially a maximum principle. We will consider the function $F(x)\coloneqq \exp(-C\phi(x))\Tr_{g}(\tilde{g})(x)= \exp(-C\phi(x))(2+\Delta\phi)$ for some positive constant $C$ (to be determined later). The function $F$ has a maximum due to the compactness of $X$, and at a maximum we have $\tilde{\Delta}F\leq 0$, where we have introduced $\tilde{\Delta}F=\tr(\tilde{g}^{-1} \nabla^2 F)=\tilde{g}^{\on \mu} \partial_\mu \partial_{\on} F$. Below there will be an extensive computation deriving a lower bound on $\tilde{\Delta}F$. This lower bound at a single point will essentially establish (a stronger version of) \eqref{eq:C2Est} at a single point, which then gets translated into a proof of \eqref{eq:C2Est} at an arbitrary point. 

The proof is long, and is subdivided into several steps. In Appendix \ref{App:Yau}, we go through some preliminary calculations, which culminate in a proof of Proposition \ref{Prop:Yau222}, which is \cite[Equation 2.22]{Yau78}. 

\begin{Prop}{\cite[Equation 2.22]{Yau78}}
\label{Prop:Yau222}
Choose holomorphic normal coordinates at $p$ and diagonalize\footnote{See \cite[Prop. 3.1.1]{KahlerRicci} for instance for a proof that these coordinates exist.} $\nabla^2 \phi(p)$, meaning $g_{\mu\bar{\nu}}(p)=\delta_{\mu\bar{\nu}}$, $g_{\mu\bar{\nu},\alpha}(p)=0$, and $\phi_{\mu\bar{\nu}}(p)=\delta_{\mu\bar{\nu}}\phi_{\mu\bar{\mu}}(p)$. Then the following inequality holds for any positive real number $C$ at the single point $p$.
\begin{align}
&\exp(C\phi)\tilde{\Delta}\Big(\exp(-C\phi)\Tr_g(\tilde{g})\Big)\notag \\&\geq \Delta \psi-4R_{1\bar{1}2\bar{2}}  -2C\Tr_g(\tilde{g})+\left(C+R_{1\bar{1}2\bar{2}}\right)    \frac{\Tr_g(\tilde{g})^2}{\det(\tilde{g})}.
\label{eq:Yau222}
\end{align}
\end{Prop}

\begin{Lem}
Let $K\colon X\to \R_{\geq 0}$, $K\coloneqq \vert Riem\vert^2_g$ denote the Kretchsmann scalar for the patchwork metric $g$. 
Introduce 
\begin{equation}
R_a\coloneqq \norm{\sqrt{K}}_{L^\infty(X,g)}.
\end{equation}
Then there are constants $C_1,C_2>0$ (independent of $a$) such that for all small enough values of $\vert a\vert$, we have the bounds
\begin{equation}
C_1 \frac{r_a}{\vert a\vert}\leq R_a\leq C_2 \frac{r_a}{\vert a \vert}.
\label{eq:Rabound}
\end{equation}
 In particular, in the special coordinates of Proposition \ref{Prop:Yau222}, we have
\begin{equation}
R_{1\bar{1}2\bar{2}}(p)\leq \sqrt{K}(p)\leq C\frac{r_a}{\vert a\vert}
\end{equation}
for an arbitrary point $p\in X$, and
\begin{equation}
R_{1\bar{1}2\bar{2}}(p)\leq \sqrt{K}(p)\leq C\vert a\vert^2
\label{eq:Sectionalatpointinneck}
\end{equation}
if $p$ is in a neck region.
\end{Lem}
\begin{proof}
In the Euclidean region, we have $K=0$. In a neck region $N_i$, \eqref{eq:NeckPotential} implies there is a constant $C>0$ such that $K\leq C  a_i^4$ for small enough values of $a_i$. In an Eguchi-Hanson patch $U_i$, \eqref{eq:EHCurvature} says $K(z)=\frac{24a_i^4}{(a_i^2+u^2)^3}$ with $u=\vert z\vert^2_{\C^2}$. All in all, we find 
\[R_a^2=\max_{1\leq i\leq 16} \frac{24}{a_i^2},\]
from which it follows that\footnote{The argument written out is this. We have $\max_i \frac{1}{a_i}=\frac{r_a}{\max_i a_i}$. Then one uses the comparison between the max-norm and Euclidean norm; $\max_i a_i\leq \vert a\vert\leq 16\max_i a_i$.} 
\[2\sqrt{6} \frac{r_a}{\vert a\vert} \leq R_a \leq 32\sqrt{6} \frac{r_a}{\vert a\vert}.\]
\end{proof}

\begin{Lem}
Let $x_m\in X$ denote a maximum of the function $F(x)=\exp(-2R_a\phi(x))\Tr_g(\tilde{g})(x)$. Then there is a constant $C>0$ such that for all small enough values of $\vert a\vert$, we have
\begin{equation}
\Delta \phi(x_m) \leq C\vert a\vert^2.
\label{eq:FirstStep}
\end{equation}
\end{Lem}
\begin{proof}

At the single point $x_m$ we introduce holomorphic normal coordinates as before such that $g_{\mu\bar{\nu}}(x_m)=\delta_{\mu\bar{\nu}}$, $g_{\mu\bar{\nu},\alpha}(x_m)=0$, and $\phi_{\mu\bar{\nu}}(x_m)=\delta_{\mu\bar{\nu}}\phi_{\mu\bar{\mu}}(x_m)$. In these coordinates, we set 
\begin{equation}
k(x_m)\coloneqq  R_{1\bar{1}2\bar{2}}(x_m)/R_a.
\end{equation}
Note that per definition, $\vert k(x_m)\vert \leq 1$.

At the maximum of $\exp(-2R_a\phi)\Tr_g(\tilde{g})$, the left hand side of \eqref{eq:Yau222} with $C=2R_a$ has to be non-positive. With our choice of notation, we may write this as
\begin{equation}
0\geq \frac{\Delta \psi(x_m)}{R_a}-4k(x_m)-4\Tr_g(\tilde{g})+(2+k(x_m))\frac{\Tr_g(\tilde{g})^2}{\det(\tilde{g})}.
\notag
\end{equation} 
%
%
Complete a square in $\Tr_g(\tilde{g})$  to arrive at
\begin{align}
&\frac{4\det(\tilde{g})^2}{(2+k(x_m))^2}-\frac{\det(\tilde{g})}{R_a(2+k(x_m))}\left(\Delta \psi(x_m)-4 k(x_m) R_a\right)\notag \\ &\geq \left(\Tr_g(\tilde{g})-\frac{2\det(\tilde{g})}{2+k(x_m)}\right)^2.
\label{eq:CompletingSquares}
\end{align}
There are now two possible cases. Either $x_m$ lies inside a neck region $N_i$ or it lies outside of all the neck regions. 
\par 
\noindent\textbf{Case 1 -- $x_m$ lies outside of the neck regions:}
\par
Outside of the necks, $\psi(x_m)=\Delta\psi(x_m)=0$ by \eqref{eq:psidef} by construction of $g$. Inserting this into \eqref{eq:CompletingSquares} gives
\begin{equation}
 4\frac{\det(\tilde{g})^2}{(2+k(x_m))^2}+4\frac{k(x_m)\det(\tilde{g})}{2+k(x_m)}\geq \left(\Tr_g(\tilde{g})-\frac{2\det(\tilde{g})}{2+k(x_m)}\right)^2.
 \label{eq:Case1Step1}
\end{equation}
 The Monge-Amp\`{e}re equation, \eqref{eq:NewMA} says in holomorphic normal coordinates for $g$ that $\det(\tilde{g})=Ae^{\psi}=1-\Upsilon \vert a\vert^2<1$, where $\Upsilon>0$ is some positive constant which was computed in  \eqref{eq:AValue}. Hence one may overestimate the first term on the left hand side of \eqref{eq:Case1Step1} by 
 \begin{equation}
 4\frac{\det(\tilde{g})^2}{(2+k(x_m))^2}<4\frac{\det(\tilde{g})}{(2+k(x_m))^2}.
 \end{equation}
  Inserting this back into \eqref{eq:Case1Step1} and recognizing a square allows one to conclude
 \begin{equation}
 4\det(\tilde{g})\frac{(k(x_m)+1)^2}{(k(x_m)+2)^2}\geq \left(\Tr_g(\tilde{g})-\frac{\det(\tilde{g})}{2+k(x_m)}\right)^2. 
 \notag
 \end{equation}
Taking a square root on both sides here and using $\det(\tilde{g})<1$ twice, one can conclude that
\begin{align}
&\frac{2(k(x_m)+1)}{k(x_m)+2}> \sqrt{\det(\tilde{g})}\frac{2(k(x_m)+1)}{k(x_m)+2}\notag \\ &\geq \Tr_g(\tilde{g})-\frac{2\det(\tilde{g})}{2+k(x_m)}>\Tr_g(\tilde{g})-\frac{2}{2+k(x_m)},
\notag
\end{align}
or
\begin{equation}
2=\frac{2(k(x_m)+2)}{k(x_m)+2}\geq \Tr_g(\tilde{g}).
\notag
\end{equation}
This proves that
\begin{equation}
\Delta\phi(x_m)\leq 0
\notag
\end{equation}
when $x_m$ lies outside of the neck regions.
\par 
\noindent\textbf{Case 2 -- $x_m$ lies inside the neck regions:}
\par
When the maximum $x_m$ lies in a neck region, we 
return to \eqref{eq:CompletingSquares} and complete the square on the left hand side. Suppressing the point $x_m$ from the notation, the result is
\begin{align}
&\frac{4\det(\tilde{g})^2}{(2+k)^2}-\frac{\det(\tilde{g})}{R_a(2+k)}\left(\Delta \psi-4 k R_a\right)\notag \\
&=\left(\frac{2\det(\tilde{g}}{2+k}-\frac{1}{4R_a}\left(\Delta \psi -4k R_a\right)\right)^2-\frac{1}{16R_a^2}\left(\Delta \psi -4k R_a\right)^2 \notag \\
&\leq \left(\frac{2\det(\tilde{g}}{2+k}-\frac{1}{4R_a}\left(\Delta \psi -4k R_a\right)\right)^2.
\label{eq:Case2Step0.5}
\end{align}
Hence \eqref{eq:CompletingSquares} says
\begin{equation}
\Tr_g(\tilde{g})\leq \frac{4\det(\tilde{g})}{2+k}-\frac{1}{4R_a}(\Delta \psi-4kR_a)\notag
\end{equation}
or
\begin{equation}
\Delta\phi(x_m) \leq \frac{4\det(\tilde{g})}{2+k}+k-2-\frac{\Delta\psi}{4R_a}=\frac{4(\det(\tilde{g})-1)+k^2}{2+k}-\frac{\Delta\psi}{4R_a}
\notag
\end{equation}
 Since we are in a neck region, the curvature $k$ is bounded, $\vert k\vert \leq C \frac{\vert a\vert^3}{r_a}$. This follows by \eqref{eq:Neckmetric}. In normal coordinates for $g$, the Monge-Amp\`{e}re equation reads $\det(\tilde{g})=Ae^{\psi}$, hence $\vert\det(\tilde{g})-1\vert \leq C\vert a\vert^2$ by \eqref{eq:AValue} and \eqref{eq:psibound}. From \eqref{eq:NeckDetg} it also follows that $\vert\Delta \psi\vert \leq C \vert a\vert^2$, and so
\[\Delta \phi(x_m) \leq C\vert a\vert^2\]
for all $\vert a\vert$ small enough.
This proves \eqref{eq:FirstStep}.
 \end{proof}
 
\begin{proof}[Proof of Equation \eqref{eq:C2Est}] 
To get an upper bound of $\nabla \phi$ at an arbitrary point, we can do as follows, where the first inequality is the definition of $x_m$.
\begin{align}
\Tr_g(\tilde{g})=&e^{2R_a\phi(x)}\Big( e^{-2R_a\phi(x)}\Tr_g(\tilde{g})\Big)\notag\\
&\leq e^{2R_a\phi(x)}\Big( e^{-2R_a\phi(x_m)}(\Tr_g(\tilde{g})(x_m))\Big) \notag\\
&\stackrel{\eqref{eq:FirstStep}}{\leq} e^{2R_a(\phi(x)-\phi(x_m))}( 2+C\vert a\vert^2)\notag\\
&\stackrel{\eqref{eq:C0Corr},\eqref{eq:Rabound}}{\leq} e^{Cr_a\vert a\vert}(2+C\vert a\vert^2) \notag \\ &\leq 2+\tilde{C}r_a\vert a\vert.
\notag
\end{align}
 This proves the upper bound in \eqref{eq:C2Est}.

 The lower bound in \eqref{eq:C2Est} is considerably easier. Let $\alpha,\beta$ denote the eigenvalues of $g^{-1}\tilde{g}$ in some local coordinates.  Then
\begin{equation}
\alpha+\beta=\Tr_g(\tilde{g}),
\end{equation}
and by the inequality of arithmetic and geometric means,
\begin{equation}
\frac{\Tr_g(\tilde{g})}{2} =\frac{\alpha+\beta}{2}\geq \sqrt{\alpha\beta}=\sqrt{\det(g^{-1}\tilde{g})}=\sqrt{\frac{\det(\tilde{g})}{\det(g)}}=\sqrt{Ae^\psi},
\label{eq:LowerBoundDet}
\end{equation}
where we have inserted the Monge-Amp\`{e}re equation, \eqref{eq:NewMA}, in the final step.
By  \eqref{eq:AValue} and \eqref{eq:psibound},  one can find a constant $C>0$ such that 
\begin{equation}
Ae^\psi\geq (1-C \vert a\vert^2)^2,
\notag
\end{equation} 
for sufficiently small values of $\vert a\vert$. This can be inserted into \eqref{eq:LowerBoundDet} to establish
\begin{equation}
2+\Delta\phi=\Tr_g(\tilde{g})\geq 2\sqrt{Ae^\psi}\geq 2(1-C \vert a\vert^2),
\notag
\end{equation}
which proves the lower bound in \eqref{eq:C2Est}.

\end{proof}

\begin{Cor}
\label{Cor:ComplexHessian}
There is a constant $C>0$ independent of $\vert a\vert$ such that
\begin{equation}
 \vert \nabla^2 \phi \vert^2_g \leq C  r_a\vert a\vert.
\label{eq:C2Norm}
\end{equation}
hold everywhere on $X$.
\end{Cor}
\begin{proof}
From \eqref{eq:NormMA}, \eqref{eq:AValue}, \eqref{eq:psibound}, and \eqref{eq:C2Est}, it follows that
\[\vert \nabla^2 \phi\vert^2_g \leq 2 Cr_a \vert a\vert+C^2 r_a^2 \vert a\vert^2+2(1-A\exp(\psi))\leq \tilde{C} r_a \vert a\vert.\] 
This shows
\begin{equation}
\vert \nabla^2 \phi \vert^2_g\leq C r_a \vert a\vert.
\label{eq:C2Norm1}
\end{equation}
\end{proof}

\begin{Rem}
The appearance of $r_a$ in the $C^2$-estimate is a consequence of using Yau's maximum principle, where the maximal holomorphic sectional curvature appears from a double derivative of the Monge-Amp\`{e}re equation. This factor of $r_a$ then propagates into the higher order estimates. We do not know if this is reflected in the actual behaviour of the solution, or if it is an artefact of the proof. We assume $r_a$ is uniformly bounded in $a$ (Assumption \ref{Assumption:ra}).

It would in general be interesting (and useful for studying the K\"{a}hler-Ricci flow on singular manifolds) if one can derive Yau's $C^2$-bounds without using the maximum of the sectional curvature.
\end{Rem}

%

\subsection{H\"{o}lder regularity}
From the above bound on the complex Hessian $\nabla^2 \phi$, we will follow Siu \cite{Siu87} and B\l ocki, \cite{Blocki00}, \cite{Blocki12}, to derive H\"{o}lder bounds on the real Hessian $D^2\phi$. Since the real and complex Hessians differ, the corresponding real and complex Monge-Amp\`{e}re equations differ. Hence one cannot simply apply the real theory directly, as \cite[p. 302]{Kob90} does. The methods go back to Evans \cite{Evans82}, 
\cite{Evans83}, Krylov \cite{Krylov82} and Trudinger \cite{Trudinger83}.

\begin{Prop}
\label{Prop:Holder}
There are constants $C>0$ and $0<\alpha<1$ which do not depend on $a$ such that
\[\vert \phi\vert_{C^{2,\alpha}(X,g)}\leq C.\]
holds for all values of $a$ small enough.
\end{Prop}

The strategy is the following. Locally, we may write $\tilde{g}=\nabla^2 \tilde{\phi}$ for some K\"{a}hler potential $\tilde{\phi}$. In suitable coordinates, the Monge-Amp\`{e}re equation reads $\det(\nabla^2 \tilde{\phi})=const.$, and by taking derivatives of this equation, we get an elliptic equation we can analyse using a local Harnack inequality. By combining bounds on sub- and supersolutions, we get a bound on the oscillation of $\tilde{\phi}$, which leads to the H\"{o}lder bound on $\tilde{\phi}$.  Since $\tilde{g}=g+\nabla^2\phi$, we may choose $\tilde{\phi}=\Phi+\phi$, where $\Phi$ is a K\"{a}hler potential for $g$ (e.g. \eqref{eq:GluedPot}). We will show that $\Phi$ is in $C^{2,\alpha}$ uniformly in the same local coordinates as for $\tilde{\phi}$. Hence $\phi$ will be uniformly bounded as well. We divide the proof into 7 steps.

\noindent\textbf{Step 1 -- Bounds on the patchwork metric:}
 
We first prove that we can cover $X$ by coordinate charts in such a way that the Monge-Amp\`{e}re equation becomes simple and the eigenvalues of the patchwork metric are under control away from the exceptional divisor. 
\begin{Lem}
\label{Lem:Eigenvalues}
For all $\vert a\vert$ small enough, there is an $a$-independent finite cover $V_i$ of holomorphic Darboux coordinate charts\footnote{If $z,w$ are coordinates on $V_i$, then $\eta=\sqrt{2} dz\wedge dw$.} of $X$ and a constant $C>0$ such that the eigenvalues of the metrics $g$ and $\tilde{g}$ in the local coordinates lie in the interval between $C^{-1}\frac{\vert a\vert}{r_a}$ and $C\frac{r_a}{\vert a\vert}$. 

For any compact set $K\subset X\setminus E$, there is an $a$-independent constant $C_K$ such that the eigenvalues of $g$ and $\tilde{g}$ in the local coordinate charts $K\cap V_i$ lie between $C_{K}^{-1}$ and $C_K$.
\end{Lem}

\begin{proof}
We start with the second part. As long as one stays away from the exceptional divisor, the patchwork metric $g$ can locally be written
\[g=g_{Euc} + \vert a\vert^2 h,\]
where $h$ is bounded with bounded derivatives. Cover the compact set $K$ with finitely many such coordinate charts to deduce the statement for $g$. The statement for $\tilde{g}$ follows by Corollary \ref{Cor:ComplexHessian} as we next show. Let $v$ be an eigenvector for $\tilde{g}$. Then
\[\vert \lambda v \vert  \leq \vert g v\vert + \vert (\nabla^2 \phi) g^{-1} g v\vert\leq C_K \vert v\vert(1+ C\sqrt{\vert a\vert}).\]
Similarly for the lower bound.

To estimate the eigenvalues also near the exceptional divisor, we take a careful look at the Eguchi-Hanson metric.
The Eguchi-Hanson metric on $\C^2\setminus \{0\}$ reads
\[g_{EH}=\sqrt{1+\frac{a^2}{u^2}}\left( \mathbb{1}-\frac{a^2}{a^2+u^2}\frac{\overline{z}\otimes z}{u}\right),\]
where $u=\vert z\vert^2$ is the Euclidean distance squared and we are writing $a$ instead of $a_i$. The eigenvalues in these coordinates are $\frac{\sqrt{a^2+u^2}}{u}$ and $\frac{u}{\sqrt{u^2+a^2}}$, hence are not bounded. We therefore need to choose different coordinates.
The metric extends to a complete metric on the total space of the cotangent bundle of $\C\P^1$, $\mathcal{O}_{\C\P^1}(-2)$, and we will first use coordinate patches on this total space. Recall that 
\[\mathcal{O}_{\C\P^1}(-2)=\left\{((z,w),(\xi:\varsigma))\, \vert \, z\varsigma^2 =w\xi^2\right\}\subset \C^2\times \C\P^1.\]
Working on the coordinate chart $\{\xi\neq 0\}$ and writing $\zeta\coloneqq \frac{\varsigma}{\xi}$, we have $w=\zeta^2 z$.  On this coordinate patch, we introduce the map
\[f_1\colon \mathcal{O}_{\C\P^1}(-2)\to \C^2/\mu_2\]
\[f_1(z,\zeta^2 z)=\sqrt{2}\left[(\sqrt{z},\zeta\sqrt{z})\right].\]
Here the brackets on the right hand side means the $\mu_2$-orbit.
This map is well-defined. The map $f_1$ along with its partner $f_2$ defined similarly on the set $\varsigma\neq 0$ realise $\mathcal{O}_{\C\P^1}(-2)$ as the blow-up of $\C^2/\mu_2$. See \cite{LyeEH} for more details. Pulling back the Eguchi-Hanson line element using $f_1$ then yields
\begin{equation}
ds^2_{EH}=\frac{1}{\sqrt{a^2+u^2}}\left((1+\vert \zeta\vert^2)^2\vert dz\vert^2+\frac{a^2+(1+\vert\zeta\vert^2)u^2}{(1+\vert \zeta\vert^2)^2} \vert d\zeta\vert^2 +4(1+\vert \zeta\vert^2)\text{Re}(\overline{z}\zeta dz d\overline{\zeta})\right),
\label{eq:EHonBundle1}
\end{equation}
where $u=2\vert z\vert(1+\vert \zeta\vert^2)$. 
Completely analogous expressions will be found on the other set $\{\varsigma\neq 0\}$, so we do not write these out.


This removes the divergence in the metric at $z=0\iff u=0$. Indeed, we simply have
\[ds^2_{EH, z=0}=\frac{(1+\vert \zeta\vert^2)^2}{a} \vert dz\vert^2 +\frac{a}{(1+\vert \zeta\vert^2)^2} \vert d\zeta\vert^2.\]
The eigenvalues of this metric are clearly uniformly bounded by $\frac{C}{a}$ and $Ca$ on the set $\vert \zeta\vert\leq 1$ and $u\leq 1$. 
For $\vert\zeta\vert\geq 1$, we make one final change of coordinates, writing
\[z=\frac{y}{\zeta^2} \;\;\;\;\; \&\;\;\;\;\; \zeta = -\frac{1}{\upsilon}.\]
In these coordinates, still at $u=0$, we find
\[ds^2_{EH}=\frac{(1+\vert \upsilon\vert^2)^2}{a} \vert dy\vert^2+\frac{a}{(1+\vert \upsilon\vert^2)^2}\vert d\upsilon\vert^2.\] These components are again uniformly bounded by $Ca$ and $\frac{C}{ a}$ for all $\vert \upsilon\vert \leq 1$.
These estimates were for a single component of the exceptional divisor with parameter $a_i$, so the upper bound has to be modified to $C\frac{1}{a_i}\leq C\frac{r_a}{\vert a\vert}$ and the lower bound to $C \vert a\vert \leq \frac{\max_i a_i}{\min_i a_i} \min_i a_i\leq r_a  a_i$. Combined with the estimates away from $E$, we the statement about the eigenvalues of $g$ follow. The bounds on $\tilde{g}$ follow from Corollary \ref{Cor:ComplexHessian} exactly as before. 

To see that the above coordinates are holomorphic Darboux coordinates, we look at what happens to the holomorphic volume form $\eta=\sqrt{2}dz_1 \wedge dz_2$ under these coordinate transformations. With the above coordinates, $z_1=\sqrt{2z}$, $ z_2=\zeta \sqrt{2z}$, $z=\frac{y}{\zeta^2}$ and $\zeta=-\frac{1}{\upsilon}$,  we find
\[dz_1\wedge dz_2= dz\wedge d\zeta = dy\wedge d\upsilon.\]
\end{proof}

%

The holomorphic Darboux coordinates are not suited for local analysis near the exceptional divisor since the ratio of the eigenvalues of $g$ is unbounded as $a\to 0$. This can be remedied by rescaling the fibre coordinate $z$ or $y$ in the above proof.
\begin{Lem}
\label{Lem:Rescaled}
Near a component $E_i$ of the exceptional divisor, we may choose finitely many coordinate patches and coordinates $(z_a, \zeta)$ such that 
\[C^{-1} a_i \leq g\leq C a_i\]
and
\[\eta=a_i \cdot \sqrt{2} dz_a \wedge d\zeta\]
in these coordinates. We shall refer to the above coordinates as \textbf{rescaled holomorphic Darboux coordinates}.
 
\end{Lem}

\begin{proof}
We return to the expression \eqref{eq:EHonBundle1}. 
By a change of scale, $z_a\coloneqq \frac{z}{a}$ and $u_a\coloneqq 2\vert z_a\vert (1+\vert\zeta\vert^2)$, \eqref{eq:EHonBundle1} becomes
\begin{equation}
ds^2_{EH}=\frac{a}{\sqrt{1+u_a^2}}\left((1+\vert \zeta\vert^2)^2\vert dz_a\vert^2+\frac{\left(1+ u_a^2(1+\vert \zeta\vert^2)\right)}{(1+\vert \zeta\vert^2)^2}\vert d\zeta \vert^2+ 4(1+\vert\zeta\vert^2)\text{Re}\left(\overline{z_a}\zeta dz_a d\overline{\zeta}\right)\right).
\label{eq:EHonBundle3}
\end{equation}
This expression has eigenvalues bounded by $Ca$ and $C^{-1} a$ for all $u\leq a$ (i.e. $ u_a\leq 1$) and $\vert \zeta\vert\leq 1$. For $\vert \zeta\vert\geq 1$, we use $y_a\coloneqq y/a$ and $\zeta=-1/\upsilon$ as before to get the same bounds $Ca$ and $C^{-1} a$.

\end{proof}

We will from now on be working locally on an open subset $V_j\subset X$. By Lemma \ref{Lem:Eigenvalues} and \ref{Lem:Rescaled}, we may choose this open set to have (rescaled) holomorphic Darboux coordinates. In particular, $V_j$ may be taken biholomorphic to a Euclidean ball $B_{2R}$ in $\C^2$ centered on the origin.  The Ricci-flat metric $\tilde{g}$ satisfied the Monge-Amp\`{e}re equation, which in these coordinates simply reads
\[\det(\tilde{g})=const.,\]
where the constant is A ($a_i^2 A$).
We may assume there exists a locally defined K\"{a}hler potential $\tilde{\phi}\colon V_j\to \R$ with $\nabla^2 \tilde{\phi}=\tilde{g}$. We will write $\tilde{\phi}$ and not $\tilde{\phi}_j$ to not clutter the notation.

\noindent\textbf{Step 2 -- A local Harnack inequality:}
The key analysis result will be the following.
\begin{Prop}[{\cite{Siu87}[p. 102]}]
Let $g$ be a K\"{a}hler metric on $B_{2R}$, the ball of radius $2R$ centered on $0\in \C^n$. Let $q>n$. Then there exists a $p>0$ and $C>0$ such that if $g^{\on \mu} \partial_{\mu}\partial_{\on} v\leq \theta$ and $v>0$ on $B_{2R}$, then
\begin{equation}
\label{eq:Harnack}
R^{-2n/p} \norm{v}_{L^p(B_R,g)}\leq C\left(\inf_{B_R} v + R^{2(q-n)/q} \norm{\theta}_{L^q(B_{2R,g})}\right).
\end{equation}
The constant $C$ depends on $n$, $\diam(B_r,g)$, $\Vol(B_r,g)$ and the constant in the Sobolov inequality
\[\norm{f}^2_{L^{2n/(n-1)}(B_{2R},g)}\leq C_{Sob}\left(\norm{\nabla f}^2_{L^2(B_{2R},g)}+ \norm{f}^2_{L^2(B_{2R},g)}\right)\]
for all compactly supported $f$.
\end{Prop}
We refer to \cite{Siu87}[pp. 107-112] for a proof. 
\begin{Cor}
Let $X$ be a Kummer K3 surface with Ricci-flat metric $\tilde{g}$. Let $V_j\cong B_{2R}$ be a holomorphic Darboux coordinate patch.   Then the constant $C$ in \eqref{eq:Harnack} can be chosen independently of $a$. 
\end{Cor}
\begin{proof}
The volume bound follows from the Monge-Amp\`{e}re equation directly, which prescribes the volume form of $\tilde{g}$. The diameter bound for $g$ is argued in the proof of Proposition \ref{Prop:PoincareIndep} in Appendix \ref{App:Param}. The diameter bound for $\tilde{g}$ follows from this and the bound on $\nabla^2 \phi$,  \eqref{eq:C2Norm}. The Sobolev constant can be controlled as long as one has upper and lower bounds on the volume and diameter and a lower bound on the Ricci curvature - see the proof of Proposition \ref{Prop:LiSob} in Appendix \ref{App:Param}. The Ricci curvature vanishes for $\tilde{g}$, hence we have a uniform Sobolev constant.
\end{proof}

\noindent\textbf{Step 3 -- Harnack inequality for supersolutions:}
The Monge-Amp\`{e}re equation reads
\[\det(\tilde{g})=const.\]
in (rescaled) holomorphic Darboux coordinates. Let $\zeta\in \C^2$ with $\vert \zeta \vert=1$ be arbitrary. Differentiating the logarithm of the equation and using the Jacobi formula yields
\[\Tr(\tilde{g}^{-1} \nabla^2 \tilde{\phi}_\zeta)=0\]
and
\[\Tr(\tilde{g}^{-1} \nabla^2 \tilde{\phi}_{\zeta\overline{\zeta}})-\Tr(\tilde{g}^{-1}\nabla^2 \tilde{\phi}_{\overline{\zeta}} \tilde{g}^{-1} \nabla^2 \tilde{\phi}_\zeta)=0.\]
The first term on the left hand side is per definition $\tilde{\Delta}\tilde{\phi}_{\zeta\overline{\zeta}}$. The second term is the tensor norm of $\nabla^2 \tilde{\phi}_{\zeta}$, hence can be dropped to give the inequality
\[\tilde{\Delta}\tilde{\phi}_{\zeta\overline{\zeta}} \geq 0.\]
This allows us to apply the Harnack inequality to the function
\[v_{sup} \coloneqq \sup_{B_{2R}} \tilde{\phi}_{\zeta\overline{\zeta}}  -\tilde{\phi}_{\zeta\overline{\zeta}} \geq 0\]
to deduce

\begin{equation}
R^{-4/p} \norm{v_{sup}}_{L^p(B_R,\tilde{g})}\leq C \inf_{B_R} v_{sup}.
\label{eq:SupHarnack}
\end{equation}

The constant $C$ on the right hand side can be chosen to be independent of $a$, as discussed in step 2.

\noindent\textbf{Step 4 -- Harnack inequality for subsolutions:}
We need two linear algebra results. The first goes as follows.
\begin{Lem}[{\cite[Lemme 1]{Gav}}]
Let $\mathcal{H}_+$ denote all $n\times n$ hermitian matrices with positive eigenvalues. Let $A\in \mathcal{H}_+$. Then
\[\det(A)^{1/n}=\frac{1}{n} \inf \{ \tr(AB)\, \vert\, B\in \mathcal{H}_+,\, \det(B)=1\}.\]
\end{Lem}
Let $x,y\in B_{2R}$. We will specify $x$ later, and $y$ will be integrated. 
Let $B=\kappa \tilde{g}^{-1}(y)$, where $\kappa=\sqrt{A}$ ($\kappa=a_i \sqrt{A}$) when we are in (rescaled) holomorphic Darboux coordinates. Then the Monge-Amp\`{e}re equation says $\det(B)=1$, and the lemma, implies
\[\sqrt{\kappa}=\sqrt{\det(\tilde{g}(x))}\leq  \frac{1}{2}\tr(B\tilde{g}(x)).\]
On the other hand, we trivially have
\[\tr(B \tilde{g}(y))=2\sqrt{\kappa},\]
so
\begin{equation}
\tr(B(\tilde{g}(y)-\tilde{g}(x))\leq 0.
\label{eq:SubHarnack1}
\end{equation}
To proceed, we need the second linear algebra result.
\begin{Lem}[{\cite[p.103]{Siu87},\cite{Blocki12}[Lemma 5.17]}]
For $0<\lambda<\Lambda<\infty$, let $S(\lambda,\Lambda)$ denote the set of hermitian $n\times n$-matrices with eigenvalues in the interval $[\lambda,\Lambda]$. Then one can find unit vectors $\zeta_1,\dots, \zeta_N\in \C^n$ and $0<\lambda_*< \Lambda_*<\infty$ depending only on $n,\lambda,$ and $\Lambda$ such that every $H\in S(\lambda,\Lambda)$ can be written
\[H=\sum_{k=1}^N \beta_k \zeta_k \otimes \overline{\zeta_k}\]
with $\beta_k\in [\lambda_*,\Lambda_*]$.

\end{Lem}
\begin{Rem}
The proof yields $\lambda_*<\lambda/N$ and $\Lambda_*>\Lambda$, but can otherwise be chosen arbitrarily. We may also assume that the finite set of vectors contains an orthonormal basis. 
\end{Rem}
With this at hand, we find locally defined functions $\beta_k$ with $\lambda_*\leq \beta_k \leq \Lambda_*$ such that
\[B=\sqrt{\kappa}\tilde{g}^{-1}(y)=\sum_{k=1}^N \beta_k(y) \zeta_k \otimes \overline{\zeta_k}.\]
Hence
\[\tr(B(\tilde{g}(y)-\tilde{g}(x)) =\sum_{k=1}^N \beta_k(y) \left(\tilde{\phi}_{\zeta_k \overline{\zeta_k}}(y)-\tilde{\phi}_{\zeta_k \overline{\zeta_k}}(x)\right).\]

Let
\[M_{k,R}\coloneqq \sup_{B_R} \tilde{\phi}_{\zeta_k \overline{\zeta_k}}\]
\[m_{k,R}\coloneqq \inf_{B_R} \tilde{\phi}_{\zeta_k \overline{\zeta_k}},\]
and introduce the oscillation
\[\osc(R)\coloneqq \sum_{k=1}^N (M_{k,R}-m_{k,R}).\]
We also introduce the short-hand
\[w_k\coloneqq \tilde{\phi}_{\zeta_k \overline{\zeta_k}}.\]
Let $\ell \in \{1,\dots,N\}$ be arbitrary. Then the Harnack inequality tells us
\begin{align*}
R^{-4/p} \norm{\sum_{k\neq \ell}\left( M_{k,2R}-w_k\right)}_{L^p(B_R)}&\leq R^{-4/p} \sum_{k\neq \ell}\norm{ M_{k,2R}-w_k}_{L^p(B_R)} \\ 
&\leq C\left(\sum_{k\neq \ell} (M_{k,2R}-M_{k,R})\right).
\end{align*}
The last sum can be estimate a bit. Since $M_{k,2R}-M_{k,R}\geq 0$, we can include the term $k=\ell$ on the right hand side. We further have 
\[M_{k,2R}-M_{k,R}\leq M_{k,2R}-M_{k,R}+(m_{k,R}-m_{k,2R})=(M_{k,2R}-m_{k,2R})-(M_{k,R}-m_{k,R}).\]
So
\begin{equation}
R^{-4/p} \norm{\sum_{k\neq \ell}\left( M_{k,2R}-w_k\right)}_{L^p(B_R)}\leq C(\osc(2R)-\osc(R)).
\label{eq:OscStep}
\end{equation}
From \eqref{eq:SubHarnack1} and $\lambda_*\leq \beta_k\leq \Lambda_*$, we find
\[\lambda_* \vert w_{\ell}(y)-w_{\ell}(x)\vert\leq \Lambda_* \sum_{k\neq \ell}\vert M_{k,2R}-w_k(y)\vert.\]
Taking averaged $L^p$-norms here and using \eqref{eq:OscStep} gives us
\[R^{-4/p}\norm{w_{\ell}-w_{\ell}(x)}_{L^p(B_R)}\leq \frac{\Lambda_*}{\lambda_*} C(\osc(2R)-\osc(R)),\]
where all the integrals are with respect to $y$.
For any $\epsilon>0$, we can find $x\in B_{2R}$ such that $w_{\ell}(x)=m_{\ell,2R}+\epsilon$. Using this on the left hand side, we deduce
\[R^{-4/p}\norm{w_{\ell}-m_{\ell,2R}}_{L^p(B_R)}\leq \frac{\Lambda_*}{\lambda_*} C(\osc(2R)-\osc(R)) + \Vol(B_1) \epsilon.\]
Since $\epsilon$ was arbitrary, we can send it to $0$ and deduce
\begin{equation}
R^{-4/p}\norm{w_{\ell}-m_{\ell,2R}}_{L^p(B_R)}\leq \frac{\Lambda_*}{\lambda_*} C(\osc(2R)-\osc(R))
\label{eq:SubHarnack}
\end{equation}
for any $\ell \in \{1,\dots,N\}$.

\noindent\textbf{Step 5 -- Combining both Harnack estimates:}
Let $\ell \in \{1,\dots, N\}$. Then we have
\begin{align*}
\Vol(B_R)^{1/p}(M_{\ell,2R}-m_{\ell,2R})&=\norm{M_{\ell,2R}-m_{\ell,2R}}_{L^p(B_R)}\\
&\leq \norm{w_{\ell}-m_{\ell,2R}}_{L^p(B_R)}+\norm{M_{\ell,2R}-w_\ell}_{L^p(B_R)}.
\end{align*} 
Multiplying both sides by $R^{-4/p}$ and using both \eqref{eq:SupHarnack} and \eqref{eq:SubHarnack} yields
\[M_{\ell,2R}-m_{\ell,2R} \leq C \frac{\Lambda_*}{\lambda_*}\left(\osc(2R)-\osc(R)\right).\]
Summing over $\ell$ gives us
\begin{equation}
\osc(R)\leq \delta \osc(2R),
\label{eq:OscEquation}
\end{equation}
where $\delta\coloneqq 1-\frac{\lambda_*}{\Lambda_*CN}$.


The inequality \eqref{eq:OscEquation} gives us the H\"{o}lder regularity by \cite[Lemma 8.23]{GT}. Indeed, let $r<R$. Choose $m>0$ so that 
\[2^{-m} R \leq r\leq 2^{-m+1} R,\]
i.e.
\[m\geq \frac{\log\left(\frac{R}{r}\right)}{\log(2)}.\]
Then \eqref{eq:OscEquation} and the monotonicity of $\osc$ say
\[\osc(r)\leq \osc(2^{-m+1} R)\leq \delta^{m-1} \osc(R)\leq \frac{1}{\delta}\cdot \delta^{\frac{\log\left(\frac{R}{r}\right)}{\log(2)}} \osc(R)=\frac{1}{\delta}\left(\frac{r}{R}\right)^{-\frac{\log(\delta)}{\log(2)}} \osc(R).\]
The oscillation $\osc(R)$ can be bounded by bounding $\vert \tilde{\phi}_{\zeta\overline{\zeta}}\vert$ for arbitrary $\zeta$. For $\zeta\in \C^2$, there are $\mu,\nu\in \C$ such that $\zeta=\mu z+\nu w$, hence
\[\tilde{\phi}_{\zeta\overline{\zeta}}=\vert \mu\vert^2 \tilde{\phi}_{z\overline{z}}+2\text{Re}(\mu\overline{\nu} \tilde{\phi}_{z\overline{w}})+\vert \nu\vert^2 \tilde{\phi}_{w\overline{w}}=\ip{\begin{pmatrix}
\mu \\ \nu 
\end{pmatrix}}{\tilde{g}\begin{pmatrix}
\mu\\ \nu
\end{pmatrix}}.\]
The right hand side can be bounded by a constant (depending on $\zeta$) and the largest eigenvalue of $\tilde{g}$. Hence, by Lemma \ref{Lem:Eigenvalues} and \ref{Lem:Rescaled},
\[\vert \tilde{\phi}_{\zeta\overline{\zeta}}\vert \leq C \Lambda,\]
and 
\begin{equation}
\label{eq:Osc}
\osc(r)\leq C \Lambda\left(\frac{r}{R}\right)^\alpha 
\end{equation}
with
\[\alpha=-\frac{\log(1-\frac{\lambda_*}{C\Lambda_*})}{\log(2)}.\]
To finish, we have to bound $\alpha$. We distinguish between being near and away from the exceptional divisor. Lemma \ref{Lem:Eigenvalues} yields uniform bounds on the eigenvalues of $\tilde{g}$ away from the exceptional divisor. So $\frac{\lambda_*}{\Lambda_*}$ is uniformly bounded, and thus $\alpha<1$ uniformly. 

Near the exceptional divisor, the eigenvalues $\lambda$, $\Lambda$ are poorly behaved in the coordinates \eqref{eq:EHonBundle1}, so we use the rescaled coordinates \eqref{eq:EHonBundle3}. In these coordinates, $\frac{\lambda}{\Lambda}$ is uniformly bounded by Lemma \ref{Lem:Rescaled} as long as $u\leq a_i$. This can be achieved by rescaling $R\mapsto \sqrt{a_i}R$. By Lemma \ref{Lem:Rescaled} and Corollary \ref{Cor:ComplexHessian}, we have $\Lambda\leq C\sqrt{\vert a\vert}$. 
Hence
\[\osc(r)\leq C\left(\frac{r}{R}\right)^{\alpha} \vert a\vert^{\frac{1-\alpha}{2}}\]
 with $\alpha<1$ uniformly. 
So we get H\"{o}lder estimates also near the exceptional divisor.

\noindent\textbf{Step 6 -- H\"{o}lder bounds on $g$:}

\begin{Lem} 
 \label{Lem:PhiBound}
 Let $\{V_i\}$ be the cover of (rescaled) holomorphic Darboux coordinate neighbourhoods of Lemma \ref{Lem:Eigenvalues} and \ref{Lem:Rescaled}. Let   $\Phi_j \colon V_i \to \R$ be K\"{a}hler potentials for the patchwork metric $g$. Then there are uniform constants $C>0$ and $0< \alpha\leq 1$ such that
 \[\norm{\Phi_j}_{C^{2,\alpha}(V_i,g)}\leq C.\] 
\end{Lem} 
 
\begin{proof}
For the Euclidean region, this is clear. On an annulus $K$ around a component of the exceptional divisor $E_i$, the potential \eqref{eq:EHPot} can be written
\[f_{EH}=f_{Euc}+ a_i^2 \xi_K\]
for some smooth function $\xi_K$ which is regular as $a_i\to 0$. In the neck region, the patchwork potential can thus be written
\[\Phi=f_{Euc}+a_i^2 \chi \xi_K\]
where $\chi$ is a smooth cutoff function as described in Section \ref{Section:Kummer}. These expressions then give the required bound.

Near the exceptional divisor, one could compute the real Hessian of the Eguchi-Hanson potential \eqref{eq:EHPot} directly and compare. But it is probably easier to just repeat steps 3-5, since the potential $\Phi$ satisfies 
\[\det(\nabla^2 \Phi)=const.\]
where the constant is 1 ($a_i^2$) in (rescaled) holomorphic Darboux coordinates. 
\end{proof}

\noindent\textbf{Step 7 -- H\"{o}lder bounds on $\phi$:}
In a coordinate patch $V_i$, we write
\[\tilde{\phi}=\Phi_i + \phi,\]
hence
\[\vert  D^2\phi\vert^2_g\leq \vert  D^2\tilde{\phi}\vert^2_g+\vert D^2\Phi\vert^2_g\leq C\vert  D^2\tilde{\phi}\vert^2_{\tilde{g}}+C\leq \tilde{C},\]
where we have used Corollary \ref{Cor:ComplexHessian} to compare the $g$- and $\tilde{g}$-norm.

\subsection{$C^1$ estimates}
The $C^1$ estimates of $\phi$ follow from a general result in Riemannian geometry.
\begin{Lem}
Let $(M,g)$ be a compact, connected Riemannian manifold with diameter $d\coloneqq \diam(M,g)$. Then there exists a monotonely increasing function $\alpha\colon [0,\infty)\to \R$ such that 
\begin{equation}
\sup_{p\in M} \vert \nabla f\vert^2_g \leq \alpha(d) \sup_{p\in M} \vert D^2 f\vert^2_g
\end{equation}
holds for any $f\in C^2(M;\R)$.
\end{Lem}

\begin{proof}
The idea is simply to integrate the double derivative along a curve. Here are the details.

For an $f\in C^2(M;\R)$, let $q\in M$ denote a critical point. This point exists due to the compactness of $M$. Let $p\in M$ be any point. For $\epsilon>0$, let $\gamma\colon [0,T]$ denote a unit speed curve with $\gamma(0)=q$, $\gamma(T)=p$, and $T=L(\gamma)\leq d(p,q)+\epsilon$. Let $\xi\colon [0,T]\to \R$ denote the function
\[\xi(t)\coloneqq \vert \nabla f\vert^2_g(\gamma(t)).\]
This is differentiable, and
\[\xi'(t)\leq 2\vert D^2 f \nabla f\vert_g \leq \vert D^2 f\vert_g^2 + \vert \nabla f\vert_g^2,\]
where we have used the unit speed condition and the inequality between arithmetic and geometric mean. This differential inequality tells us
\[\frac{d}{dt} e^{-t}\xi(t)\leq e^{-t}\vert D^2 f\vert_g^2(\gamma(t)),\]
which integrates to
\[e^{-T} \xi(T)\leq \int_0^T e^{-t}\vert D^2 f\vert_g^2(\gamma(t))\, dt\leq T \sup_{p\in M}\vert D^2 f\vert_g^2(p).\]
This then yields
\[\xi(T)\leq (d+\epsilon)e^{d+\epsilon} \sup_{p\in M}\vert D^2 f\vert_g^2(p).\]
Doing the same construction for other points tells us
\[\sup_{p\in M} \vert \nabla f\vert^2_g(p) \leq (d+\epsilon)e^{d+\epsilon} \sup_{p\in M}\vert D^2 f\vert_g^2(p).\]
Letting $\epsilon\to 0$ gives the required bound with $\alpha(d)=d e^d$.
\end{proof}
\begin{Rem}
The above constructed $\alpha$ is not sharp. Indeed, if there are no critical points for $f$ between $\gamma(0)=q$ and $\gamma(T)=p$, then the function $\xi(t)=\vert \nabla f\vert_g(\gamma(t))$ is differentiable, and repeating the above argument yields $\xi'(t)\leq \vert D^2 f\vert_g(\gamma(t))$,
hence
\[ \sup_{p\in M} \vert \nabla f\vert_g \leq d \sup_{p\in M} \vert D^2 f\vert_g.\]
Since Morse functions are dense in the $C^2$-topology, we can approximate an arbitrary function with a function with isolated critical points. By shifting the curve $\gamma$ a bit, we can make sure there are no critical points between $q$ and $p$, and use the above estimate. All in all one finds $\alpha(d)=d^2$ as a better constant.  
\end{Rem}

\subsection{Higher order estimates}
 Kobayashi's approach to proving the higher order derivatives is worth sketching here. For $t\in [0,1]$ consider the equation
\begin{equation}
\left(\omega+i\partial \overline{\partial} \phi_t\right)^2=\left(1+t\left(Ae^{\psi}-1\right)\right) \omega^2
\label{eq:MAt}
\end{equation}
subject to
\[\int_X \phi_t\, \omega^2=0.\]
Here $A,\psi$ and $\omega$ are as before. This is what one considers for the continuity method to show that the Monge-Amp\`{e}re equation has a solution. Computing the $t$-derivative of \eqref{eq:MAt} yields
\begin{equation}
\tilde{\Delta}_t\left(\frac{\partial \phi_t}{\partial t}\right)=\frac{Ae^{\psi}-1}{1+t(Ae^{\psi}-1)},
\label{eq:Poisson}
\end{equation}
where we have introduced 
\[\tilde{\Delta}_t\coloneqq \Tr\left(\left(g+\nabla^2 \phi_t\right)^{-1}\nabla^2\right).\] 
Then \eqref{eq:AValue} and \eqref{eq:psibound} tell us that the right hand side is uniformly bounded from above and below,
\[\left\vert \frac{Ae^{\psi}-1}{1+t(Ae^{\psi}-1)}\right\vert\leq C\vert a\vert^2.\]
Indeed, writing $\varpi \coloneqq \frac{Ae^{\psi}-1}{1+t(Ae^{\psi}-1)}$, we have
\[\norm{\varpi}_{C^k(X,g)}\leq C_k \vert a\vert^2\]
for all $k\geq 0$.
Repeating the proofs\footnote{There is a slight added difficulty with the H\"{o}lder bound. The Monge-Amp\`{e}re equation for $\phi_t$ in holomorphic Darboux coordinates reads $\det(g+\nabla^2 \phi_t)=(1-t)e^{-\psi}+tA$, so the right hand side is not constant in the neck region when $t\neq 1$. The same proof goes through, however, with some harmless additional terms appearing in \eqref{eq:OscEquation}. We refer to \cite[pp. 100-107]{Siu87} for the necessary modifications.} of the $C^{2.\alpha}$-bounds for $\phi_t$ instead of $\phi$ gives us uniform H\"{o}lder bounds on $\nabla^2 \phi_t$. Staying away from the exceptional divisor, we can write $g=g_{Euc}+\mathcal{O}(\vert a\vert^2)$. So \eqref{eq:Poisson} is a Poisson equation with H\"{o}lder continuous coefficients and H\"{o}lder continuous right hand side. So Schauder estimates, \cite[Theorem 6.2]{GT} from elliptic theory gives local $C^{k,\alpha}$-bounds on $\partial_t \phi_t$, hence also on $\phi=\phi_1$ by integrating and using $\phi_0=0$. 

Near the exceptional divisor, one has to use coordinates like \eqref{eq:EHonBundle1} or \eqref{eq:EHonBundle3}. The last coordinates require a rescaling, hence the drop in powers of $a$ in \cite[equation 47, 48]{Kob90}.

\section{Discussion and Outlook}
\label{Section:Disc}
\subsection{Flatness of the tori}
In the author's unpublished PhD thesis \cite{PhD} there is an argument that a certain special Lagrangian torus $L$, which can be found by similar arguments as in Theorem \ref{Thm:StableGeod} using an anti-holomorphic isometry of rank 2  of Theorem is flat. The argument was that a special Lagrangian submanifold has a prescribed volume form
\[ \Vol_L=\sqrt{A} \text{Re}(\eta)_{\vert L} =\sqrt{A} g(\cdot,J_0 \cdot),\]
where $J_0$ is a complex structure on the flat torus $T$ which is orthogonal to $I$ and $A$ is the constant in \eqref{eq:AValue}. But $L$ is K\"{a}hler with the K\"{a}hler form $\Vol_L=\omega_J=\tilde{g}(\cdot, J\cdot)$. Hence
\[\tilde{g}(\cdot, J\cdot)=\sqrt{A} g(\cdot,J_0 \cdot).\]
Using the isometries again \cite{PhD}, argues $J_{\vert L}=J_0$, hence $\tilde{g}_{\vert L}=g_{\vert L}$ and the torus is flat. The proof of $J_{\vert L}=J_0$ is sadly wrong, however.

\subsection{Reducing the amount of symmetry}
Theorem \ref{Thm:StableGeod} only works for very special tori. A natural questions is what happens when one perturbs the underlying lattice $\Gamma$ or the sizes $a_i$. We do now know the answer to this, but would remark that closed geodesics need not behave nicely under metric perturbations. The oldest example of this is due to Morse \cite[Chapter IX, Theorem 4.1]{Morse}, where he studies ellipsoids which are almost spherical. What he finds is  that there are closed geodesics which become infinitely long when the ellipsoid is deformed to a round sphere. There is a modern proof in \cite{KlingenbergRiemann}, and generalizations by Ballmann in \cite{Bal83} and Bangert \cite{Ban86}. What this tells us is that one cannot in general expect the closed geodesics of a perturbed metric to be perturbations of the original closed geodesics.

We do now know if the situation is improved when considering hyperk\"{a}hler deformations, and this is a question we hope to return to in the future.

\section{Acknowledgements}
This work is both a summary and continuation of the authors PhD thesis \cite{PhD} written at the Albert-Ludwigs-Universit\"{a}t Freiburg under the expert guidance of Nadine Gro\ss e and Katrin Wendland. The project was suggested by them, and several ideas and suggestions (and the absence of several mistakes) are due to them. They both took the time to read several earlier drafts of this paper in detail and provided helpful feedback. Any remaining errors are solely the author's fault.

\appendix

\section{Laplacian of the Riemann tensor}
\label{App:Riemann}
Here we present part of the computation underlying Proposition \ref{Prop:RiemannCrit}. We start with a couple of lemmas.
\begin{Lem}
\label{Lem:LaplaceLemma1}
Assume $(X,\tilde{g})$ is a  hyperk\"{a}hler 4-manifold. Let $V$ be any (locally defined) tangent vector field. Then
\begin{align}
\ip{\Delta R(V,IV)IV}{V}&=\ip{[\nabla_{JV},\nabla_{V}]R(V,KV)IV}{V}+\ip{[\nabla_{JV},\nabla_{IV}]R(V,JV)IV}{V}\notag \\ &+\ip{[\nabla_{V},\nabla_{KV}]R(V,JV)IV}{V}+\ip{[\nabla_{KV},\nabla_{IV}]R(V,KV)IV}{V},
\end{align}
where 
\[\Delta R=(\nabla^2_V+\nabla^2_{IV}+\nabla_{JV}^2+\nabla_{KV}^2)R\]
is the Laplacian acting on 4-tensors and all the commutators on the right hand side are acting on the 4-tensor $R$.
\end{Lem}
\begin{proof}
This follows from using the second Bianchi identity twice along with some hyperk\"{a}hler identities. By the second Bianchi identity and $IJ=K$, we find
\[\nabla_{JV}R(V,IV)=-\nabla_VR(IV,JV)-\nabla_{IV}R(JV,V)=\nabla_V R(V,KV)+\nabla_{IV}R(V,JV).\]
Hence
\begin{align}
\nabla^2_{JV}R(V,IV)&=\nabla_{JV}\nabla_V R(V,KV)+\n_{JV} \n_{IV} R(V,JV) \notag \\
&=[\n_{JV},\n_{V}]R(V,KV)+\n_V\n_{JV} R(V,KV) \notag  \\ 
&+[\n_{JV},\n_{IV}]R(V,JV)+\n_{IV}\n_{JV}R(V,JV).
\end{align}
To the second term we apply the second Bianchi identity again to get
\begin{align*}
\n_V\n_{JV} R(V,KV)&=-\n_V^2 R(KV,JV) -\n_V \n_{KV}R(JV,V) \\
&=-\n_V^2 R(V,IV)+\n_V \n_{KV}R(V,JV).
\end{align*}
The fourth term we handle similarly, using the second Bianchi identity in the last two components to write
\begin{align*}
\ip{\n_{IV}\n_{JV}R(V,JV)IV}{V} &=-\ip{\n_{IV}^2R(V,JV)V}{JV}-\ip{\n_{IV} \n_V R(V,JV)JV}{IV}\\
&=\ip{\n_{IV}^2R(V,JV)JV}{V}+\ip{\n_{IV} \n_V R(V,JV)KV}{V}.
\end{align*}
Hence we have
\begin{align*}
\ip{\nabla^2_{JV}R(V,IV)IV}{V}&= \ip{[\n_{JV},\n_{V}]R(V,KV)IV}{V}\\
&+\ip{[\n_{JV},\n_{IV}]R(V,JV)IV}{V} \\
&-\ip{\n_V^2 R(V,IV)IV}{V}+\ip{\n_V \n_{KV}R(V,JV)IV}{V}\\
&+\ip{\n_{IV}^2R(V,JV)JV}{V}+\ip{\n_{IV} \n_V R(V,JV)KV}{V}.
\end{align*}
Performing the same computation with $KV$ instead of $JV$, we find

\begin{align*}
\ip{\nabla^2_{KV}R(V,IV)IV}{V}&= \ip{[\n_{V},\n_{KV}]R(V,JV)IV}{V}\\
&+\ip{[\n_{KV},\n_{IV}]R(V,KV)IV}{V}\\
&-\ip{\n_V^2 R(V,IV)IV}{V}-\ip{\n_{V} \n_{JV}R(V,KV)IV}{V}\\
&+\ip{\n_{IV}^2R(V,KV)KV}{V}-\ip{\n_{IV} \n_V R(V,KV)JV}{V}.
\end{align*}
Using the Bianchi identity again, we rewrite
\[-\ip{\n_{V} \n_{JV}R(V,KV)IV}{V}=\ip{\n_{V}^2R(KV,JV)IV}{V}+\ip{\n_V \n_{KV}R(JV,V)IV}{V},\]
so
\begin{align*}
\ip{\nabla^2_{KV}R(V,IV)IV}{V}&= \ip{[\n_{V},\n_{KV}]R(V,JV)IV}{V}\\
&+\ip{[\n_{KV},\n_{IV}]R(V,KV)IV}{V}\\
&+\ip{\n_V \n_{KV}R(JV,V)IV}{V}\\
&+\ip{\n_{IV}^2R(V,KV)KV}{V}-\ip{\n_{IV} \n_V R(V,KV)JV}{V}.
\end{align*}
Adding these two expressions and using the Ricci-flatness (equivalently; the first Bianchi identity) to write
\[\ip{\n_{IV}^2R(V,JV)JV}{V}+\ip{\n_{IV}^2R(V,KV)KV}{V}=-\ip{\n_{IV}^2R(V,IV)IV}{V},\]
we arrive at
\begin{align*}
\ip{\Delta R(V,IV)IV}{V}&=\ip{[\nabla_{JV},\nabla_{V}]R(V,KV)IV}{V}+\ip{[\nabla_{JV},\nabla_{IV}]R(V,JV)IV}{V}\notag \\ &+\ip{[\nabla_{V},\nabla_{KV}]R(V,JV)IV}{V}+\ip{[\nabla_{KV},\nabla_{IV}]R(V,KV)IV}{V}
\end{align*}
as announced.
\end{proof}

\begin{Lem}
\label{Lem:LaplaceLemma2}
Assume $(X,\tilde{g})$ is a hyperk\"{a}hler 4-manifold. Let $U,V,W$ be any tangent vector fields and write $(I_1,I_2,I_3)=(I,J,K)$ for the complex structures satisfying $IJ=K$. Then
\begin{align}
\ip{[\nabla_W,\nabla_U]R(V,I_iV)I_jV}{V}&=\ip{\nabla_{[W,U]}R(V,I_iV)I_jV}{V}\notag \\ &-2\ip{R(R(W,U)V,I_iV)I_jV}{V} \notag \\
&-2\ip{R(V,I_iV)I_jV}{R(W,U)V} 
\end{align} 

\end{Lem}

\begin{proof}
This follows from the definitions and is quite standard. We start by computing
\begin{align*}
\nabla_U \ip{R(V,I_iV)I_jV}{V}&= \ip{\nabla_UR(V,I_iV)I_jV}{V}\\ &+2 \ip{R(\nabla_UV,I_iV)I_jV}{V}\\ &+ 2\ip{R(V,I_iV)I_jV}{\nabla_UV}.
\end{align*}
Similarly,
\begin{align*}
\nabla_W\nabla_U \ip{R(V,I_iV)I_jV}{V}&= \ip{\nabla_W\nabla_UR(V,I_iV)I_jV}{V}\\ 
&+ 2\ip{\nabla_UR(\nabla_W V,I_iV)I_jV}{V}\\
&+2\ip{\nabla_UR(V,I_iV)I_jV}{\nabla_WV}\\
&+2 \ip{\nabla_WR(\nabla_UV,I_iV)I_jV}{V}\\
&+ 2 \ip{R(\nabla_W\nabla_UV,I_iV)I_jV}{V}\\
&+2 \ip{R(\nabla_UV,I_i\nabla_WV)I_jV}{V} \\
&+4 \ip{R(\nabla_UV,I_iV)I_jV}{\n_WV}\\
&+2\ip{\nabla_W R(V,I_iV)I_jV}{\nabla_UV}\\
&+4\ip{R(\nabla_WV,I_iV)I_jV}{\nabla_UV}\\
&+2\ip{R(V,I_iV)I_j\nabla_WV}{\nabla_UV}\\
&+2\ip{R(V,I_iV)I_jV}{\nabla_W\nabla_UV}.
\end{align*}
Subtracting the same expression with $U$ and $W$ swapped, we get
\begin{align*}
\nabla_{[W,U]} \ip{R(V,I_iV)I_jV}{V}&= \ip{[\nabla_W,\nabla_U]R(V,I_iV)I_jV}{V}\\ 
&+ 2 \ip{R([\nabla_W,\nabla_U]V,I_iV)I_jV}{V}\\
&+2\ip{R(V,I_iV)I_jV}{[\nabla_W,\nabla_U]V}.
\end{align*}
The left hand side can be written
\begin{align*}
\nabla_{[W,U]} \ip{R(V,I_iV)I_jV}{V} &=  \ip{\nabla_{[W,U]}R(V,I_iV)I_jV}{V} \\ 
&+  2\ip{R(\nabla_{[W,U]}V,I_iV)I_jV}{V} \\
&+  2\ip{R(V,I_iV)I_jV}{\nabla_{[W,U]}V}.
\end{align*}
On the right hand side, we use the definition of the curvature tensor to write
\[[\nabla_W,\nabla_U]V=R(W,U)V+\n_{[W,U]}V.\]
This yields the claimed formula.
\end{proof}

From now on, we assume $V$ is a (locally defined) unit vector field. 
We recall that $\sigma_{IJ}=\sigma_{IJ}(V)$ etc. are defined so that
\[R(V,IV)V=-\sigma_{II}IV-\sigma_{IJ}JV-\sigma_{IK}KV\]
and so on for $R(V,JV)$ and $R(V,KV)$. Using this and the previous two lemmas, we arrive at
\begin{Prop}

Assume $(X,\tilde{g})$ is a hyperk\"{a}hler 4-manifold. Let $V$ be any locally defined unit tangent vector field. Then
\begin{align}
\label{eq:LaplacianGen}
\ip{\Delta R(V,IV)IV}{V}&=\ip{\nabla_{[JV,IV]}R(V,JV)IV}{V} \\  \notag
&+\ip{\nabla_{[JV,V]}R(V,KV)IV}{V} \\ \notag
&+\ip{\nabla_{[V,KV]}R(V,JV)IV}{V}\\ \notag
&+\ip{\nabla_{[KV,IV]}R(V,KV)IV}{V}\\ \notag
&-4(\sigma_{II}^2+2\sigma_{JJ}\sigma_{KK}+\sigma_{IJ}^2+\sigma_{IK}^2-2\sigma_{JK}^2) \notag
\end{align}
\end{Prop}
\begin{proof}
Lemma \ref{Lem:LaplaceLemma2} says
\begin{align*}
 \ip{[\nabla_{JV},\nabla_V]R(V,KV)IV}{V}&= \ip{\nabla_{[JV,V]}R(V,KV)IV}{V}\notag \\ &-2\ip{R(R(JV,V)V,KV)IV}{V} \notag \\
&-2\ip{R(V,KV)IV}{R(JV,V)V}. 
\end{align*}
Using $R(JV,V)V=\sigma_{IJ}IV+\sigma_{JJ}JV+\sigma_{JK}KV$, we find
\begin{align*}
\ip{R(R(JV,V)V,KV)IV}{V}&=\sigma_{IJ}\ip{R(IV,KV)IV}{V}+\sigma_{JJ}\ip{R(JV,KV)IV}{V}\\ &=\sigma_{IJ}^2-\sigma_{II}\sigma_{JJ}.
\end{align*}
and
\begin{align*}
\ip{R(V,KV)IV}{R(JV,V)V}&=\sigma_{JJ}\ip{R(V,KV)IV}{JV}+\sigma_{JK}\ip{R(V,KV)IV}{KV}\\ &=\sigma_{JJ}\sigma_{KK}-\sigma_{JK}^2.
\end{align*}
Hence
\begin{align*}
 \ip{[\nabla_{JV},\nabla_V]R(V,KV)IV}{V}&= \ip{\nabla_{[JV,V]}R(V,KV)IV}{V}\notag \\ &+2\left(\sigma_{JJ}\sigma_{II}+\sigma_{JK}^2-\sigma_{JJ}\sigma_{KK}-\sigma_{IJ}^2\right). 
\end{align*}
Similar computations yield
\begin{align*}
 \ip{[\nabla_{JV},\nabla_{IV}]R(V,JV)IV}{V}&= \ip{\nabla_{[JV,IV]}R(V,JV)IV}{V}\notag \\ &+2\left(\sigma_{II}\sigma_{KK}+\sigma_{JK}^2-\sigma_{JJ}\sigma_{KK}-\sigma_{IK}^2\right), 
\end{align*}
\begin{align*}
 \ip{[\nabla_{V},\nabla_{KV}]R(V,JV)IV}{V}&= \ip{\nabla_{[V,KV]}R(V,JV)IV}{V}\notag \\ &+2\left(\sigma_{II}\sigma_{KK}+\sigma_{JK}^2-\sigma_{JJ}\sigma_{KK}-\sigma_{IK}^2\right), 
\end{align*}
and
\begin{align*}
 \ip{[\nabla_{KV},\nabla_{IV}]R(V,KV)IV}{V}&= \ip{\nabla_{[KV,IV]}R(V,KV)IV}{V}\notag \\ &+2\left(\sigma_{JJ}\sigma_{II}+\sigma_{JK}^2-\sigma_{JJ}\sigma_{KK}-\sigma_{IJ}^2\right). 
\end{align*}
Adding up these four contributions, writing $\sigma_{KK}+\sigma_{JJ}=-\sigma_{II}$, and using Lemma \ref{Lem:LaplaceLemma1} gives \eqref{eq:LaplacianGen}.
\end{proof}

Along the fixed point set $M$ of an order 4 holomorphic isometry $f\colon X\to X$ we may simplify greatly.
\begin{Prop}
\label{Prop:LaplaceRiemann}
Assume the setup of Proposition \ref{Prop:RiemannCrit}. Let $\alpha\coloneqq  \ip{IV}{\nabla_{V}V}$ and $\beta\coloneqq \ip{IV}{\nabla_{IV} V}$. Then 
\begin{equation}
\ip{\Delta R(V,IV)IV}{V}=\left(\alpha+\beta\right)\n_V \sigma_{II}(V)-\alpha \n_{IV}\sigma_{II}(V)-6\sigma_{II}^2
\label{eq:LaplacianonM}
\end{equation}
holds for any point on $M$.
\end{Prop}

\begin{proof} 
Let $f\colon X\to X$ denote the isometry with $M$ as fixed point set.
Due to $f_*(JV)=\pm KV$, we find
\begin{align}
\label{eq:LaplaceStep}
\ip{\Delta R(V,IV)IV}{V}&=2\ip{\nabla_{[JV,IV]}R(V,JV)IV}{V} \\ \notag
&+2\ip{\nabla_{[JV,V]}R(V,KV)IV}{V} \\ \notag
&-6\sigma_{II}^2
\end{align}
along $M$. So we analyze the commutators. 
 Since $V$ is unit speed, we have $\ip{V}{\nabla_W V}=0$ for any $W$. Since $f_*$ fixes $V,IV$ and rotates $JV, KV$, there have to be functions $\alpha,\beta,\mu,\nu\colon M\to \R$ such that
\[\n_V V=\alpha IV\]
\[\n_{IV}V=\beta IV,\]
\[\n_{JV}V=\mu JV+\nu KV,\]
\[\n_{KV}V=\mu KV-\nu JV.\]
Using the formula $[W,U]=\n_WU-\n_UW$ we thus arrive at
\[[JV,IV]=-\nu JV +(\mu+\beta)KV,\]
\[[JV,V]=\mu JV+(\alpha+\nu)KV.\]
 The second Bianchi identity tells us
\[\ip{\nabla_{JV}R(V,JV)IV}{V}=\ip{\nabla_{IV}R(V,JV)JV}{V}-\ip{\nabla_V R(V,JV)JV}{V}\] 
 and
\[\ip{\nabla_{KV}R(V,JV)IV}{V}=\ip{\nabla_{V}R(V,IV)IV}{V}+\ip{\nabla_{JV} R(V,KV)IV}{V}.\] 
Applying the isometry $f$ to these further gives
 \[\ip{\nabla_{JV}R(V,JV)IV}{V}=\ip{\nabla_{KV}R(V,KV)IV}{V}\]
 and
 \[\ip{\nabla_{KV}R(V,JV)IV}{V}=-\ip{\nabla_{JV}R(V,KV)IV}{V},\]
 hence
 \[2\ip{\nabla_{KV}R(V,JV)IV}{V}=\ip{\nabla_{V}R(V,IV)IV}{V}.\]
 We have
 \[\ip{\nabla_{IV}R(V,JV)JV}{V}=\nabla_{IV} \sigma_{JJ}-4\ip{R(\nabla_{IV}V,JV)JV}{V}=\nabla_{IV} \sigma_{JJ}+4\beta \sigma_{JK},\]
 so $\sigma_{JJ}+\sigma_{KK}=-\sigma_{II}$ along with $\sigma_{JJ}=\sigma_{KK}$ tell us
 \[2\ip{\nabla_{IV}R(V,JV)JV}{V}=-\nabla_{IV} \sigma_{II}\]
 on $M$. Similarly
 \[2\ip{\nabla_{V}R(V,JV)JV}{V}=-\nabla_{V} \sigma_{II}.\]
 So
 \[2\ip{\nabla_{JV}R(V,JV)IV}{V}=\nabla_V \sigma_{II}-\nabla_{IV} \sigma_{II}=2\ip{\nabla_{KV}R(V,KV)IV}{V}\]
 and
 \[2\ip{\nabla_{KV}R(V,JV)IV}{V}=\nabla_V \sigma_{II}=-2\ip{\nabla_{JV}R(V,KV)IV}{V}\]
Inserting these into \eqref{eq:LaplaceStep} then results in
\begin{align*}
\ip{\Delta R(V,IV)IV}{V}&=-2\nu \ip{\nabla_{JV}R(V,JV)IV}{V}+2(\mu+\beta)\ip{\nabla_{KV}R(V,JV)IV}{V}\\
&+2\mu\ip{\nabla_{JV}R(V,KV)IV}{V}+2(\alpha+\nu)\ip{\nabla_{KV}R(V,KV)IV}{V} \\
&=(\alpha+\beta)\n_V \sigma_{II}(V)-\alpha \n_{IV}\sigma_{II}(V)-6\sigma_{II}^2
\end{align*}
For any point on $M$. 
\end{proof}
\begin{Rem}
We note how the right hand side of \eqref{eq:LaplacianonM} is expressed solely in quantities computable on $M$, even though the Laplacian on the left hand side involves derivatives normal to $M$.
\end{Rem}


\section{Appendix - Parameter Independence}
\label{App:Param}
In this appendix, we prove the parameter independence of the inequalities used in Kobayashi's estimates.

\begin{Prop}
\label{Prop:PoincareIndep}
Let $(X,g)$ denote a Kummer K3 with patchwork metric $g$ and K\"{a}hler form $\omega$. Then, for all $\vert a\vert$ small enough, there is a constant $C>0$ independent of $\vert a\vert$ such that
\begin{equation}
\int_X f^2 \omega^2 \leq C \int_X \vert df\vert^2_g \,\omega^2
\label{eq:Poincare}
\end{equation}
holds for all $f\in H^1(X,g)=\left\{f\in L^2(X,g)\, \big\vert \, \vert df\vert_g \in L^2(X,g)\right\}$ with 
\[\int_X f \, \omega^2=0.\]
\end{Prop}
\begin{proof}
This is the Poincar\'{e} inequality, with the optimal constant $\lambda^{-1}_1$, where $\lambda_1$ is the first eigenvalue of the scalar Laplacian $-\Delta$. So we need to deduce an $\vert a\vert$-independent lower bound on $\lambda_1$. We do this via Cheeger's isoperimetric constant $h(X)$.

\begin{Thm}[{\cite{Che70}}]
Let $(M,g)$ be a compact Riemannian manifold without boundary. Let $\lambda_1(M)$ denote the first non-zero eigenvalue of the Laplacian. Then
\begin{equation}
\lambda_1(M)\geq \frac{h(M)^2}{4}.
\label{eq:CheegerThm}
\end{equation}
\end{Thm} 

What Yau shows in \cite{Yau75} is the following (the first of the two Theorem 7's in his paper).
\begin{Thm}[{\cite{Yau75}[Theorem 7]}]
Let $(M,g)$ be an $m$-dimensional Riemannian manifold without boundary whose Ricci curvature is bounded from below by $-(m-1)K$ for some constant $K>0$. Let $\omega_m$ denote the (Euclidean) volume of the $(m-1)$-dimensional unit sphere. Then the following inequality holds.
\begin{equation}
h(M)^{-1}\leq \omega_m \text{diam}_g(M)\Vol_g(M)^{-1}\int_{0}^{\diam_g(M)}\left(\frac{\sinh(\sqrt{K}r)}{\sqrt{K}}\right)^{m-1}\, dr.
\label{eq:YauBound}
\end{equation}
\end{Thm}
Via Cheeger's theorem, Yau's bound gives a lower bound on $\lambda_1(M)$ using the diameter, volume, and a lower bound on the Ricci curvature of $M$. We therefore set about bounding these quantities.

That the diameter is bounded, can be seen as follows. The distance between suitably close pairs of points in the Euclidean region is independent of $a$. For pairs of point $p,q$ in the Eguchi-Hanson region, one can use the triangle inequality and compare with radial geodesics going from $p$ to the zero section and out again to $q$. The radial distance is bounded by the Euclidean distance, as one sees by computing the length of the curve $z(t)=tp$,
\[d_{EH}(0,p)^2=\int_0^1 g_{EH}(\dot{z},\dot{z})\, dt=\vert p\vert^4_{\C^2} \int_0^1 \frac{t^2}{\sqrt{a_i^2+t^4\vert p\vert^4_{\C^2}}}\, dt \leq \vert p\vert^2_{\C^2}=d_{Euc}(0,p)^2.\]
In the neck region, the distance is close to Euclidean by \eqref{eq:NeckPotential}. So the distance between any two points in $X$ can be bounded (using the triangle inequality several times) uniformly by $a$-independent quantities.

 The volume form is Euclidean ($\omega^2 =\eta \wedge \overline{\eta}$) outside of the neck regions where the gluing takes place. The Ricci curvature vanishes outside of the necks. So it only remains to bound the volume form and Ricci-curvature in the neck regions. However, on any compact subset of the complement of the zero section in Eguchi-Hanson space, $K\subset \mathcal{O}_{\C\P^1}(-2)\setminus \C\P^1$, it follows from the explicit form of the K\"{a}hler potential \eqref{eq:EHPot} that we can find a smooth function $\xi_K$ such that
\[f_{EH}=f_{Euc}+a^2 \xi_K.\]
Furthermore, the function $\xi_K$ is regular as $a\to 0$.
From this it follows that the potential for the patchwork metric in the neck regions reads
\[\Phi_a=f_{Euc}+\vert a\vert^2\chi \xi_K,\]
and the patchwork metric reads
\begin{equation}
g=g_{Euc} + \vert a\vert^2 h
\label{eq:Neckmetric}
\end{equation}
for some uniformly bounded $h$ with uniformly bounded derivatives. This then tells us
\begin{equation}
\det(g)=1+\vert a\vert^2 \tr(h)+\vert a\vert^4 \det(h)
\label{eq:NeckDetg}
\end{equation}
has the form $\det(g)=1+\vert a\vert^2 \cdot bounded$ and 
\[Ric=-\nabla^2 \ln \det(g)=\vert a\vert^2 \tilde{\xi}\]
for some uniformly bounded $\tilde{\xi}$. All in all, there is an $\vert a\vert-$ independent  constant $C>0$ such that 
\[(1-C\vert a\vert^2)\eta \wedge \overline{\eta} \leq \omega^2\leq (1+C\vert a\vert^2)\eta \wedge \overline{\eta}  \]
and
\[ \vert Ric\vert_g\leq C\vert a\vert^2\]
hold on all of $X$. With these, we get $\vert a\vert-$ independent bounds on the volume of $X$, and we may choose $K=-3C^2 \alpha^2$ as lower bound Ricci-curvature of $g$ for all $\vert a\vert\leq \alpha$. 
This shows that the upper bound \eqref{eq:YauBound} can be chosen independently of $\vert a\vert$ (as long as $\vert a\vert$ is bounded from above).

\end{proof}

Next we need to uniformly bound a Sobolev constant. This follows by a version due to Peter Li and some isoperimetric estimates due to Christopher B. Croke.
\begin{Prop}
\label{Prop:LiSob}
Let $(X,g)$ denote a Kummer K3 with patchwork metric $g$ and K\"{a}hler form $\omega$. Then, for all $\vert a\vert$ small enough, there is a constant $C>0$ independent of $\vert a\vert$ such that
\begin{equation}
\norm{df}^2_{L^2(X,g)} \geq C \left( \norm{f}^2_{L^4(X,g)}-\Vol_g(X)^{-1/2} \norm{f}^2_{L^2(X,g)}\right)
\label{eq:LiSob}
\end{equation}
holds for all $f\in H^1(X,g)$.
\end{Prop}
\begin{proof}
The first piece is Li's Sobolev inequality. 
\begin{Prop}[{\cite{Li80}[Lemma 2]}]
Let $(M,g)$ be a compact Riemannian of dimension $m\geq 3$. Let $C_0$ denote the optimal Sobolev constant in
\[C_0 \inf_{a\in \R} \norm{f-a}^m_{L^{\frac{m}{m-1}}(M,g)}\leq \norm{df}^m_{L^1(M,g)}.\]
for all functions $f\in W^{1,1}(M,g)$. Then there exists a constant $C=C(m)$ depending only on the dimension such that
\[\norm{df}^2_{L^2(M,g)}\geq C(m)C_0^{2/m}\left(\norm{f}^2_{L^{\frac{2m}{m-2}}(M,g)}-\Vol_g(M)^{-2/m} \norm{f}^2_{L^2(M,g)}\right).\] 
\end{Prop}

What remains to prove in the case of a Kummer K3 is that $C_0$ can be uniformly bounded. But the Sobolev constant $C_0$ is related to an isoperimetric constant $C_1$, namely the optimal constant in
\[C_1 \left(\min\{\Vol_g(M_1),\Vol_g(M_2)\}\right)^{m-1}\leq \Vol_g(N)^{m},\]
where the notation is as in Li's Sobolev inequality, $M=M_1\cup M_2$ and $N=\partial M_1=\partial M_2$. The volume on the right hand side is the $m-1$-dimensional volume on $n$ induced by $g$. The relation between the constants is
\[2C_1\geq C_0\geq C_1,\]
see \cite{GeometricAnalysis}[Theorem 9.6] for a proof. So we are left bounding $C_1$, which is done by Croke.
\begin{Thm}[{\cite{Croke}[Theorem 13]}]
Assume $(M,g)$ is a compact Riemannian manifold of dimension $m$. Let $D=\diam(M)$ and let $K\in \R$ be such that $Ric\geq -(m-1)Kg$. Then there is a constant $C=C(m)$ depending only on the dimension such that
\[C_1\geq C(m)\left( \frac{\Vol_g(M)}{\int_0^D (\sqrt{1/K}\sinh(r\sqrt{K}))^{m-1}\, dr}\right)^{m+1}.\]
\end{Thm} 
The quantities involved in this lower bound are the same as in Yau's estimate \eqref{eq:YauBound}. Hence we may uniformly bound $C_1$ from below for all finite values of $\vert a\vert$.
\end{proof}

\begin{Rem}
Looking of the above proofs, one could formulate similar uniform Poincar\'{e} and Sobolev estimates for any compact manifold $M$ with a family of metric $g_{\lambda}$ as long as there are $\lambda$-independent 
\begin{itemize}
\item upper and lower bounds on the volume,
\[C\leq \Vol_{g_{\lambda}}(X)\leq C^{-1},\]
\item  upper and lower bounds on the diameter
\[C\leq \diam_{g_{\lambda}}(X)\leq C^{-1},\]
\item and a lower bound on the the Ricci-curvature
\[-C g_{\lambda}\leq Ric_{g_{\lambda}}.\]
\end{itemize}
\end{Rem}

\section{Appendix  - Yau Estimates}
\label{App:Yau}
In this section we give a slightly easier proof of \cite[Equation 2.22]{Yau78} in complex dimension 2.

Let us start by computing a couple of derivatives of the Monge-Amp\`{e}re equation.
Recall that for a matrix $B$ depending differentiably on a parameter $t$, the Jacobi formula says
\begin{equation}
\frac{d}{dt}\det(B(t))=\tr\left(\adj(B)\frac{dB(t)}{dt}\right).
\label{eq:JacobiFormula}
\end{equation}

The Monge-Amp\`{e}re equation \eqref{eq:MA} can, in local coordinates, be written as
\begin{equation}
\det(\tilde{g})=\det\left(g+ \nabla^2  \phi\right)=Ae^\psi \det(g),
\label{eq:MongeRecap}
\end{equation}
with $\psi$ and $A$ given by \eqref{eq:psidef} and \eqref{eq:ADef} respectively.
 
Let $t$ be some coordinate, $t=z^\mu$ or $t=\bar{z}^\mu$. By \eqref{eq:JacobiFormula} and the fact that both $g$ and $\tilde{g}$ are invertible, we find as derivatives of \eqref{eq:MongeRecap}
\begin{equation}
\det(\tilde{g})\tr\left(\tilde{g}^{-1}\frac{\partial\tilde{g}}{\partial t}\right)=Ae^\psi \det(g)\left(\frac{\partial \psi}{\partial t}+\tr\left(g^{-1}\frac{\partial g}{\partial t}\right)\right).
\label{eq:MongeFirstDerStep0}
\end{equation} 
Using \eqref{eq:MongeRecap}, we can simplify \eqref{eq:MongeFirstDerStep0}, resulting in

\begin{equation}
\tr\left(\tilde{g}^{-1}\frac{\partial\tilde{g}}{\partial t}\right)=\frac{\partial \psi}{\partial t}+\tr\left(g^{-1}\frac{\partial g}{\partial t}\right).
\label{eq:PsiderNoCor}
\end{equation} 
Let $s$ be another coordinate, $s=z^\nu$ or $s=\bar{z}^{\nu}$. Then we can compute the second derivative using that the trace satisfies Leibniz' rule, along with the formula $\frac{dA^{-1}}{ds}=-A^{-1}\frac{dA}{ds}A^{-1}$ for any invertible matrix $A$. The result is as follows.

\begin{align}
&-\tr\left(\tilde{g}^{-1}\frac{\partial \tilde{g}}{\partial s}\tilde{g}^{-1}\frac{\partial \tilde{g}}{\partial t}\right)+\tr\left(\tilde{g}^{-1}\frac{\partial^2\tilde{g}}{\partial s\partial t}\right)\notag\\&=\frac{\partial^2 \psi}{\partial t\partial s}-\tr\left(g^{-1}\frac{\partial g}{\partial s}g^{-1}\frac{\partial g}{\partial t}\right)+\tr\left(g^{-1}\frac{\partial^2 g}{\partial s\partial t}\right).
\label{eq:PsidderNoCor}
\end{align}

Choosing $t=z^\alpha$ and $s=\bar{z}^\beta$, both equations \eqref{eq:PsiderNoCor} and \eqref{eq:PsidderNoCor} can be written out using indices. The results look as follows.

\begin{equation}
\frac{\partial \psi}{\partial z^\alpha}=\tilde{g}^{\bar{\nu}\mu}\left(\frac{\partial g_{\mu\bar{\nu}}}{\partial z^\alpha}+\frac{\partial^3 \phi}{\partial z^\mu \partial \bar{z}^\nu \partial z^\alpha}\right)-g^{\bar{\nu}\mu}\frac{\partial g_{\mu\bar{\nu}}}{\partial z^\alpha},
\label{eq:PsiderCor}
\end{equation}
and
\begin{align}
\frac{\partial^2 \psi}{\partial z^\alpha\partial \bar{z}^\beta}&= -\tilde{g}^{\bar{\sigma}\rho}\tilde{g}^{\bar{\nu}\mu}\left(\frac{\partial g_{\rho\bar{\nu}}}{\partial z^\alpha}+\frac{\partial^3\phi}{\partial z^\alpha \partial z^\rho \partial\bar{z}^\nu}\right)\left(\frac{\partial g_{\mu\bar{\sigma}}}{\partial \bar{z}^\beta}+\frac{\partial^3 \phi}{\partial \bar{z}^\beta \partial z^\mu \partial \bar{z}^\sigma}\right)\notag \\
&+\tilde{g}^{\bar{\nu}\mu}\left(\frac{\partial^2 g_{\mu\bar{\nu}}}{\partial z^\alpha \partial\bar{z}^\beta}+\frac{\partial^4\phi}{\partial z^\alpha\partial\bar{z}^\beta\partial z^\mu \partial \bar{z}^\nu}\right) \notag \\
&+g^{\bar{\sigma}\rho}g^{\bar{\nu}\mu}\frac{\partial g_{\rho\bar{\nu}}}{\partial z^\alpha}\frac{\partial g_{\mu\bar{\sigma}}}{\partial \bar{z}^\beta}-g^{\bar{\nu}\mu}\frac{\partial^2 g_{\mu\bar{\nu}}}{\partial z^\alpha \partial \bar{z}^\beta}. 
\label{eq:PsidderCor}
\end{align}
Equations \eqref{eq:PsiderCor} and \eqref{eq:PsidderCor}  are \cite[Eq. 2.4]{Yau78} and \cite[Eq. 2.5]{Yau78} respectively. The last line in equation \eqref{eq:PsidderCor} can be identified as $R_{\alpha\bar{\beta}}$. The term $\frac{\partial^2 g_{\mu\bar{\nu}}}{\partial z^\alpha \partial\bar{z}^\beta}$ can be written $-R_{\mu\bar{\nu}\alpha\bar{\beta}}+g^{\bar{\sigma}\rho}\frac{\partial g_{\mu\bar{\sigma}}}{\partial z^\alpha} \frac{\partial g_{\rho\bar{\nu}}}{\partial\bar{z}^\beta}$. Inserting these two observations into \eqref{eq:PsidderCor}, we may write
\begin{align}
\frac{\partial^2 \psi}{\partial z^\alpha\partial \bar{z}^\beta}&= -\tilde{g}^{\bar{\sigma}\rho}\tilde{g}^{\bar{\nu}\mu}\left(\frac{\partial g_{\rho\bar{\nu}}}{\partial z^\alpha}+\frac{\partial^3\phi}{\partial z^\alpha \partial z^\rho \partial\bar{z}^\nu}\right)\left(\frac{\partial g_{\mu\bar{\sigma}}}{\partial \bar{z}^\beta}+\frac{\partial^3 \phi}{\partial \bar{z}^\beta \partial z^\mu \partial \bar{z}^\sigma}\right)\notag \\
&+\tilde{g}^{\bar{\nu}\mu}\left(-R_{\mu\bar{\nu}\alpha\bar{\beta}}+g^{\bar{\sigma}\rho}\frac{\partial g_{\mu\bar{\sigma}}}{\partial z^{\al}}\frac{\partial g_{\rho\bar{\nu}}}{\partial \overline{z}^\beta}+\frac{\partial^4\phi}{\partial z^\alpha\partial\bar{z}^\beta\partial z^\mu \partial \bar{z}^\nu}\right) +R_{\alpha\bar{\beta}}. 
\label{eq:PsidderCor2}
\end{align}

From just the definition of $\Delta f=g^{\on \mu} \partial_\mu \partial_{\on} f$ and $\tilde{\Delta}=\tilde{g}^{\on \mu} \partial_\mu \partial_{\on} f$ one easily computes the following.
\begin{align}
\tilde{\Delta}(\Delta \phi)&=\tilde{g}^{\bar{\nu}\mu}\frac{\partial^2}{\partial z^\mu\partial\bar{z}^\nu}\left(g^{\bar{\beta}\alpha}\frac{\partial^2\phi}{\partial z^\alpha\partial \bar{z}^\beta}\right)\notag \\
&=\tilde{g}^{\bar{\nu}\mu} g^{\bar{\beta}\alpha}\frac{\partial^4 \phi}{\partial z^\alpha\partial \bar{z}^\beta \partial z^\mu\partial \bar{z}^\nu}+\tilde{g}^{\mu\bar{\nu}}\frac{\partial^2 g^{\bar{\beta}\alpha}}{\partial z^\mu \partial \bar{z}^\nu}\frac{\partial^2\phi}{\partial z^\alpha\partial \bar{z}^\beta}\notag\\
&+\tilde{g}^{\bar{\nu}\mu}\frac{\partial g^{\bar{\beta}\alpha}}{\partial z^\mu}\frac{\partial^3 \phi}{\partial z^\alpha \partial\bar{z}^\beta \partial \bar{z}^\nu}+\tilde{g}^{\bar{\nu}\mu}\frac{\partial g^{\bar{\beta}\alpha}}{\partial \bar{z}^\nu}\frac{\partial^3 \phi}{\partial z^\alpha \partial \bar{z}^\beta \partial z^\mu}.
\label{eq:PhiDer}
\end{align}
Let us introduce the same abbreviations as in \cite{Yau78}, namely
\begin{equation}
\phi_{\alpha\bar{\beta}}\coloneqq \frac{\partial^2 \phi}{\partial z^\alpha \partial \bar{z}^\beta},
\notag
\end{equation}
\begin{equation}
\phi_{\alpha\bar{\beta}\mu}\coloneqq \frac{\partial^3 \phi}{\partial z^\alpha \partial \bar{z}^\beta\partial z^\mu},
\notag
\end{equation}
and similarly for other indices. We are ready for our first lemma.
\begin{Lem}
The following holds at any point $p\in X$.
\begin{align}
\tilde{\Delta}(\Delta \phi)=\Delta \psi-R+\tilde{g}^{\bar{\nu}\mu}\left(R_{\mu\bar{\nu}}+R_{\mu\bar{\nu}}{}^{\bar{\beta}\alpha}\phi_{\alpha\bar{\beta}}\right)+\tilde{g}^{\bar{\nu}\mu}\tilde{g}^{\bar{\sigma}\rho}g^{\bar{\beta}\alpha}\phi_{\alpha\rho\bar{\nu}}\phi_{\bar{\beta}\mu\bar{\sigma}}.
\label{eq:PhiDerLemma}
\end{align}
Here $R_{\mu\bar{\nu}}\coloneqq g^{\bar{\beta}\alpha}R_{\mu\bar{\nu}\alpha\bar{\beta}}$ and $R\coloneqq g^{\bar{\nu}\mu}R_{\mu\bar{\nu}}$ are the Ricci and scalar curvatures respectively.
\end{Lem}
\begin{proof}
Choose holomorphic normal coordinates at $p$, meaning $g_{\mu\bar{\nu}}(p)=\delta_{\mu\bar{\nu}}$ and $g_{\mu\bar{\nu},\alpha}(p)=0$. At this point in the chosen coordinates, one sees from \eqref{eq:PhiDer} that
\begin{equation}
\tilde{\Delta}(\Delta\phi)(p)=\tilde{g}^{\bar{\nu}\mu}g^{\bar{\beta}\alpha}\phi_{\alpha\bar{\beta}\mu\bar{\nu}}+R^{\bar{\beta}\alpha}{}_{\mu\bar{\nu}}\tilde{g}^{\bar{\nu}\mu}\phi_{\alpha\bar{\beta}}.
\label{eq:PhiDerAtP}
\end{equation}

Furthermore, one sees from \eqref{eq:PsidderCor2} that, at the point $p$ in special coordinates, we have
\begin{equation}
\Delta\psi(p)=-\tilde{g}^{\bar{\sigma}\rho}\tilde{g}^{\bar{\nu}\mu}g^{\bar{\beta}\alpha}\phi_{\alpha\rho\bar{\nu}}\phi_{\bar{\beta}\mu\bar{\sigma}}+\tilde{g}^{\bar{\nu}\mu}(-R_{\mu\bar{\nu}}+g^{\bar{\beta}\alpha}\phi_{\alpha\bar{\beta}\mu\bar{\nu}})+R.
\label{eq:PsiDerAtP}
\end{equation}
Combining equations \eqref{eq:PhiDerAtP} and \eqref{eq:PsiDerAtP} yields, at the point $p$,
\begin{align}
\tilde{\Delta}(\Delta \phi)(p)=\Delta \psi-R+\tilde{g}^{\bar{\nu}\mu}\left(R_{\mu\bar{\nu}}+R_{\mu\bar{\nu}}{}^{\bar{\beta}\alpha}\phi_{\alpha\bar{\beta}}\right)+\tilde{g}^{\bar{\nu}\mu}\tilde{g}^{\bar{\sigma}\rho}g^{\bar{\beta}\alpha}\phi_{\alpha\rho\bar{\nu}}\phi_{\bar{\beta}\mu\bar{\sigma}}.
\notag
\end{align}

This shows that the claimed identity, \eqref{eq:PhiDerLemma}, holds at a point with a particular choice of coordinates. Since this is a scalar equation (the equation consists of tensors\footnote{There is a small subtlety here. The expressions $\phi_{\alpha\rho\bar{\nu}}$ and $\phi_{\bar{\beta}\mu\bar{\sigma}}$ are a priori \textit{not} tensors, but in our chosen coordinates, we have $\phi_{\alpha\rho\bar{\nu}}(p)=\nabla_{\alpha} \phi_{\rho\bar{\nu}}(p)$ and $\phi_{\bar{\beta}\mu\bar{\sigma}}(p)=\nabla_{\bar{\beta}} \phi_{\mu\bar{\sigma}}(p)$, which are clearly tensors. Here, the covariant derivative is with respect to the metric $g$.}   contracted by the metric on the right hand side, and a concatenation of Laplacians on the left hand side), it has to hold everywhere independent of coordinates.
\end{proof}

\begin{Lem}
Choose holomorphic normal coordinates at a point $p$ and diagonalize $\nabla^2 \phi$ in this point. Then we may write, at the single point $p$,
\begin{equation}
\tilde{\Delta}(\Delta\phi)(p)=\Delta \psi+R_{1\bar{1}2\bar{2}}\left(\frac{\Tr_g(\tilde{g})^2}{\det(\tilde{g})}-4\right)+\sum_{\mu\nu\alpha}(1+\phi_{\mu\bar{\mu}})^{-1}(1+\phi_{\nu\bar{\nu}})^{-1}\vert \phi_{\alpha\bar{\mu}\nu}\vert^2.
\label{eq:PhiDerExpr}
\end{equation}
\end{Lem}
\begin{proof}
This is mostly a matter of inserting the coordinate choice into equation \eqref{eq:PhiDerLemma}, where one needs to observe a couple of things. With our special coordinates, we have
\begin{equation}
R_{\alpha\bar{\beta}}(p)=\sum_{\mu=1}^2 R_{\mu\bar{\mu}\alpha\bar{\beta}}(p)
\notag
\end{equation}
and
\begin{equation}
R(p)=\sum_{\mu,\nu=1}^2 R_{\mu\bar{\mu}\nu\bar{\nu}}(p).
\notag
\end{equation}
Inserting these two equations into \eqref{eq:PhiDerLemma} results in (at the point $p$)
\begin{align}-R+\tilde{g}^{\bar{\nu}\mu}(R_{\mu\bar{\nu}}+R_{\mu\bar{\nu}}{}^{\bar{\beta}\alpha}\phi_{\alpha\bar{\beta}})&=\sum_{\mu,\nu=1}^2 \left(-1+\frac{1}{1+\phi_{\mu\bar{\mu}}}(1+\phi_{\nu\bar{\nu}})\right)R_{\mu\bar{\mu}\nu\bar{\nu}}\notag \\
&=\sum_{\mu,\nu=1}^2 \left(\frac{\phi_{\nu\bar{\nu}}-\phi_{\mu\bar{\mu}}}{1+\phi_{\mu\bar{\mu}}}\right)R_{\mu\bar{\mu}\nu\bar{\nu}}.
\label{eq:RicciandScalaratpoint}
\end{align}
The sum on the right hand side of \eqref{eq:RicciandScalaratpoint} becomes of course a sum over $\nu\neq \mu$, i.e. a sum with two terms. Note that the denominators in \eqref{eq:RicciandScalaratpoint} are never $0$, seeing how $\tilde{g}_{\mu\bar{\nu}}(p)=\delta_{\mu\bar{\nu}}+\phi_{\mu\bar{\mu}}(p)\delta_{\mu\bar{\nu}}$ is supposed to be a metric. Since $R_{1\bar{1}2\bar{2}}=R_{2\bar{2}1\bar{1}}$, we may write \eqref{eq:RicciandScalaratpoint} as

\begin{align}
-R+\tilde{g}^{\bar{\nu}\mu}(R_{\mu\bar{\nu}}+R_{\mu\bar{\nu}}{}^{\bar{\beta}\alpha}\phi_{\alpha\bar{\beta}})&=R_{1\bar{1}2\bar{2}}\sum_{\mu\neq \nu}\frac{\phi_{\nu\bar{\nu}}-\phi_{\mu\bar{\mu}}}{1+\phi_{\mu\bar{\mu}}}\notag \\&=R_{1\bar{1}2\bar{2}}\sum_{\mu\neq \nu}\frac{1+\phi_{\nu\bar{\nu}}-(1+\phi_{\mu\bar{\mu}})}{1+\phi_{\mu\bar{\mu}}}\notag \\&=R_{1\bar{1}2\bar{2}}\left(\sum_{\mu\neq \nu}\left(1+\frac{1+\phi_{\nu\bar{\nu}}}{1+\phi_{\mu\bar{\mu}}}\right)-4\right)\notag \\ &= R_{1\bar{1}2\bar{2}}\left(\frac{(2+\phi_{1\bar{1}}+\phi_{2\bar{2}})^2}{(1+\phi_{1\bar{1}})(1+\phi_{2\bar{2}})}-4\right)\notag \\&=R_{1\bar{1}2\bar{2}}\left(\frac{\Tr_g(\tilde{g})^2}{\det(\tilde{g})}-4\right).
\notag
\end{align}

A last observation one needs to make in deriving \eqref{eq:PhiDerExpr} is that $\phi_{\alpha\mu\bar{\nu}}=\overline{\phi_{\bar{\alpha}\bar{\mu}\nu}}$ since $\phi$ is real.
\end{proof}

The stage has been set, and we proceed with the estimates of the expression
 $\tilde{\Delta}((\exp(-C\phi)\Tr_g(\tilde{g}))$ for some constant $C>0$. Here is the first step.
\begin{Lem}
Let $C$ be any positive real constant. Then we have, at an arbitrary point $p\in X$, that
\begin{align}
&\exp(C\phi)\tilde{\Delta}\Big(\exp(-C\phi)\text{Tr}_g(\tilde{g})\Big)\notag \\
&=\Tr_g(\tilde{g})\left\vert C\partial \phi-\frac{\partial \Tr_g(\tilde{g})}{\Tr_g(\tilde{g})}\right\vert_{\tilde{g}}^2-\Tr_g(\tilde{g})^{-1} \vert \partial \Tr_g(\tilde{g})\vert_{\tilde{g}}^2+\tilde{\Delta}(\Delta \phi)-C(\tilde{\Delta}\phi)\Tr_g(\tilde{g})
\notag
\end{align}
and thus
\begin{equation}
\exp(C\phi)\tilde{\Delta}\Big(\exp(-C\phi)\Tr_g(\tilde{g})\Big)\geq -\Tr_g(\tilde{g})^{-1} \vert \partial \Tr_g(\tilde{g})\vert_{\tilde{g}}^2 +\tilde{\Delta}(\Delta(\phi))-C\Tr_g(\tilde{g})(\tilde{\Delta}\phi).
\label{eq:Yau212}
\end{equation}
\end{Lem}


\begin{proof}
Using that $\tilde{\Delta} (F\cdot G)=F\tilde{\Delta} G+G\tilde{\Delta} F+2\text{Re}\ip{\partial F}{\partial G}_{\tilde{g}}$ holds for any functions $F,G$, we find 

\begin{align}
&\tilde{\Delta}\Big(\exp(-C\phi)\Tr_g(\tilde{g})\Big)\notag \\
&=\exp(-C\phi)\left(C^2\vert\partial \phi\vert_g^2\Tr_g(\tilde{g})-2C\text{Re}\ip{\partial \phi}{\partial \Tr_g(\tilde{g})}_{\tilde{g}}-C(\tilde{\Delta}\phi)\Tr_g(\tilde{g})+\tilde{\Delta}(\Delta \phi)\right).
\label{eq:Trick2}
\end{align}
The first two terms on the right hand side of \eqref{eq:Trick2} can be written differently, using that for any pair of vectors $a,b$, 
\begin{equation}
\vert a\vert^2+2\text{Re}(\ip{a}{b})=\vert a+b\vert^2-\vert b\vert^2
\notag
\end{equation}
holds for any Hermitian inner product. In our case, if we set $a=C\partial \phi$ and $b=-\frac{\partial \Tr_g(\tilde{g})}{\Tr_g(\tilde{g})}$,
we arrive at
\begin{align}&\exp(-C\phi)\left(C^2\vert\partial \phi\vert_g^2\Tr_g(\tilde{g})-2C\text{Re}\ip{\partial \phi}{\partial \Tr_g(\tilde{g})}_{\tilde{g}}\right)\notag \\&=\exp(-C\phi)\Tr_g(\tilde{g})\left( \left\vert C\partial \phi-\frac{\partial \Tr_g(\tilde{g})}{\Tr_g(\tilde{g})}\right\vert_{\tilde{g}}^2-\frac{\vert \partial \Tr_g(\tilde{g})\vert_{\tilde{g}}^2}{\Tr_g(\tilde{g})^2}\right)\notag\\&=\exp(-C\phi)\Tr_g(\tilde{g})\left\vert C\partial \phi-\frac{\partial \Tr_g(\tilde{g})}{\Tr_g(\tilde{g})}\right\vert_{\tilde{g}}^2-\exp(-C\phi)\Tr_g(\tilde{g})^{-1} \vert \partial \Tr_g(\tilde{g})\vert_{\tilde{g}}^2.
\label{eq:Trick3}
\end{align}
The first term on the right hand side of \eqref{eq:Trick3} is clearly non-negative, hence
\begin{align}
&\exp(-C\phi)\left(C^2\vert\partial \phi\vert_g^2\Tr_g(\tilde{g})-2C\text{Re}\ip{\partial \phi}{\partial \Tr_g(\tilde{g})}_{\tilde{g}}\right)\notag \\ &\geq -\exp(-C\phi)\Tr_g(\tilde{g})^{-1} \vert \partial \Tr_g(\tilde{g})\vert_{\tilde{g}}^2.
\label{eq:Trick05}
\end{align}
Inserting \eqref{eq:Trick05} into \eqref{eq:Trick2} allows us to deduce \eqref{eq:Yau212}.

\end{proof}
For those keeping track, \eqref{eq:Yau212} is precisely \cite[Eq. 2.13]{Yau78}, albeit with slightly different notation.

\begin{Lem}
\label{Lem:Yau217}
Choose holomorphic normal coordinates at a point $p$ and diagonalize $\nabla^2 \phi$ in this point. Then the following inequality holds at $p$.
\begin{align}
\exp(+C\phi)\tilde{\Delta}\Big(\exp(-C\phi)\Tr_g(\tilde{g})\Big)(p)&\geq \Delta \psi+R_{1\bar{1}2\bar{2}}\left(\frac{\Tr_g(\tilde{g})^2}{\det(\tilde{g})}-4\right)\notag \\& -C \Tr_g(\tilde{g})\tilde{\Delta}(\phi).
\label{eq:Yau217}
\end{align}
\end{Lem}
\begin{proof}
We will use \eqref{eq:Yau212} along with our expression for $\tilde{\Delta} (\Delta \phi)(p)$ given as \eqref{eq:PhiDerExpr}. First, we need to have a closer look at a term from  \eqref{eq:Yau212}. The first claim we need to establish is that
\begin{equation}
\vert \partial \Tr_g(\tilde{g})\vert_{\tilde{g}}^2(p)=\sum_\mu \frac{1}{1+\phi_{\mu\bar{\mu}}}\left\vert \sum_\nu \phi_{\nu\bar{\nu}\mu}\right\vert^2. 
\label{eq:Claim1}
\end{equation}
This will be seen to be a consequence of the coordinate choice. Using that $g^{\bar{\beta}\alpha}(p)=\delta^{\bar{\beta}\alpha}$, $g^{\bar{\beta}\alpha}{}_{,\mu}(p)=0$, and $\tilde{g}^{\bar{\nu}\mu}(p)=\delta^{\bar{\nu}\mu}(1+\phi_{\mu\bar{\mu}})^{-1}$, we find that
\begin{equation}
\partial_\mu (\Tr_{g}(\tilde{g}))(p)=\partial_\mu (2+\Delta\phi)=\partial_\mu( g^{\bar{\beta} \alpha} \phi_{\alpha\bar{\beta}})=\sum_{\alpha} \phi_{\alpha\bar{\alpha}\mu}.
\notag
\end{equation}
From this it follows that
\begin{align}
\vert \partial \Tr_g(\tilde{g})\vert_{\tilde{g}}^2(p)&=\tilde{g}^{\bar{\nu}\mu}\partial_\mu \Tr_g(\tilde{g}) \partial_{\bar{\nu}}\Tr_g(\tilde{g})\notag \\ &=\sum_\mu \frac{1}{1+\phi_{\mu\bar{\mu}}}\left(\sum_\alpha \phi_{\alpha\bar{\alpha}\mu}\right)\left(\sum_\beta \phi_{\beta\bar{\beta}\bar{\mu}}\right)\notag \\ &=\sum_\mu \frac{1}{1+\phi_{\mu\bar{\mu}}}\left\vert \sum_\nu \phi_{\nu\bar{\nu}\mu}\right\vert^2,
\notag
\end{align}
which is precisely \eqref{eq:Claim1}.
The second claim is that
\begin{equation}
\Tr_g(\tilde{g})^{-1}\vert \partial \Tr_g(\tilde{g})\vert^2_{\tilde{g}}(p)\leq \sum_{\alpha\mu\nu} \frac{1}{1+\phi_{\mu\bar{\mu}}}\frac{1}{1+\phi_{\nu\bar{\nu}}}\vert \phi_{\mu\bar{\nu}\alpha}\vert^2.
\label{eq:Claim2}
\end{equation}
To see this, observe that we have the following chain of (elementary) inequalities.
\begin{align}
\left \vert \sum_\nu \phi_{\nu\bar{\nu}\mu}\right\vert^2 &\leq \sum_{\nu}\vert \phi_{\nu\bar{\nu}\mu}\vert^2\notag \\
&= \sum_{\nu} \frac{1}{1+\phi_{\nu\bar{\nu}}} (1+\phi_{\nu\bar{\nu}})\vert\phi_{\nu\bar{\nu}\mu}\vert^2\notag \\
&\leq \sum_{\nu} \left(\frac{1}{1+\phi_{\nu\bar{\nu}}} \vert \phi_{\nu\bar{\nu}\mu}\vert^2\right)\sum_{\beta}(1+\phi_{\beta\bar{\beta}}).
\label{eq:Phider3}
\end{align}
One recognizes the last factor on the right hand side of \eqref{eq:Phider3} as $\sum_{\beta} (1+\phi_{\beta\bar{\beta}})=2+\Delta \phi=\Tr_{g}(\tilde{g})$ in our special coordinates. Inserting \eqref{eq:Phider3} into \eqref{eq:Claim1} and dividing by $\Tr_g(\tilde{g})$,  we establish
\begin{align}
\Tr_g(\tilde{g})^{-1}\vert \partial \Tr_g(\tilde{g})\vert^2_{\tilde{g}}(p)&\leq \sum_{\mu\nu} \frac{1}{1+\phi_{\mu\bar{\mu}}}\frac{1}{1+\phi_{\nu\bar{\nu}}}\vert \phi_{\mu\bar{\nu}\nu}\vert^2\notag \\ &\leq\sum_{\alpha\mu\nu} \frac{1}{1+\phi_{\mu\bar{\mu}}}\frac{1}{1+\phi_{\nu\bar{\nu}}}\vert \phi_{\mu\bar{\nu}\alpha}\vert^2.
\notag
\end{align}
This is what we claimed as \eqref{eq:Claim2}.
Next, insert \eqref{eq:Claim2} into \eqref{eq:PhiDerExpr}  to conclude
\begin{align}
\tilde{\Delta}(\Delta\phi)(p)-\Tr_g(\tilde{g})^{-1}\vert \partial \Tr_g(\tilde{g})\vert^2_{\tilde{g}}(p)&\geq \Delta \psi +\left(\frac{\Tr_g(\tilde{g})^2}{\det(\tilde{g})}-4\right)R_{1\bar{1}2\bar{2}}.
\notag
\end{align}
Inserting this back into \eqref{eq:Yau212} gives us \eqref{eq:Yau217}.
\end{proof}
We have almost got the inequality we need. We only need to make one more observation.
\begin{Lem}
\label{Lem:CorrectedLapLemma}
In holomorphic normal coordinates at a point $p$ and diagonalized $\nabla^2 \phi$ in  this point, we can write
\begin{equation}
\tilde{\Delta}(\phi)(p)=2-\frac{\Tr_g(\tilde{g})}{\det(\tilde{g})}(p).
\label{eq:CorrectedLapLemma}
\end{equation}
\end{Lem}
\begin{proof}
This is a short computation:
\begin{align}
\tilde{\Delta}(\phi)(p)&=\sum_{\mu} \frac{\phi_{\mu\bar{\mu}}}{1+\phi_{\mu\bar{\mu}}}=2-\sum_{\mu}\frac{1}{1+\phi_{\mu\bar{\mu}}}\notag \\&=2-\frac{2+\phi_{1\bar{1}}+\phi_{2\bar{2}}}{(1+\phi_{1\bar{1}})(1+\phi_{2\bar{2}})}=2-\frac{\Tr_g(\tilde{g})}{\det(\tilde{g})}.
\notag
\end{align}
\end{proof}
Armed with Lemmas \ref{Lem:Yau217} and \ref{Lem:CorrectedLapLemma}, we are able to prove what is a special case of \cite[Eq. 2.22]{Yau78}. The statement is as follows.
\begin{Prop}
\label{Prop:Yau222Appendix}
Choose holomorphic normal coordinates at a point $p$ and diagonalize $\nabla^2 \phi$ in this point. Then the following inequality holds for any positive real number $C$.
\begin{align}
&\exp(C\phi)\tilde{\Delta}\Big(\exp(-C\phi)\Tr_g(\tilde{g})\Big)(p)\notag \\&\geq \Delta \psi-4R_{1\bar{1}2\bar{2}}  -2C\Tr_g(\tilde{g})+\left(C+R_{1\bar{1}2\bar{2}}\right)    \frac{\Tr_g(\tilde{g})^2}{\det(\tilde{g})}.
\label{eq:Yau222Appendix}
\end{align}
\end{Prop}
\begin{proof}
Insert \eqref{eq:CorrectedLapLemma} into \eqref{eq:Yau217} and rearrange slightly.
\end{proof}

\bibliographystyle{halpha}
\bibliography{PhD}

\end{document}